\documentclass[10pt, letterpaper]{article}
\input{preambule}
\usepackage[T1]{fontenc}
\usepackage[margin=0.95in]{geometry}
\usepackage[utf8]{inputenc}
\usepackage{fix-cm}
\usepackage{microtype}

\usepackage{amsmath,amssymb,amsfonts,mathtools}
\usepackage{mathrsfs}
\usepackage{bm}
\usepackage{bbm}
\usepackage{dsfont}
\usepackage{stmaryrd}
\usepackage{euscript}

\mathtoolsset{showonlyrefs}
\numberwithin{equation}{section}

\usepackage{graphicx}
\usepackage{subfigure} 
\usepackage{booktabs}
\usepackage{array}
\usepackage{algorithm}
\usepackage{algorithmicx}
\usepackage{algpseudocode}
\usepackage[algo2e,ruled,vlined]{algorithm2e}

\usepackage{enumerate}
\usepackage{comment}
\usepackage{xcolor}
\usepackage{xcolor,soul}
\usepackage[normalem]{ulem}

\usepackage{tikz}
\usetikzlibrary{fit,matrix,chains,positioning,decorations.pathreplacing,arrows}

\usepackage{natbib}
\usepackage{titlesec}
\titleformat{\section}{\large\bfseries}{\thesection.}{0.5em}{}
\titleformat{\subsection}{\normalsize\bfseries}{\thesubsection.}{0.5em}{}

\title{\LARGE Policy Gradient Learning for Distributionally Robust Markov Decision Processes under Wasserstein Ambiguity\vspace{0.35cm}}

\author{
Yadh Hafsi\thanks{École Polytechnique, CMAP, \textsf{yadh.hafsi at polytechnique.edu}. Supported by the Chaire Risques Financiers, Société Générale, and the Institut Europlace de Finance.}
\and
Samy Mekkaoui\thanks{École Polytechnique, CMAP, \textsf{samy.mekkaoui at polytechnique.edu}. Supported by the S-G Chair ``Risques Financiers'' and the ``Deep Finance and Statistics'' Qube-RT Chair.}
\and
Huy\^en Pham\thanks{École Polytechnique, CMAP, \textsf{huyen.pham at polytechnique.edu}. Supported by the BNP-PAR Chair ``Futures of Quantitative Finance'', the Chair ``Risques Financiers'', FiME, and the ``Finance and Sustainable Development'' EDF--CACIB Chair.}
\and
Kaixin Yan\thanks{School of Mathematical Sciences, Xiamen University. Supported by the China Scholarship Council. This research was completed during her visit to CMAP, École Polytechnique, from November 2024 to April 2026.}
}

\begin{document}
\sloppy

\maketitle


\begin{abstract}
We study finite-horizon Markov decision processes under distributional uncertainty in the transition kernels and develop a policy-gradient framework for Wasserstein distributionally robust control. Ambiguity is modeled by Wasserstein balls of common radius centered at state--action-dependent nominal transition kernels, leading to a max--min problem over randomized policies and admissible transition laws. Because the worst-case transition law depends implicitly on the policy parameters, the standard policy-gradient argument does not apply directly. We address this difficulty by combining the dynamic programming recursion with Wasserstein duality and a primal envelope argument. In general, the right and left directional derivatives of the one-step worst-case value are obtained by taking the minimum or maximum expected downstream value derivative over the set of worst-case transition laws. In finite state--action spaces, this set is characterized through the optimal face of a transport linear program, yielding an exact directional-derivative recursion. Under the required stability conditions and uniqueness of the active dual and transport optimizers, the derivative becomes linear in the policy perturbation and admits an explicit vector-valued policy-gradient recursion. Building on this representation, we propose a robust actor--critic implementation and evaluate it on benchmark examples.
\end{abstract}

\vspace{5mm}


\vspace{5mm}



\section{Introduction}

In many reinforcement learning and control problems, the transition law of the system is not known exactly. Even when a nominal model is available, the system may later be tested or deployed in an environment that differs from the one used for calibration. Different data sources, market conditions, or operating regimes may induce different probability measures on trajectories. As a result, a policy that performs well under one reference model may perform poorly when the underlying dynamics change. This motivates robust reinforcement learning, whose goal is to construct policies that perform well under a family of plausible transition laws rather than under a single nominal model.

This issue appears naturally in many applications. In inventory control and supply chain management, demand distributions are difficult to estimate precisely and may vary over time because of seasonality, structural change, or external shocks. Robust and distributionally robust methods have therefore been used to hedge against ambiguity in demand models; see for example \cite{BertsimasThiele2006,Scarf1958,GallegoMoon1993,KimChung2024}. In portfolio optimization,
market dynamics shift across regimes, with changes in volatility, correlations, and return
distributions that may occur at different time scales. A strategy calibrated to one regime may break
down after a regime shift. This has led to distributionally robust formulations under Wasserstein ambiguity \cite{BlanchetChenZhou2022}, as well as robust continuous-time control under model uncertainty and adaptive learning \cite{CarteaBhudisaksangSanchezBetancourt2025}. In all these settings, uncertainty acts through the law of the controlled state process. In our framework, this uncertainty is represented by set-valued transition kernels $\Pc_t(x,a)$, and robust optimization over these sets leads to policies that are more stable under model misspecification.

We study finite-horizon robust Markov decision processes with randomized policies and set-valued transition kernels. Our main focus is on ambiguity sets defined around reference transition laws. This setting is attractive because it models local distributional perturbations in a flexible and interpretable way. At the same time, it creates a basic difficulty for policy optimization. Under general ambiguity classes, such as Wasserstein or Kullback--Leibler balls, the worst-case transition kernel is typically not available in closed form. It is selected adversarially according to the value function generated by the current policy, so changing the policy can also change the worst-case model. This makes direct differentiation of the robust Bellman recursion difficult, and generally intractable without strong regularity assumptions on the optimizer (see Remark \ref{untractability_P_comment}).

\paragraph{Related work.} There are two main lines of literature behind our approach: robust Markov decision processes (MDPs) and distributionally robust optimization (DRO). On the robust MDPs side, the foundational papers \cite{Iyengar2005} and \cite{NilimElGhaoui2005} proved robust dynamic programming under rectangular ambiguity sets, showing that the Bellman recursion remains tractable when uncertainty separates across state--action pairs. This line was subsequently extended to $s$-rectangular ambiguity by \cite{wiesemannetal2013} and to distributionally robust MDPs with moment and nested ambiguity sets by \cite{xumannor2012}. %
Since then, robust MDPs have been studied under broader forms of model uncertainty (see \cite{LimXuMannor2013} for example). MDPs under general model uncertainty are studied in \cite{neufeld_markov_2023}. Policy gradient methods for robust MDPs have recently been developed in \cite{WangHoPetrik2023}, while robust Q-learning under Wasserstein ambiguity is studied in \cite{NeufeldSester2024}. Wasserstein-robust reinforcement learning has also been developed through direct minimax training \citep{AbdullahEtAl2019}, first-order methods for Wasserstein distributionally robust MDPs \citep{GrandClementKroer2021}, fast Bellman updates for Wasserstein robust Bellman operators \citep{FastBellmanWDRMDP2023}, and distributionally robust policy evaluation and learning for contextual bandits \citep{sietal2020,WDRContextualBandits2023}. In a related direction, \cite{coache2024robust} develop an actor--critic algorithm for robust \emph{risk-aware} reinforcement learning, in which robustness is obtained by considering all models within a Wasserstein ball around a reference model and policy gradients are derived from the quantile representation of dynamic distortion risk measures. A distinct line places Wasserstein ambiguity directly on the law of the return functional: \cite{jaimungaletal2022} assess a policy by the worst-case distribution within a Wasserstein ball around it under rank-dependent expected utility. Both are risk-sensitive and rely on quantile or distortion representations, whereas our objective is the risk-neutral worst-case return, with ambiguity acting directly on the transition kernels and the policy gradient obtained from directional derivatives of the robust Bellman recursion.

On the distributionally robust optimization side, a general convex framework and tractable reformulations are given in \cite{WiesemannKuhnSim2014}. For Wasserstein DRO, dual reformulations and performance guarantees were developed in \cite{EsfahaniKuhn2018}, while structural properties of worst-case distributions were analyzed in \cite{GaoKleywegt2023}. The duality result most relevant for our work is \cite{BlanchetMurthy2019}, which gives a scalar dual representation for Wasserstein DRO through optimal transport. Related links between Wasserstein robustness and regularization in statistical learning are established in \cite{BlanchetKangMurthy2019}. Sensitivity of robust values with respect to the Wasserstein radius is studied in \cite{BartlDrapeauOblojWiesel2021}, with related results under financial constraints in \cite{SaulduboisTouzi2024}, and dynamic robustness on path space in \cite{CompointSaulduboisTouzi2025}.

The closest policy-gradient works are
\cite{wang_policy_2022,WangHoPetrik2023,KumarDermanGeistLevyMannor2023}.
\cite{wang_policy_2022} exploit the closed-form structure of
$R$-contamination sets to derive robust policy-gradient and actor--critic
methods in discounted MDPs. \cite{WangHoPetrik2023} propose a double-loop
method that alternates between approximately computing a worst-case transition
kernel and applying the classical policy gradient under that kernel, with
global convergence guarantees in discounted tabular RMDPs.
\cite{KumarDermanGeistLevyMannor2023} derive an explicit robust policy gradient
for rectangular discounted RMDPs through a characterization of the worst
occupation measure. Our contribution is complementary. We combine the Wasserstein dual formulation of \cite{BlanchetMurthy2019} with a primal envelope argument over worst-case transition laws to derive exact one-sided sensitivities of the robust Bellman recursion without differentiating the optimal transition kernel. Our framework also differs from the robust Q-learning method of \cite{NeufeldSester2024}, which considers infinite-horizon discounted problems with deterministic policies and tabular updates. By contrast, we study finite-horizon problems with randomized parametric policies. In finite state--action spaces, our sensitivity recursion holds without additional stability assumptions and leads to an exact active-set ascent scheme. In compact continuous spaces, we provide complementary primal and dual sensitivity formulas under explicit stability conditions.  We further develop a function-approximation actor--critic implementation based on the resulting vector-valued recursion when the robust Bellman derivative is linear in the policy perturbation.
\begin{table}[H]
\centering
\caption{Comparison with related robust reinforcement learning and Wasserstein control methods.}
\vspace{-0.15cm}
\label{tab:related-work-comparison}
\resizebox{\textwidth}{!}{%
\begin{tabular}{lcccc}
\toprule
\textbf{Work} 
& \textbf{Ambiguity set} 
& \textbf{Criterion / horizon} 
& \textbf{Method} 
& \textbf{Wasserstein dual} \\
\midrule
\cite{wang_policy_2022} 
& $R$-contamination 
& finite horizon  
& policy gradient / actor--critic
& no \\
\cite{WangHoPetrik2023}
& rectangular transition uncertainty
& discounted
& policy gradient
& no \\
\cite{KumarDermanGeistLevyMannor2023}
& rectangular transition uncertainty
& finite horizon 
& policy gradient
& no \\
\cite{LiLanMurthySrikant2024}
& transition-kernel uncertainty
& average-cost
& policy mirror descent
& no \\
\cite{YangGuoXuLiuAnandkumar2023}
& distributional shift
& contextual bandit
& policy gradient
& problem-specific \\
\cite{NeufeldSester2024} 
& Wasserstein transition ambiguity
& discounted 
& Q-learning 
& yes \\
\cite{yang2021,KimYang2021} 
& Wasserstein LQ 
& finite-horizon control 
& Riccati 
&  penalized \\
\textbf{This paper} 
& Wasserstein transition kernels 
& finite horizon 
& policy gradient / actor--critic 
& yes \\
\bottomrule
\end{tabular}%
}
\end{table}
\paragraph{Our main contributions.}
We derive directional policy-gradient formulas for finite-horizon robust MDPs
with state--action-dependent Wasserstein ambiguity sets. Our approach combines
the dynamic programming recursion with Wasserstein duality and a primal
envelope argument, thereby avoiding differentiation of the implicit worst-case
transition kernel. Our main contributions are as follows.

\begin{enumerate}

\item \textbf{Dynamic programming and Wasserstein duality.}
We formulate a finite-horizon distributionally robust control problem with
randomized parametric policies and state--action rectangular Wasserstein
ambiguity sets. We establish the corresponding dynamic programming recursion
and use Wasserstein duality to express each one-step worst-case value as a
scalar maximization over the dual multiplier.

\item \textbf{Directional derivatives.}
We derive primal and dual envelope formulas for the one-sided directional
derivatives of the one-step worst-case value. In general, these derivatives are
the minimum or maximum expected downstream value derivative over the set of
worst-case transition laws
(Proposition~\ref{prop:primal-envelope}). In finite state--action spaces, this
set is a polytope obtained from the optimal face of a transport linear program,
and the resulting backward recursion holds without additional stability
assumptions (Theorem~\ref{thm: finite_case}). A complementary formula in terms
of the active dual multipliers and transport minimizers is obtained under an
explicit stability condition (Theorem~\ref{gradient J and F}).

\item \textbf{Vector policy gradient.}
Under the stability conditions of the dual formula and uniqueness of the active
dual multiplier and transport minimizer, the directional derivative is linear
in the policy perturbation and admits a vector-valued backward recursion. This
recursion combines the usual score-function term with the propagation of the
value-function derivative through the active worst-case transition law. When
the worst-case law is unique, it unrolls to the classical likelihood-ratio
policy-gradient formula under the frozen worst-case path measure.

\item \textbf{Exact finite-state ascent.}
In finite state--action spaces, we represent the robust objective as the
minimum of finitely many smooth policy-performance functions associated with
vertex transition selections. Based on their active sets, we construct an
exact max--min ascent scheme, prove monotonic improvement, and show that
accumulation points of the refinement iterates are directionally stationary.

\item \textbf{Numerical study.}
We evaluate the proposed recursions on finite-state benchmarks, where exact
robust dynamic programming provides a reference solution, and on a
continuous-control example with function approximation. The finite-state
experiments show that the proposed method closely recovers the exact benchmark
policies and that retaining the full backward derivative recursion is important
for policy accuracy. Additional experiments examine robustness under model misspecification.

\end{enumerate}

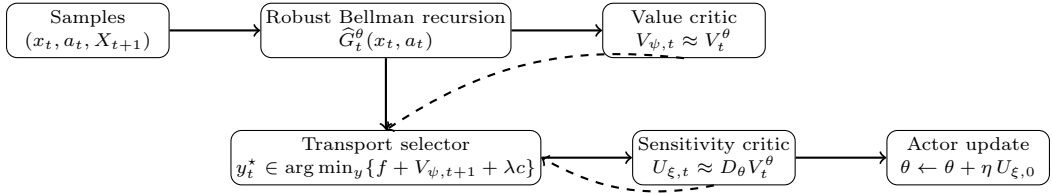
\begin{figure}[H]
\centering
\begin{tikzpicture}[
node distance=0.95cm and 1.2cm,
every node/.style={font=\scriptsize},
box/.style={
    rectangle, draw, rounded corners,
    align=center,
    minimum width=2.15cm,
    minimum height=0.72cm,
    inner sep=1.8pt
},
arrow/.style={->, thick}
]

\node[box] (sample) {Samples\\$(x_t,a_t,X_{t+1})$};

\node[box, right=of sample] (backup) {Robust Bellman recursion\\$\widehat G_t^\theta(x_t,a_t)$};

\node[box, right=of backup] (value) {Value critic\\$V_{\psi,t}\approx V_t^\theta$};

\node[box, below=of backup] (transport) {Transport selector\\$y_t^\star \in \arg\min_y\{f+V_{\psi,t+1}+\lambda c\}$};

\node[box, right=of transport] (grad) {Sensitivity critic\\$U_{\xi,t}\approx D_\theta V_t^\theta$};

\node[box, right=of grad] (actor) {Actor update\\$\theta \leftarrow \theta+\eta\,U_{\xi,0}$};

\draw[arrow] (sample) -- (backup);
\draw[arrow] (backup) -- (value);

\draw[arrow] (backup) -- (transport);
\draw[arrow] (transport) -- (grad);
\draw[arrow] (grad) -- (actor);

\draw[arrow, dashed] (value.south) to[bend right=20] (transport.north);
\draw[arrow, dashed] (grad.south) to[bend left=20] (transport.east);

\end{tikzpicture}
\caption{
Computational structure of the robust policy-gradient algorithm under uniqueness
of the active dual and transport optimizers. The value critic approximates the
Wasserstein-dual dynamic programming recursion, while the unique transport
selector $Y_{t,\theta,\lambda}^{*}$ drives the backward recursion for
$D_{\theta}V_t^{\theta}$ used in the actor update.
}
\label{fig:robust_ac_pipeline}
\end{figure}
\paragraph{Outline of the paper.}
The paper is organized as follows. Section \ref{sec: Problem formulation}
formulates the finite-horizon robust MDP with randomized policies and
set-valued transition kernels, defines the admissible path measures and robust
value functions, and establishes the corresponding dynamic programming
principle. Section \ref{sec: Actor-Critic algorithms} specializes the ambiguity
sets to Wasserstein balls, derives the associated dual recursion, and establishes
primal and dual formulas for the one-sided directional derivatives with respect
to the policy parameters. It also derives a vector-valued policy-gradient
recursion when the derivative of the one-step robust value is linear in the
policy perturbation, and introduces an exact active-set ascent method in finite
state--action spaces. Section \ref{sec: Numerical experiments} presents the
robust actor--critic implementation and reports numerical experiments on
benchmark models. Technical proofs and additional numerical results are
provided in the appendices.

\section{Problem formulation}\label{sec: Problem formulation}
\subsection{Notations}

\begin{itemize}
\item[$\bullet$] For a real Hilbert space $H$, we denote by $\langle \cdot,\cdot\rangle_H$ its scalar product and by
$\|h\|_H := \sqrt{\langle h,h\rangle_H}$ the associated norm. When there is no ambiguity on $H$, we
simply write $\langle \cdot,\cdot\rangle$ and $\|\cdot\|$.
\item[$\bullet$] Let $\Pc(B)$ denote the space of Borel probability measures on a Borel measurable space $B$.
\item[$\bullet$] We denote by $\Cc_b(B;\R)$ the space of bounded continuous real-valued
functions on $B$.
\item[$\bullet$] $\Cc_q(B; \R)$ denotes the set of continuous functions with polynomial growth at most of $q\in\R_+$.

\item[$\bullet$] For $S\subseteq\R^{d}$, $\mathrm{conv}(S)$ denotes its convex hull. For finite $S=\{v_1,\dots,v_m\}$,
$\mathrm{conv}(S)=\{\sum_{i}\alpha_i v_i:\alpha_i\ge0,\ \sum_i\alpha_i=1\}$.

\item [$\bullet$] Fix $q\ge 1$ and let $d : \Xc \times \Xc \to [0,+\infty)$ be a metric (typically derived from a Euclidean norm) on the compact space $\Xc$. The $q$-Wasserstein distance between two probability measures $\mu$ and $\mu'$ on $\Xc$ is defined by
\begin{align}
 W_q(\mu,\mu') :=
\left(
\inf_{\pi\in\Pi(\mu,\mu')}
\int_{\Xc\times\Xc} d(x,y)^q\,\pi(\d x,\d y)
\right)^{1/q},
\end{align}
where $\Pi(\mu,\mu')$ denotes the set of probability measures on
$\mathcal X\times\mathcal X$ with first marginal $\mu$ and second marginal
$\mu'$.

\item [$\bullet$] Let $T \in \N^*$. We denote by $\Tc$ the set of integers $\lbrace 0, 1,\ldots,T -1 \rbrace$ and by  $\bar{\Tc} :=  \Tc \cup \lbrace T \rbrace $. Given $t \in \Tc$, we also define $\Tc_t = \big \lbrace t, t+1,\ldots, T-1 \rbrace$.

\end{itemize}

\subsection{Preliminaries}

Let $\Xc \subset \R^d$ and $A \subset \R^m$ be two compact sets, endowed with their Borel
$\sigma$-algebras $\Sigma$ and $\Ac$, representing the state and action spaces, and let $T \in \N^*$
be a finite time horizon. We work on the canonical path space
$$
\Omega := (\Xc\times A)^T \times \Xc,
\quad
\omega=(x_0,a_0,\ldots,x_{T-1},a_{T-1},x_T)\in\Omega,
$$
with coordinate maps $X_{t}(\omega) := x_{t}$, $a_{t'}(\omega):= a_{t'}$ for
$(t,t') \in \bar{\Tc} \times \Tc$, and trajectory
map $S(\omega):=(X_0(\omega),a_0(\omega),\ldots,X_{T-1}(\omega),a_{T-1}(\omega),X_T(\omega))$.
At each time $t\in \Tc$, the set-valued transition kernel
$$
\Pc_t: \Xc\times A \rightrightarrows \Pc(\Xc),
\quad
(x,a)\mapsto \Pc_t(x,a)\subset \Pc(\Xc),
$$
captures distributional uncertainty: the conditional law of $X_{t+1}$ given $(X_t,a_t)=(x,a)$ is
only known to belong to $\Pc_t(x,a)$. Next, we define the set of admissible policies.
\begin{Definition}\label{rmk : admissible policy}
A  policy $\pi = (\pi_t)_{t \in \Tc}$ is said to be admissible if it verifies the following conditions:

\begin{itemize}
    \item [(i)] For any $t \in \Tc$, the mapping $\Xc \ni x \mapsto \pi_t(x,\d a) \in \Pc(A)$ is Borel-measurable.
    \item [(ii)] For any $t \in \Tc$, the map $\Xc \ni x \mapsto \pi_t(x,\cdot) \in \Pc(A)$ is continuous for the topology of weak convergence, i.e., for any $x \in \Xc$ and any sequence $(x_n)_n$ converging to $x$, the following holds true    \begin{align}
    \int_A g(a)\, \pi_t(x_n,\d a) \underset{n\to+\infty}{\longrightarrow} \int_A g(a)\, \pi_t(x,\d a ), \quad \forall g \in \Cc(A;\R).
\end{align}
    \item [(iii)] Fix $q \ge 1$. There exists a constant $C \geq 1$ such that for any $(t,x) \in \Tc \times \Xc$, we have 
\begin{align}
    \int_A \lVert a \rVert^q \pi_t(x,\d a) \leq C \big( 1 +  \lVert x \rVert^q).
\end{align}
\end{itemize}
\end{Definition}

We  denote by $\Pi := \big \lbrace \pi = (\pi_t)_{t \in \Tc}: \text{$\pi$ is admissible} \big \rbrace$ the set of admissible randomized control policies. For each $t \in \Tc$, we consider a set-valued transition kernel
$$\Pc_t : \Xc \times A \;\rightrightarrows\; \Pc(\Xc).$$
For every $(x,a) \in \Xc \times A$, the set $\Pc_t(x,a)$ represents the
collection of admissible probability laws for the next-state variable at time $t$
when the system is in state $x$ and action $a$ is applied.

\begin{Assumption}\label{assumptions : set of measures} Assume the following holds:
\begin{itemize} 
\item[(i)]  We endow $\Pc(\Xc)$ with the $q$-Wasserstein metric $W_q$ associated with the ground metric $d$ on $\Xc$. 
For nonempty compact sets $K_1,K_2 \subset \Pc(\Xc)$, we define the Hausdorff distance induced by $W_q$ by 
\begin{align*}
d_H(K_1,K_2)
:= \sup\Big\{
\sup_{\mu\in K_1}\inf_{\nu\in K_2} W_q(\mu,\nu),\;
\sup_{\nu\in K_2}\inf_{\mu\in K_1} W_q(\mu,\nu)
\Big\}.
\end{align*}
For any $t \in \Tc$, the correspondence
\begin{align*}
\Xc \times A \ni (x,a) \rightrightarrows \Pc_t(x,a) \subset \Pc(\Xc)
\end{align*}
is assumed to be nonempty, compact-valued, and continuous in the sense that
\begin{align*}
(x_n,a_n) \to (x,a)
\quad \Longrightarrow \quad
d_H\big(\Pc_t(x_n,a_n), \Pc_t(x,a)\big) \to 0.
\end{align*}
\item[(ii)] There exists a constant $C_P \geq 1$ such that for any $(t,x,a) \in \Tc \times \Xc \times A$ and $\P_t(x,a,\cdot) \in \Pc_t(x,a)$,
\begin{align}
    \int_{\Xc}(1+ \lVert x' \rVert^q )\P_t(x,a, \d x')\leq C_P ( 1 + \lVert x \rVert^q + \lVert a \rVert^q)
\end{align}

\end{itemize} 
\end{Assumption}
We consider a set of randomized policies $\pi=(\pi_t)_{t\in\Tc}\in\Pi$ (see Definition \ref{rmk : admissible policy}),
with $a_t$ drawn from $\pi_t(X_t,\d a)$. For a given initial state $x\in \Xc$ and policy
$\pi\in\Pi$, we denote by $\Bc_{x,\pi}\subset \Pc(\Omega)$ (see Definition \ref{defi:path_measures_space}) the set of path measures
consistent with $\pi$ and an admissible selection of transition laws from $\Pc_t(X_t,a_t)$. The agent chooses a policy $\pi$, and nature responds with the worst-case measure $\P$ in
$\Bc_{x,\pi}$. 
\begin{Definition}\label{defi:path_measures_space}
Under Assumption \ref{assumptions : set of measures}, and for any initial state
$x \in \Xc$ and any policy $\pi \in \Pi$, we define the corresponding set of
admissible probability measures on $\Omega$ as
\begin{equation*}
\begin{aligned}\label{eq : Admissible set of measures}
    \Bc_{x, \pi} := \Big \lbrace   \P &\in \Pc(\Omega) : \P(\d \omega) = \delta_x(\d x_0) \Prod_{t=0}^{T-1} \pi_t(x_t ,\d a_t) \Prod_{t=0}^{T-1} \P_t(x_t,a_t, \d x_{t+1})   \notag,~ \text{with}~~ \\
    & \P_t : (\Xc \times A)  \to \Pc(\Xc) \text{ Borel-measurable and $\P_t(x,a) \in \Pc_t(x,a)$}, \text{ for all $(x,a)\in \Xc \times A$ and $t \in \Tc$} \Big \rbrace.
\end{aligned}
 \end{equation*}
We also define, for each $t \in \bar{\Tc}$, the truncated set 
$\Bc_{x,\pi \mid t}$ as
\begin{equation*}
\begin{aligned}
    \Bc_{x,\pi | t} :&= \Big \lbrace \P \in \Pc\big((\Xc \times A)^{T-t} \times \Xc \big) : \P(\d \omega) = \delta_{x}(\d x_{t}) \prod_{l=t}^{T-1} \pi_l(x_l, \d a_l) \prod_{l=t}^{T-1} \P_l(x_l, a_l, \d x_{l+1})  \notag,~ \text{with}  \\
    & \qquad \P_l : (\Xc \times A) \to \Pc(\Xc) \text{ Borel-measurable and $\P_l(x,a) \in \Pc_l(x,a)$},  \text{ for all $(x,a)\in \Xc \times A$ and $l \in \Tc_t$} \Big \rbrace,
\end{aligned}
 \end{equation*}
where we set the empty product to be equal to $1$ by convention. By a slight abuse of notation, the symbol $\mathbb P$ is used both for path
measures and, with a time index, for transition kernels. When written as
$\mathbb P\in\mathcal P(\Omega)$ it denotes a measure on trajectories, whereas
$\mathbb P_l(x,a,\cdot)$ denotes a one-step transition probability on
$\mathcal X$.
\end{Definition}
\begin{Remark}\label{rmk : recutangular_ambiguity_sets} 
In our setting, define
$$
\mathcal P_l^0
:=
\left\{
\mathbb P_l:\mathcal X\times A\to\mathcal P(\mathcal X)
:
\mathbb P_l \text{ is Borel-measurable and }
\mathbb P_l(x,a,\cdot)\in\mathcal P_l(x,a)
\text{ for all }(x,a)\in\mathcal X\times A
\right\}.
$$
We say that the ambiguity set is rectangular if $\mathcal B=\prod_{l\in\mathcal T}\mathcal P_l^0.$ Rectangularity means that admissible transition kernels can be selected
separately at each time, with no additional coupling constraints across stages.
This product structure is standard in robust MDPs and is the structural condition
that makes dynamic programming possible. In particular, for any function $\Phi:\mathcal B\to\mathbb R$,
$$
\inf_{(\mathbb P_l)_{l\in\mathcal T}\in\mathcal B}
\Phi((\mathbb P_l)_{l\in\mathcal T})
=
\inf_{\mathbb P_{T-1}\in\mathcal P_{T-1}^0}
\cdots
\inf_{\mathbb P_0\in\mathcal P_0^0}
\Phi((\mathbb P_l)_{l\in\mathcal T}).
$$
This identity follows from the Cartesian-product structure. In the dynamic
programming proof, it is combined with measurable selection to reduce the robust optimization problem to a sequence of one-step minimizations.
\end{Remark}

\subsection{Robust MDP setting} 
Consider Borel measurable functions
$$
f:\mathcal T\times\mathcal X\times A\times\mathcal X\to\mathbb R,
\qquad
g:\mathcal X\to\mathbb R,
$$
representing the running and terminal rewards, respectively. Recall that the canonical process on the trajectory space $\Omega$ is $S = (X_0,a_0,X_1,a_1,\ldots,X_{T-1},a_{T-1},X_T)
$. We define the total reward mapping $R : \Omega \to \mathbb R$ by
\begin{align}\label{eq : value function V}
    R : \Omega \ni \omega = (x_0, a_0 , x_1,a_1 , \ldots, x_{T-1} ,  a_{T-1}, x_T) \mapsto \sum_{t=0}^{T-1} f(t, x_t, a_t,x_{t+1}) + g(x_T).
\end{align}
For a fixed admissible policy $\pi\in\Pi$, its robust value function is defined,
for $t\in\bar{\mathcal T}$ and $x\in\mathcal X$, by
$$
V_t^\pi(x)
:=
\inf_{\mathbb P\in\mathcal B_{x,\pi|t}}
\mathbb E^{\mathbb P}
\left[
\sum_{l=t}^{T-1} f(l,X_l,a_l,X_{l+1})+g(X_T)
\right],
$$
with the convention that $\sum_{l=T}^{T-1}=0$. The corresponding optimal robust value function is
$$
V_t(x)
:=
\sup_{\pi\in\Pi}
\inf_{\mathbb P\in\mathcal B_{x,\pi|t}}
\mathbb E^{\mathbb P}
\left[
\sum_{l=t}^{T-1} f(l,X_l,a_l,X_{l+1})+g(X_T)
\right].
$$
In particular, the value of the original problem is $V_0(x)$.

 We now make the following assumptions on the running and terminal cost functions.
\begin{Assumption}\label{assumption : reward functions}Assume the following: 
\begin{itemize}
\item[(i)] For any $t \in \Tc $, the mapping 
\begin{align}
     \Xc \times A \times \Xc \ni (x,a, x') \mapsto f(t,x,a,x') \in \R
\end{align}
is continuous. In particular, for any $q>0$,
$f(t,\cdot,\cdot,\cdot)\in C_q(\mathcal X\times A\times\mathcal X;\mathbb R).$

\item[(ii)] Moreover, the mapping 
\begin{align}
    \Xc \ni x \mapsto g(x),
\end{align}
is continuous. Moreover, for any $q>0$, $g\in C_q(\mathcal X;\mathbb R).$
 \end{itemize}
\end{Assumption}
\begin{Remark}
Compactness of $\mathcal X$ is used as a convenient sufficient condition in
several parts of the analysis: to obtain attainment in $\mathcal F^\lambda$ and
in the argmin set $Y^*_{t,\theta,\lambda}$, boundedness of $c$ in
\eqref{eq:Phi-lambda-Lip}, applicability of Berge's maximum theorem in Theorem
\ref{thm : DPP + P^*}, and uniform boundedness of $V_t^\pi$ in dominated
convergence arguments. These arguments may extend to non-compact state spaces under additional
coercivity, inf-compactness, continuity, and moment assumptions. For example,
when $c(x,y)=\|x-y\|^q$ with $q>p$, coercivity can help recover attainment,
provided the value functions have at most $p$-polynomial growth. Moment bounds,
inf-compact versions of Berge's theorem, and polynomial growth estimates can
then replace the compactness-based arguments. We impose compactness here to
avoid these extra technical conditions.
\end{Remark}
\subsection{Dynamic programming principle}
\begin{Theorem}\label{thm : DPP + P^*} Under Assumptions \ref{assumptions : set of measures} and \ref{assumption : reward functions}, we have that:
\begin{itemize}
\item[(i)] The following dynamic programming principle holds true, for all $(t,x) \in \Tc\times\Xc$ and $\pi \in \Pi$,
\begin{align}\label{eq : Dynamic programming for robust value functions} 
\left\{\begin{aligned}  
       V^{\pi}_T(x) &= g(x),  \\
       V^{\pi}_{t}(x)  &= \E_{a \sim \pi_t(x,\cdot)} \Big[ \underset{\P_{t} \in \Pc_t(x,a)}{\inf}\E^{\P_{t}} \big[ f \big(t,x, a,X_{t+1} \big) + V^{\pi}_{t+1}(X_{t+1})\big] \Big].
    \end{aligned}
        \right.
\end{align}
Moreover, it implies that for any $t \in \bar{\Tc}$ and any $\pi \in\Pi$, the value function $V_t^{\pi}$ lies in $\Cc(\Xc;\R)$.

\item[(ii)] There exists a Borel-measurable selector $\P_t^{*,\pi}:\mathcal X\times A\to\mathcal P(\mathcal X)$
such that $\P_t^{*,\pi}(x,a)\in\mathcal P_t(x,a)$ and
\begin{equation}
\label{eq : local optimization problem}
\inf_{\P_t\in\mathcal P_t(x,a)}
\int_{\mathcal X}
\left[
f(t,x,a,x')+V_{t+1}^\pi(x')
\right]
\P_t(\d x')
=
\int_{\mathcal X}
\left[
f(t,x,a,x')+V_{t+1}^\pi(x')
\right]
\P_t^{*,\pi}(x,a,\d x'),
\end{equation}
for all $(t,x,a)\in\mathcal T\times\mathcal X\times A$. We then define the path measure $\mathbb P^{*,\pi}\in\mathcal P(\Omega)$ by
\begin{align}\label{eq : Definition P^*}
\mathbb P^{*,\pi}(\d\omega)
:=
\delta_{x_0}(\d x_0)
\prod_{t=0}^{T-1}\pi_t(x_t,\d a_t)
\prod_{t=0}^{T-1}\P_t^{*,\pi}(x_t,a_t,\d x_{t+1}).
\end{align}
Consequently,
\begin{align}\label{eq : representation of V^pi}
V_0^\pi(x_0)
=
\inf_{\mathbb P\in\mathcal B_{x_0,\pi}}
\mathbb E^{\mathbb P}[R]
=
\mathbb E^{\mathbb P^{*,\pi}}[R].
\end{align}
\end{itemize}
\end{Theorem}
\begin{proof}
    The proof follows the robust dynamic programming argument of \cite{neufeld_markov_2023}, whose setting covers finite-horizon problems on compact Borel state--action spaces with nonempty, compact-valued, Hausdorff-continuous rectangular ambiguity correspondences; randomized admissible policies are accommodated by disintegrating the path measure over the action variable at each stage. The argument is consistent
with the classical framework of \cite{Iyengar2005}. The reduction of the trajectory-wise infimum over $\Bc_{x,\pi|t}$ to iterated one-step minimizations combines the rectangular product structure of Remark \ref{rmk : recutangular_ambiguity_sets} with the interchange of infimum and integration (Theorem \ref{thm : Interchange of Min and Int}) and the measurable-selection theorem (Theorem \ref{thm : measurable_selection}), applied to the normal integrand defined by the one-step robust value. By Assumption \ref{assumptions : set of measures}, the correspondence
$(x,a)\mapsto\mathcal P_t(x,a)$ is nonempty, compact-valued, and Hausdorff
continuous. Since $x'\mapsto f(t,x,a,x')+V_{t+1}^\pi(x')$ is continuous and bounded on the compact set $\mathcal X$, the map
\[
(x,a,\mathbb P)\mapsto
\int_{\mathcal X}
\left[f(t,x,a,x')+V_{t+1}^\pi(x')\right]\mathbb P(dx')
\]
is continuous on the graph of $\mathcal P_t$. Berge's maximum theorem gives
continuity of the one-step robust value, and the measurable maximum theorem
gives a Borel selector $\mathbb P_t^{*,\pi}(x,a)$. Backward induction yields the displayed
Bellman recursion and the representation by the product measure
$\mathbb P^{*,\pi}$.
\end{proof}

\begin{Remark}
    Note that we can define also the $Q$-function of this Robust MDP such that
    \begin{align}
        V_t^{\pi}(x) = \E_{a \sim \pi_t(x,\cdot)} \big[ Q^{\pi}_t(x,a) \big]
    \end{align}
for every $x\in\Xc$ and $a\in A$, where the Robust $Q$-function can be defined as 
\begin{align}\label{eq : Dynamic programming for robust Q functions}
\left\{\begin{aligned}  
       Q^{\pi}_T(x,a) &= g(x),  \\
       Q^{\pi}_{t}(x,a)  &:= \underset{\P_{t} \in \Pc_{t}(x,a)}{\inf}\E^{\P_{t}} \big[ f(t,x,a, X_{t+1}) +   V^{\pi}_{t+1}(X_{t+1})\big], \quad t \in \Tc.
    \end{aligned}
        \right.
\end{align}

\end{Remark}

While the representation \eqref{eq : representation of V^pi} is useful
theoretically, it is not directly tractable in general, since the \eqref{eq : local optimization problem} require optimization over infinite-dimensional
sets of probability measures. This motivates the Wasserstein specification and
dual reformulation developed in the next section.

\section{Robust policy gradient representation}\label{sec: Actor-Critic algorithms}
\subsection{Parametric Wasserstein ambiguity sets}
In this section, we specify the ambiguity sets $\mathcal P_t(x,a)$, for
$(t,x,a)\in\mathcal T\times\mathcal X\times A$, in order to obtain a tractable
form of the robust optimization problem \eqref{eq : value function V}. We work
with distributional ambiguity measured by the Wasserstein distance
$W_q$, for some $q\in\mathbb N^*$. For each $t\in\mathcal T$, let $\P_t^0$ be a reference transition kernel such
that
$$
\mathcal X\times A\ni (x,a)\mapsto \P_t^0(x,a,\cdot)\in\mathcal P(\mathcal X)
$$
is Borel-measurable and satisfies $\int_{\mathcal X}\|x'\|^q\,\P_t^0(x,a,\d x')<+\infty,$ for all $(x,a)\in\mathcal X\times A.$
For $\varepsilon > 0$ and $t \in \Tc $, we define the following ambiguity set
\begin{align}\label{eq : Wasserstein ambiguity set}
    \Xc \times A \ni (x,a) \mapsto  \mathbb{B}_t^{\varepsilon,q} \big(\P_t^0(x,a,\cdot)\big) := \big \lbrace \P \in \Pc(\Xc) : \Wc_q \big(\P,\P_t^0(x,a,\cdot) \big) \leq \varepsilon \big \rbrace.
\end{align}
We recall from \cite{neufeld_markov_2023} (see Proposition 3.1) that the mapping \eqref{eq : Wasserstein ambiguity set} satisfies
Assumption \ref{assumptions : set of measures}.

We next restrict the admissible randomized policies to a parametrized class.
Let $\nu$ be a reference measure on $A$ and let $\Theta$ be a compact parameter
set. For each $\theta\in\Theta$, we assume that the policy
$\pi^\theta=(\pi_t^\theta)_{t\in\mathcal T}$ admits a density with respect to
$\nu$, namely
$$
\pi^\theta_t(x,\d a)
=
\pi^\theta_t(x,a)\,\nu(\d a)\in\mathcal P(A),
\quad \forall(t,x)\in\mathcal T\times\mathcal X,
$$
where we define the mapping 
\begin{align}
    \Theta \ni \theta \mapsto \pi^{\theta} \in \Fc := \big \lbrace \pi : \Tc \times \Xc \times A \to \R^+: \int_{A} \pi_t(x,a) \nu (\d a ) = 1, \quad \forall(t,x) \in \Tc \times \Xc \big \rbrace.
\end{align}

\begin{Remark} Examples of admissible policies include:
    \begin{itemize}
        \item  When $A$ is finite, a natural choice is the softmax parametrization defined as
    \begin{align}
        \pi^{\theta}_t(x,a) := \frac{e^{\mu_{\theta}(t,x,a)}}{\sum_{a' \in A} e^{\mu_{\theta}(t,x,a')}}, \quad \forall a \in A,
    \end{align}
    with respect to the counting measure $\nu(\d a) = \sum_{a' \in A} \delta_{a'}(\d a)$.
        \item When $A$ is continuous, we can choose for instance a Gaussian distribution
        $$
\pi_t^\theta(x,a)
=
\frac{
\exp\left(
-\frac{\|a-\mu_\theta(t,x)\|^2}{2\sigma_\theta^2}
\right)
}{
\int_A
\exp\left(
-\frac{\|b-\mu_\theta(t,x)\|^2}{2\sigma_\theta^2}
\right)\d b
},
\quad \forall a\in A.
$$
        with respect to the Lebesgue measure on $A$.
    \end{itemize}
\end{Remark}
\subsection{Duality results}

 In this framework, the problem \eqref{eq : value function V} can be rewritten as
\begin{align}\label{eq : value function V Gradient}
    \Xc \ni x  \mapsto  V(x) :=\underset{\theta \in \Theta}{\sup}~ \underset{ \P \in \Bc_{x,\pi^{\theta}}}{\inf} \E^{\P} \big[ R(S) \big].
\end{align}
The supremum over $\theta\in\Theta$ is understood as the supremum over the
parametrized policy class $\{\pi^\theta:\theta\in\Theta\}$.
\begin{Assumption}\label{assump:regularity_pi_theta} Assume the following holds:
    \begin{itemize}
        \item[(i)] For each $t\in\Tc$, $(x,a)\mapsto \P_t^0(x,a,\cdot)\in\Pc(\Xc)$ is continuous for the weak topology, i.e.,
        $$\int_{\Xc}\varphi(x')\,\P_t^0(x_n,a_n,\d x')
\underset{n\to +\infty}{\to}
\int_{\Xc}\varphi(x')\,\P_t^0(x,a,\d x'),
\quad \forall \varphi\in C_b(\Xc),$$
for any sequence $(x_n,a_n) \underset{n \to +\infty}{\to} (x,a)$.
        \item[(ii)] For each $t\in\Tc$ the map
$(\theta,x)\mapsto \pi^\theta_t(x,\cdot)\in\Pc(A)$ is weakly continuous on $\Theta\times\Xc$, and $\theta\mapsto \pi^\theta_t(x,\cdot)$ is $C^1(\Theta)$ in the sense that
for all $\varphi\in C(A)$ the map $\theta\mapsto \int_A \varphi(a)\,\pi_t^\theta(x,\d a)$ is $C^1(\Theta)$ with derivative continuous in $(\theta,x)$.
Moreover, the score is uniformly bounded, such that
$$\sup_{\theta\in\Theta} \sup_{(t,x,a)\in\Tc\times\Xc\times A}\big\|\nabla_\theta\log \pi_t^\theta(x,a)\big\|\le C_\pi<+\infty.$$
    \end{itemize}
\end{Assumption}
Let $\mu_0\in\mathcal P(\mathcal X)$ denote the initial distribution of $X_0$.
Let $\P_t^{*,\theta}$ denote the measurable worst-case transition selector
associated with the policy $\pi^\theta$, as obtained in Theorem
\ref{thm : DPP + P^*}. Define the corresponding path measure
$\mathbb P^{*,\theta}\in\mathcal P(\Omega)$ by
$$
\mathbb P^{*,\theta}(\d\omega)
:=
\mu_0(\d x_0)
\prod_{t=0}^{T-1}
\pi_t^\theta(x_t,\d a_t)
\prod_{t=0}^{T-1}\P_t^{*,\theta}(x_t,a_t,\d x_{t+1}).
$$
Then, by \eqref{eq : representation of V^pi}, the robust performance criterion
is
$$
J(\theta)
:=
\mathbb E^{\mathbb P^{*,\theta}}[R].
$$

In the following, we write indifferently $V^{\theta} = V^{\pi^{\theta}}$.
We are interested in a suitable representation of the following mapping
\begin{align}\label{eq : J(theta)}
    \Theta \ni \theta \mapsto J(\theta) = \E_{X_0\sim \mu_0} \big[  V^{\theta}_0(X_0) \big].
\end{align}
Note that we no longer optimize over the full admissible policy
class $\Pi$. Instead, we restrict attention to the parametric family
$\{\pi^\theta:\theta\in\Theta\}$. The gradient results characterize local
sensitivity of $J(\theta)$ within this family. They do not imply global
optimality over $\Pi$ unless the parametrization is sufficiently rich and the optimization over $\Theta$ is solved globally.
\begin{Remark}\label{untractability_P_comment}
We recall that our objective is to obtain a tractable representation of the
derivative of the robust performance criterion $J(\theta)$ defined in
\eqref{eq : J(theta)}. The main difficulty is that the worst-case transition
selector depends on the policy parameter $\theta$ through the value function.

Indeed, for $t\in\mathcal T$, the dynamic programming recursion can be written as
$$
V_t^\theta(x)
=
\int_A\int_{\mathcal X}
\left[
f(t,x,a,x')+V_{t+1}^\theta(x')
\right]
\P_t^{*,\theta}(x,a,\d x')
\pi^\theta_t(x,a)\,\nu(\d a).
$$
A direct differentiation of this expression would require differentiability, in
a suitable weak sense, of the selector
$\theta\mapsto \P_t^{*,\theta}(x,a,\cdot)$, together with domination conditions
allowing the derivative to pass through the integrals. Such regularity is
generally difficult to verify, since $\P_t^{*,\theta}$ is defined implicitly as
an optimizer of the local robust problem and depends on
$V_{t+1}^\theta$. We emphasize, however, that envelope-type arguments never require
differentiating the optimizer map itself. The genuine obstruction is that the
worst-case law may be non-unique and may switch abruptly with $\theta$. We
therefore develop two complementary routes that avoid differentiating
$\P_t^{*,\theta}$: a dual formulation, which replaces the optimization over
measures by a scalar dual problem, and a primal envelope formula
(Section \ref{subsec:primal}), which differentiates directly through the
$\theta$-independent Wasserstein ball and handles non-unique worst-case laws.
\end{Remark}

We now specialize to the Wasserstein ambiguity sets defined in
\eqref{eq : Wasserstein ambiguity set}. To tackle these ambiguity sets, we use the duality results from \cite{BlanchetMurthy2019}. We first introduce the auxiliary quantities needed
for this dual formulation.

\begin{Proposition}[Theorem 1, \cite{BlanchetMurthy2019}]\label{prop:duality}
    Let $c : \Xc  \times \Xc \ni (x,y) \mapsto c(x,y) = \lVert x- y \rVert^q$ and $\lambda \geq 0$. We define the operator $\Fc^{\lambda}$ as 
\begin{align}\label{dual_operator}
\mathcal F^\lambda(V)(x)
:=
\sup_{y\in\mathcal X}
\left\{
V(y)-\lambda c(x,y)
\right\},
\quad \forall x\in\mathcal X.
\end{align}
Then, the following duality result holds for any $V \in C(\mathcal X;\mathbb R)$:
\begin{align}\label{eq : duality result Wasserstein distance}
    \underset{\P \in \mathbb{B}^{\varepsilon,q}(\P^0)}{\inf}~ \E^{\P} \big[V(X) \big] = \underset{\lambda \geq 0}{\sup }~  \E^{\P^0} \big[ - \Fc^{\lambda}(-V)(X) \big]- \varepsilon^q \lambda,
\end{align}
where $X\sim \P^0$. 
\end{Proposition}
In the notation of \cite{BlanchetMurthy2019}, the radius constraint
$\Wc_q(\P,\P^0)\le\varepsilon$ corresponds to the transport-cost budget
$\delta=\varepsilon^q$ for the cost $c(x,y)=\|x-y\|^q$, which produces the term
$-\varepsilon^q\lambda$ in \eqref{eq : duality result Wasserstein distance};
continuity of $V$ on the compact set $\Xc$ (hence upper semicontinuity and
integrability) suffices for their Theorem 1, and the infimum on the left-hand
side is attained since the ball is weakly compact and $\P\mapsto\E^\P[V(X)]$ is
weakly continuous.

The duality result \eqref{eq : duality result Wasserstein distance} reduces the optimization problem from a set of probability measures to a one-dimensional optimization over the dual variable $\lambda$.

\begin{Remark}
    \label{G_dpp}
    Note that the DPP defined in \eqref{eq : Dynamic programming for robust value functions} can be written as
\begin{align}\left\{\begin{aligned}
       V_T^\theta(x)&=g(x),  \\
       V_t^\theta(x)&=\E_{a\sim\pi^\theta_t(x,\cdot)}\Big[\sup_{\lambda\ge 0}\Big(
\E_{X\sim \P_{t}^0(x,a,\cdot)}\big[-\Fc^{\lambda}(-f(t,x,a,\cdot)-V^\theta_{t+1}(\cdot))(X)\big]
-\varepsilon^q\lambda
\Big)\Big],~ t=T-1,\dots,0.
    \end{aligned}
        \right.
\end{align}
\end{Remark}

\begin{Remark}
\label{prop:neural_operator_Phi}
The operator $\mathcal F^\lambda$ may also be approximated by operator-learning
methods. For example, Theorem 4.5 in \cite{bayraktar2025dno} shows that, under
compactness assumptions on $\mathcal X$ and on the relevant class of input
functions, $\mathcal F^\lambda$ can be approximated uniformly over
$\lambda\in[0,\Lambda]$. More precisely, if
$\mathcal O\subset C(\mathcal X;\mathbb R)$ is compact in the supremum norm,
then for every $\varepsilon>0$ there exist integers $N_1,N_2$ and parameters
$\theta$ such that
$$
\sup_{\lambda\in[0,\Lambda]}\sup_{b\in\mathcal O}
\|\mathcal F_\theta^\lambda(b)-\mathcal F^\lambda(b)\|_\infty
\leq \varepsilon.
$$
This suggests a possible offline approximation of the robust Bellman operator in
high-dimensional implementations, although this approximation is not used in
the theoretical results below.
\end{Remark}
\subsection{Measurability of selectors}\label{Measurability and uniqueness of selectors}
The Wasserstein dual formulation involves two optimization problems: a pointwise
minimization over the transported next state $y$ and a scalar maximization over
the dual multiplier $\lambda$. We first isolate the optimizer correspondences
that will be used in the policy-gradient analysis.

For $t\in\mathcal T$, $\theta\in\Theta$, $(x,a)\in\mathcal X\times A$,
$x'\in\mathcal X$, and $\lambda\geq0$, define
$$
Y_{t,\theta,\lambda}^*(x,a;x')
:=
\argmin_{y\in\mathcal X}
\left\{
f(t,x,a,y)+V_{t+1}^\theta(y)+\lambda c(x',y)
\right\}.
$$
The corresponding scalar dual objective is
$$
F_t^\theta(\lambda;x,a)
:=
\int_{\mathcal X}
\inf_{y\in\mathcal X}
\left\{
f(t,x,a,y)+V_{t+1}^\theta(y)+\lambda c(x',y)
\right\}
\P_t^0(x,a,\d x')
-
\varepsilon^q\lambda,
$$
and the set of optimal dual multipliers is
$$
\Lambda_t^{*,\theta}(x,a)
:=
\argmax_{\lambda\geq0}F_t^\theta(\lambda;x,a).
$$

Here $\P_t^0(x,a,\cdot)$ is the nominal transition law. Thus, after fixing
$(x,a)$, the Wasserstein dual formulation replaces the optimization over
admissible next-state distributions by a scalar optimization over $\lambda$ and
a pointwise transport problem in $y$. The sets
$Y_{t,\theta,\lambda}^*(x,a;x')$ and $\Lambda_t^{*,\theta}(x,a)$ encode the
corresponding optimizers and determine the sensitivity of the robust Bellman
operator.

The results below focus on measurability, which is sufficient for the
directional derivative formulas to be well-defined. Sufficient conditions for
uniqueness of the relevant optimizers are given in Appendix \ref{app:unique}.

\begin{Proposition}\label{eq:uniform-cont}
Under Assumptions \ref{assumption : reward functions} and
\ref{assump:regularity_pi_theta}, for every $\theta\in\Theta$ and every
$t\in\bar{\mathcal T}$, $V_t^\theta\in C(\mathcal X;\mathbb R)$.
\end{Proposition}
\begin{proof}
    See Appendix \ref{app: proof_cont}. 
\end{proof}

\begin{Lemma}
\label{lem:compact_dual_maximizer_set}
Assume that Assumption \ref{assumption : reward functions} holds, that
$c(x,x)=0$ for all $x\in\mathcal X$, and that $\varepsilon>0$. Then, for every
$(t,\theta,x,a)\in\mathcal T\times\Theta\times\mathcal X\times A$,
$$
F_t^\theta(\lambda;x,a)\to -\infty
\qquad\text{as }\lambda\to\infty.
$$
Consequently, the supremum over $\lambda\geq0$ is attained on a compact
interval. More precisely, if the value functions $V_{t+1}^\theta$ are uniformly
bounded over $(t,\theta)$, then there exists $\Lambda<+\infty$, independent of
$(t,\theta,x,a)$, such that
$$
\Lambda_t^{*,\theta}(x,a)
=
\argmax_{\lambda\in[0,\Lambda]}F_t^\theta(\lambda;x,a).
$$
\end{Lemma}

\begin{proof}
By Assumption \ref{assumption : reward functions} and compactness of
$\mathcal X$ and $A$, the functions $f$ and $g$ are bounded. Since the horizon
is finite, the robust Bellman recursion implies that the value functions
$V_s^\theta$ are uniformly bounded. Hence there exists $C<+\infty$ such that $|f(t,x,a,x')+V_{t+1}^\theta(x')|\leq C$ for all $(t,\theta,x,a,x')\in [0,T]\times \Theta\times \Xc\times A\times \Xc$. For any $x'\in\mathcal X$, the admissible choice $y=x'$ gives
$$
\inf_{y\in\mathcal X}
\left\{
f(t,x,a,y)+V_{t+1}^\theta(y)+\lambda c(x',y)
\right\}
\leq
f(t,x,a,x')+V_{t+1}^\theta(x'),
$$
because $c(x',x')=0$. Therefore,
$$
F_t^\theta(\lambda;x,a)
\leq
\int_{\mathcal X}
\left[
f(t,x,a,x')+V_{t+1}^\theta(x')
\right]
\P_t^0(x,a,\d x')
-
\varepsilon^q\lambda
\leq
C-\varepsilon^q\lambda.
$$
Since $\varepsilon>0$, this implies
$$
F_t^\theta(\lambda;x,a)\to -\infty
\qquad\text{as }\lambda\to\infty.
$$

Moreover, by the same boundedness, there exists $C_0<+\infty$ such that $F_t^\theta(0;x,a)\geq -C_0$
for all $(t,\theta,x,a)\in [0,T]\times \Theta\times \Xc\times A$. Choose $\Lambda>0$ such that $C-\varepsilon^q\Lambda<-C_0.$ Then, for every $\lambda\geq\Lambda$,
$$
F_t^\theta(\lambda;x,a)
\leq
C-\varepsilon^q\lambda
\leq
C-\varepsilon^q\Lambda
<
-C_0
\leq
F_t^\theta(0;x,a).
$$
Thus no maximizer lies in $[\Lambda,\infty)$. By continuity of
$\lambda\mapsto F_t^\theta(\lambda;x,a)$, the maximum is attained
on $[0,\Lambda]$, and
$$
\Lambda_t^{*,\theta}(x,a)
=
\argmax_{\lambda\in[0,\Lambda]}F_t^\theta(\lambda;x,a).
$$
\end{proof}

Such compact reductions are a standard consequence of Wasserstein dual
formulations in compact settings. Here, compactness of $\mathcal X$ makes the
transport cost bounded, while the term $-\varepsilon^q\lambda$ forces the dual
objective to $-\infty$ as $\lambda\to\infty$. Therefore the maximization over
$\lambda\geq0$ can be restricted to a compact interval. This is consistent with
the Wasserstein DRO duality framework used in
\cite{BlanchetMurthy2019,EsfahaniKuhn2018,GaoKleywegt2023}.
\begin{Proposition}
\label{prop:measurable_inner_selector}
Fix $t\in\mathcal T$ and $\theta\in\Theta$. Under Assumption \ref{assumption : reward functions}, the correspondence
$$
(x,a,x',\lambda)
\mapsto
Y_{t,\theta,\lambda}^*(x,a;x')
$$
from $\mathcal X\times A\times\mathcal X\times[0,\Lambda]$ into
$\mathcal X$ is nonempty, compact-valued, and Borel measurable. In particular,
it admits a Borel measurable selector. If $Y_{t,\theta,\lambda}^*(x,a;x')$ is a singleton for every
$(x,a,x',\lambda)$, then the unique selector is Borel measurable.
\end{Proposition}

\begin{proof}
By Assumption \ref{assumption : reward functions}, the map
$(x,a,y)\mapsto f(t,x,a,y)$ is continuous. Moreover,
$V_{t+1}^\theta$ is continuous on $\mathcal X$, and
$c(x',y)=\|x'-y\|^q$ is continuous in $(x',y)$. Therefore
$$
(x,a,x',\lambda,y)
\mapsto
f(t,x,a,y)+V_{t+1}^\theta(y)+\lambda c(x',y)
$$
is jointly continuous on $\mathcal X\times A\times\mathcal X\times[0,\Lambda]\times\mathcal X.$
In particular, it is a Carathéodory integrand with parameter
$(x,a,x',\lambda)$ and decision variable $y$. For fixed $(x,a,x',\lambda)$, continuity in $y$ and compactness of
$\mathcal X$ imply that the minimum is attained. Hence
$Y_{t,\theta,\lambda}^*(x,a;x')$ is nonempty and compact.

Since the objective is a Carathéodory integrand, the measurable maximum theorem,
Lemma \ref{Lemma : measurable_maximum_Theorem}, implies that the argmin
correspondence is weakly measurable and admits a measurable selector. If the
argmin is singleton-valued, every selector coincides with the unique minimizer.
Therefore the unique minimizer map is Borel measurable.
\end{proof}

\begin{Proposition}
\label{prop:measurable_dual_selector}
Fix $t\in\mathcal T$ and $\theta\in\Theta$. Under
Assumption \ref{assumption : reward functions}, the correspondence
$$
(x,a)
\mapsto
\argmax_{\lambda\in[0,\Lambda]}F_t^\theta(\lambda;x,a)
$$
is nonempty, compact-valued, and weakly measurable. In particular, it admits a
Borel measurable selector. If the maximizer is unique for every $(x,a)$, then the unique dual selector is Borel measurable.
\end{Proposition}

\begin{proof}
By Lemma \ref{lem:compact_dual_maximizer_set}, the maximization over
$\lambda\ge0$ may be restricted to a compact interval $[0,\Lambda]$. We first show that $(x,a,\lambda)\mapsto F_t^\theta(\lambda;x,a)$ is continuous.
For fixed $(x,a,\lambda)\in\mathcal X\times A\times[0,\Lambda]$,
$$
F_t^\theta(\lambda;x,a)
=
\int_{\mathcal X}
\inf_{y\in\mathcal X}
\left\{
f(t,x,a,y)+V_{t+1}^\theta(y)+\lambda c(x',y)
\right\}
\P_t^0(x,a,\ud x')
-
\varepsilon^q\lambda.
$$
By Assumption \ref{assumption : reward functions}, Theorem \ref{thm : DPP + P^*},
and continuity of $c(x',y)=\|x'-y\|^q$, the map
$$
(x,a,x',\lambda,y)
\mapsto
f(t,x,a,y)+V_{t+1}^\theta(y)+\lambda c(x',y)
$$
is continuous on the compact set $\mathcal X\times A\times\mathcal X\times[0,\Lambda]\times\mathcal X.$
Therefore, by Berge's maximum theorem, 
$$
(x,a,x',\lambda)
\mapsto
\inf_{y\in\mathcal X}
\left\{
f(t,x,a,y)+V_{t+1}^\theta(y)+\lambda c(x',y)
\right\}
$$
is continuous. Since $\P_t^0(x,a,\cdot)$ is weakly continuous by
Assumption \ref{assump:regularity_pi_theta}(i), and since the integrand above is
continuous and bounded on the compact domain, Theorem \ref{thm : Continuity under weak convergence} implies that $(x,a,\lambda)
\mapsto
F_t^\theta(\lambda;x,a)$
is continuous. For each fixed $(x,a)\in\mathcal X\times A$, the map
$\lambda\mapsto F_t^\theta(\lambda;x,a)$ is continuous on the compact interval
$[0,\Lambda]$. Hence the maximum is attained, and the argmax set is nonempty and
compact. Since $F_t^\theta$ is continuous, it is in particular a Carathéodory
integrand with parameter $(x,a)$ and decision variable $\lambda$. The measurable
maximum theorem (see Lemma \ref{Lemma : measurable_maximum_Theorem}) then implies
that $(x,a)
\mapsto
\argmax_{\lambda\in[0,\Lambda]}F_t^\theta(\lambda;x,a)$
is weakly measurable and admits a measurable selector. If the argmax is singleton-valued, then the unique maximizer map is Borel
measurable.
\end{proof}
\begin{Remark}
Measurability ensures that expressions involving selected optimizers are
well-defined. Uniqueness, however, is not required for the one-sided directional
formulas derived later. When the inner or dual optimizer is set-valued, the
derivative is computed over the corresponding active set. Thus set-valued optimizers are compatible with
directional differentiability, although they need not yield a single
vector-valued gradient.
\end{Remark}
\begin{Remark}
Whenever the optimizer is unique, the corresponding selector is continuous, and
uniformly continuous on compact parameter sets by Berge's maximum theorem
\cite[Theorem~17.31]{AliprantisBorder2006}. Continuity of the objective alone gives hemicontinuity of the optimizer
correspondence, but not necessarily a continuous selector. For example, for
$\ell(p,y)=py$ on $[-1,1]\times[-1,1]$,
$$
\argmin_{y\in[-1,1]}\ell(p,y)
=
\begin{cases}
\{-1\}, & p>0,\\
[-1,1], & p=0,\\
\{1\}, & p<0,
\end{cases}
$$
so no continuous selector exists. Without uniqueness, one should keep the
set-valued directional formula or impose an explicit selection rule.
\end{Remark}

\subsection{Policy-gradient representation}\label{sec: policy gradient directional}
We now establish the regularity properties of the robust value functions needed to differentiate the robust performance criterion with respect to the policy parameter. The next step is to quantify how the robust value functions vary with the policy parameter.
Since the dependence on $\theta$ enters both through the policy $\pi^\theta$ and through the
dual robust Bellman operator, we work with directional derivatives in $\theta$.
We recall the notion of directional differentiability that will be used to state our regularity results. 

\begin{Definition}\label{def:one_sided_dir_derivative}
Let $O\subset\R^{d}$ be open, let $f:O\to\R$, and let $ \theta\in\operatorname{int}(O)$. For a direction $r\in\R^{d}$, the right and left directional derivatives of $f$ at $\theta$ in direction $r$ are defined, whenever the limits exist, by
$$D_\theta^+f(\theta)[r]
:=
\lim_{h\downarrow0}\frac{f(\theta+hr)-f(\theta)}{h},
\quad
D_\theta^-f(\theta)[r]
:=
\lim_{h\uparrow0}\frac{f(\theta+hr)-f(\theta)}{h}.$$
When the side is clear from the context, $D_\theta$ denotes the corresponding
one-sided directional derivative. Statements involving $D_\theta$ are to be read separately for the right derivative $D_\theta^+$ and for the left derivative $D_\theta^-$.
\end{Definition}
\subsubsection{Directional sensitivity on compact state-action spaces}
In what follows, all directional derivatives with respect to $\theta$ are taken at points
$\theta\in\operatorname{int}(\Theta)$. If $\theta\in\partial\Theta$, the same
statements are understood only for feasible directions $r$, i.e., directions for
which $\theta+hr\in\Theta$ for all sufficiently small one-sided $h$. The next result uses a value-level stability condition. 
This condition should not be interpreted as a primitive smoothness assumption on the model data. 

In Wasserstein robust dynamic programming, the active dual multiplier and the active transport minimizer may be set-valued, and active sets may switch under perturbations of the policy parameter. 
Thus, without uniqueness or an explicit stability condition (see Section \ref{app:unique}), one cannot expect a classical gradient or a continuous optimizer selector. 
The equicontinuity assumption below is precisely the condition used to pass from pointwise active-set sensitivities to continuous uniform value sensitivities. The value-level condition is automatic in finite state spaces. The dual-level condition, in contrast, involves the continuum of dual variables $[0,\Lambda]$ and is a genuine stability assumption even in finite state spaces: it fails whenever tied transport targets carry heterogeneous value sensitivities at some kink of the piecewise-affine dual objective. Section \ref{subsec:primal} develops a primal envelope route that dispenses with the dual-level condition altogether.

Primitive sufficient conditions can be given in special cases, such as strictly convex inner problems with unique stable selectors, but such conditions are substantially stronger than the assumptions needed for the directional derivative formula and would exclude important nonsmooth examples, including finite-state and q=1 Wasserstein models. We therefore state the main differentiability result under these stability conditions, while treating primitive sufficient conditions as model-specific refinements. Appendix \ref{app:unique} collects primitive convexity and smoothness conditions under which the relevant selectors are unique and stable.

\begin{Proposition}\label{eq:uniform-dir-diffuniform-bound-dq}
Suppose Assumptions \ref{assumption : reward functions} and
\ref{assump:regularity_pi_theta} hold. Fix
$\theta\in\Theta$, $t\in\bar{\mathcal T}$, and a feasible direction
$r\in\mathbb R^{d_\theta}$. Suppose also that there exists $h_0>0$ such that, for every $s\in\bar\Tc$, the
family of one-sided robust value quotients
$$
\bigg\{
x\mapsto
\frac{V_s^{\theta+hr}(x)-V_s^\theta(x)}{h}
:\;0<|h|\le h_0,\ \theta+hr\in\Theta
\Bigg\}
$$
is equicontinuous on $\mathcal X$, and that, for every $s\in\Tc$, the
family of one-sided robust dual quotients
$$
\bigg\{
(x,a,\lambda)\mapsto
\frac{
F_s^{\theta+hr}(\lambda;x,a)-F_s^\theta(\lambda;x,a)
}{h}
:\;0<|h|\le h_0,\ \theta+hr\in\Theta
\bigg\}
$$
is equicontinuous on $\mathcal X\times A\times[0,\Lambda]$. Then the right and left directional
derivatives of $V_t^\theta$ in the direction $r$ exist. Moreover, the
corresponding one-sided derivative satisfies
$D_\theta V_t^\theta(\cdot)[r]\in C(\mathcal X;\mathbb R)$ and
\begin{equation*}
\left\|
\frac{V_t^{\theta+hr}-V_t^\theta}{h}
-
D_\theta V_t^\theta(\cdot)[r]
\right\|_\infty
\longrightarrow 0,
\end{equation*}
as $h\downarrow0$ for the right derivative and as $h\uparrow0$ for the left
derivative. In addition, for each side there exists $h_1>0$ such that
\begin{equation*}
\sup_{0<|h|\le h_1}
\left\|
\frac{V_t^{\theta+hr}-V_t^\theta}{h}
\right\|_\infty
<+\infty,
\end{equation*}
where $h$ is restricted to the corresponding one-sided neighborhood of $0$.
\end{Proposition}
\begin{proof}
    See Appendix \ref{app: proof_diff}.
\end{proof}
Having established continuity and directional differentiability of $V^\theta_0$,
we now differentiate the robust performance map $J(\theta)$. Since $\mu_0$ does
not depend on $\theta$, an immediate application of
Proposition \ref{eq:uniform-dir-diffuniform-bound-dq} and the dominated
convergence Theorem gives, for any $r\in\R^d$,
\begin{align}\label{eq:pg}
    D_\theta J(\theta)[r]
    = \E_{\mu_0}\left[D_\theta V_0^\theta(x)[r]\right].
\end{align}
The content of the directional policy gradient therefore reduces entirely to characterizing
$D_\theta V_0^\theta(x)[r]$ for each $x\in\Xc$ and $r\in\R^d$. This is carried out through
the backward recursion in Corollary \ref{corr: dpp_grad}, whose key ingredient
is differentiating
\begin{equation}\label{eq:G_def2}
G_t^\theta(x,a)
:=\sup_{\lambda\ge 0}\Big(
\E_{X\sim \P_t^0(x,a,\cdot)}\big[-\Fc^{\lambda}(-f(t,x,a,\cdot)-V_{t+1}^\theta(\cdot))(X)\big]
-\varepsilon^q\lambda
\Big),
\end{equation} with respect to $\theta$, for $(t,x,a)\in\Tc\times\Xc\times A$. By Remark \ref{G_dpp}, the quantity $G_t^\theta(x,a)$ can be interpreted as a robust state--action value function. Indeed,
$$
V_t^\theta(x)
=
\E_{a\sim\pi_t^\theta(x,\cdot)}
\big[G_t^\theta(x,a)\big],
$$
so $G_t^\theta(x,a)$ represents the robust value associated with selecting action $a$ in state $x$ at time $t$ under the Wasserstein ambiguity set and the policy $\pi^\theta$. In this sense, $G_t^\theta$ plays the role of a robust $Q$-function. The policy-gradient recursion may therefore be viewed as differentiating a robust $Q$-function with respect to the policy parameters. 

The value-quotient condition in Proposition \ref{eq:uniform-dir-diffuniform-bound-dq} ensures that the
fixed-$\lambda$ directional derivatives of $F_t^\theta(\lambda;x,a)$ exist, while the dual-quotient condition ensures that they converge
uniformly in the dual variable $\lambda$, which is the condition needed to
differentiate the outer supremum over $\lambda$. Section \ref{subsec:primal} gives an alternative primal route to the derivative of $G_t^\theta$ that dispenses with the dual-quotient condition.\\

Since $G_t^\theta$ is defined as a supremum over the dual variable $\lambda$, a Danskin-type argument can be used to differentiate it with respect to $\theta$. This is the content of the following theorem. \begin{Theorem}\label{gradient J and F} Suppose Assumption \ref{assump:regularity_pi_theta} holds. Fix $\theta\in\Theta$, $t\in\Tc$, $(x,a)\in\Xc\times A$, and a feasible direction $r\in\R^d$. Assume that the hypotheses of Proposition \ref{eq:uniform-dir-diffuniform-bound-dq} hold. Then the map $\theta\mapsto G_t^\theta(x,a)$ admits one-sided directional derivatives, and \begin{align} D_\theta^+G_t^\theta(x,a)[r] &= \sup_{\lambda\in\Lambda_t^{*,\theta}(x,a)} \E_{X\sim \P_t^0(x,a,\cdot)} \bigg[ \inf_{y\in Y_{t,\theta,\lambda}^*(x,a;X)} D^+_\theta V_{t+1}^\theta(y)[r] \bigg], \label{eq:grad_G_plus_thm}\\ D_\theta^-G_t^\theta(x,a)[r]
&=
\inf_{\lambda\in\Lambda_t^{*,\theta}(x,a)}
\E_{X\sim \P_t^0(x,a,\cdot)}
\bigg[
\sup_{y\in Y_{t,\theta,\lambda}^*(x,a;X)}
D_\theta^- V_{t+1}^\theta(y)[r]
\bigg], \label{eq:grad_G_minus_thm} \end{align} where \begin{equation}\label{def_selector_lambda} \begin{split} \Lambda_t^{*,\theta}(x,a) &:= \argmax_{\lambda\ge 0} \left\{ \E_{X\sim \P_t^0(x,a,\cdot)} \left[ -\Fc^\lambda \left( -f(t,x,a,\cdot)-V_{t+1}^\theta(\cdot) \right)(X) \right] -\varepsilon^q\lambda \right\}, \\ Y^*_{t,\theta,\lambda}(x,a;x') &:= \argmin_{y\in\Xc} \left\{ f(t,x,a,y)+V_{t+1}^\theta(y)+\lambda c(x',y) \right\}. \end{split} \end{equation} \end{Theorem}
\begin{proof}[Proof of Theorem \ref{gradient J and F}]
We fix $\theta \in \Theta$, $\lambda \geq 0$ and $(x,a,x') \in \mathcal{X} \times A \times \mathcal{X}$. For a perturbation
$H\in C(\mathcal X;\mathbb R)$, define the scalar map
$m_{x,a,x',\lambda}:\mathbb R\to\mathbb R$ by
\begin{align}
    m_{x,a,x',\lambda}(h)
    &:=
    - \mathcal{F}^{\lambda}
    \big(
        - f(t,x,a,\cdot)
        - (V_{t+1}^{\theta}+hH)(\cdot)
    \big)(x') \notag =
    \inf_{y \in \mathcal{X}}
    \Big\{
        f(t,x,a,y)
        + V_{t+1}^{\theta}(y)
        + hH(y)
        + \lambda c(x',y)
    \Big\},
\end{align}
for $h$ in a neighborhood of $0$. Let $y^{\star} \in Y_{t,\theta,\lambda}^{\star}(x,a;x')$
be any minimizer for $h=0$, where
$$Y_{t,\theta,\lambda}^{\star}(x,a;x')
:=
\argmin_{y\in\mathcal X}
\Big\{
f(t,x,a,y)
+
V_{t+1}^{\theta}(y)
+
\lambda c(x',y)
\Big\}.$$
Such a minimizer exists by compactness of $\mathcal X$ and continuity of the
inner objective. Then, for all $h > 0$,
\begin{align}
    m_{x,a,x',\lambda}(h)
    &\leq
    f(t,x,a,y^\star)
    + V_{t+1}^{\theta}(y^\star)
    + hH(y^\star)
    + \lambda c(x',y^\star) \notag =
    m_{x,a,x',\lambda}(0)
    + hH(y^\star).
\end{align}
Hence, we have $\frac{
m_{x,a,x',\lambda}(h)-m_{x,a,x',\lambda}(0)
}{h}
\leq H(y^\star)$. Taking the infimum over $Y_{t,\theta,\lambda}^{\star}(x,a,x')$ yields
\begin{equation}\label{eq : lim_sup_ineq}
    \limsup_{h\downarrow 0}
\frac{
m_{x,a,x',\lambda}(h)-m_{x,a,x',\lambda}(0)
}{h}
\leq
\inf_{y\in Y_{t,\theta,\lambda}^{\star}(x,a;x')}
H(y).  
\end{equation}
Conversely, for each $h > 0$, choose a minimizer $y_h \in \argmin_{y\in\Xc} \lbrace  f(t,x,a,y) + V_{t+1}^{\theta}(y) + h H(y)  + \lambda c(x',y) \big \rbrace$ so that 
$m_{x,a,x',\lambda}(h) = f(t,x,a,y_h) + V_{t+1}^{\theta}(y_h) + h H(y_h) + \lambda c(x',y_h).$ Since $$m_{x,a,x',\lambda}(0) \leq  f(t,x,a,y_h)+ V_{t+1}^{\theta}(y_h) + \lambda c(x',y_h),$$ we obtain $\frac{m_{x,a,x',\lambda}(h)-m_{x,a,x',\lambda}(0)}{h} \geq H(y_h).$
Let $h_n\downarrow0$. By compactness of $\mathcal X$, along a subsequence,
$y_{h_n}\to\bar y\in\mathcal X$. By continuity of
$f(t,x,a,\cdot)$, $V_{t+1}^\theta$, and $c(x',\cdot)$, we have
$\bar y\in Y_{t,\theta,\lambda}^\star(x,a;x')$. Since $H$ is continuous, $\lim_{n\to\infty}H(y_{h_n})=H(\bar y).$ Consequently, 
\begin{align}\label{eq: lim_inf}
\liminf_{n\to+\infty}\frac{m_{x,a,x',\lambda}(h_n)-m_{x,a,x',\lambda}(0)}{h_n}
\ge \lim_{n\to+\infty}H(y_{h_n})
= H(\bar{y})
\ge \inf_{y\in Y^*_{t,\theta,\lambda}(x,a;x')}H(y). \notag 
\end{align}
Since $(h_n)_{n\ge 0}$ was arbitrary,
\begin{equation}\label{eq:liminf-bound-G}
\liminf_{h\to 0}\frac{m_{x,a,x',\lambda}(h)-m_{x,a,x',\lambda}(0)}{h}
\ge \inf_{y\in Y^*_{t,\theta,\lambda}(x,a;x')}H(y).  
\end{equation}
Combining \eqref{eq : lim_sup_ineq} and \eqref{eq:liminf-bound-G},
the directional derivative of $m_{x,a,x',\lambda}$ exists at $0$ and satisfies $m_{+}^{\prime}(0) = \inf_{y\in Y^*_{t,\theta,\lambda}(x,a;x')}H(y).$
We now compute the left derivative. Again fix $y^*\in Y^*$. For $h<0$, $m_{x,a,x',\lambda}(h)\le m_{x,a,x',\lambda}(0)+hH(y^*).$ Dividing by $h<0$ reverses the inequality $\frac{m_{x,a,x',\lambda}(h)-m_{x,a,x',\lambda}(0)}{h}\ge H(y^*).$ Taking the supremum over $y^*\in Y^*$ yields
$$\liminf_{h\uparrow0}\frac{m_{x,a,x',\lambda}(h)-m_{x,a,x',\lambda}(0)}{h}
\ge
\sup_{y\in Y^*_{t,\theta,\lambda}(x,a;x')}H(y).$$
Conversely, if $y_h$ is a minimizer for $m_{x,a,x',\lambda}(h)$ with $h<0$, the same comparison with $m(0)$ gives $m_{x,a,x',\lambda}(h)-m_{x,a,x',\lambda}(0)\ge hH(y_h).$
Dividing by the negative number $h$ gives $\frac{m_{x,a,x',\lambda}(h)-m_{x,a,x',\lambda}(0)}{h}\le H(y_h).$ Along any sequence $h_n\uparrow0$, a subsequence of $y_{h_n}$ converges to some $\bar y\in Y^*$, and therefore
$$\limsup_{h\uparrow0}\frac{m_{x,a,x',\lambda}(h)-m_{x,a,x',\lambda}(0)}{h}
\le
H(\bar y)
\le
\sup_{y\in Y^*_{t,\theta,\lambda}(x,a;x')}H(y).$$
Thus
$m'_-(0)=\sup_{y\in Y^*_{t,\theta,\lambda}(x,a;x')}H(y).$
We now apply these identities to the policy perturbation. We set $H(\cdot):=D_{\theta}V_{t+1}^\theta(\cdot)[r].
$ Set
$$
U_h
:=
-f(t,x,a,\cdot)-V^{\theta+hr}_{t+1}(\cdot),
\quad
\widetilde U_h
:=
-f(t,x,a,\cdot)-(V^\theta_{t+1}+hH)(\cdot),
~~\text{and}~~
U_0
:=
-f(t,x,a,\cdot)-V^\theta_{t+1}(\cdot).
$$
By Proposition \ref{eq:uniform-dir-diffuniform-bound-dq},
$H\in C(\mathcal X;\mathbb R)$ and $V^{\theta+hr}_{t+1}
=
V^\theta_{t+1}
+
hH
+
o(h)$ in $\|\cdot\|_\infty$ as $h\downarrow0.$ Since $f$ does not depend on $\theta$, then
$$
-f(t,x,a,\cdot)-V^{\theta+hr}_{t+1}(\cdot)
=
-f(t,x,a,\cdot)-V^\theta_{t+1}(\cdot)-hH(\cdot)+o(h)
$$
in $\|\cdot\|_\infty$. Hence, using the Lipschitz property of $-\mathcal F^\lambda$, we obtain
$$
\begin{aligned}
&\bigg\|
\frac{-\mathcal F^\lambda(U_h)+\mathcal F^\lambda(U_0)}{h}
-
\frac{-\mathcal F^\lambda(\widetilde U_h)+\mathcal F^\lambda(U_0)}{h}
\bigg\|_\infty\le
\frac{\|U_h-\widetilde U_h\|_\infty}{|h|}
=
\frac{
\|V^{\theta+hr}_{t+1}-(V^\theta_{t+1}+hH)\|_\infty
}{|h|}
\xrightarrow[h\downarrow0]{}0.
\end{aligned}
$$
Therefore, for each $x'\in\mathcal X$,
\begin{equation*}
\begin{split}
&\frac{
-\mathcal F^\lambda(-f(t,x,a,\cdot)-V^{\theta+hr}_{t+1}(\cdot))(x')
+
\mathcal F^\lambda(-f(t,x,a,\cdot)-V^\theta_{t+1}(\cdot))(x')
}{h}\xrightarrow[h\downarrow0]{}
\inf_{y\in Y^*_{t,\theta,\lambda}(x,a;x')}
D_{\theta} V_{t+1}^\theta(y)[r].
\end{split}
\end{equation*}

\noindent By Lemma \ref{lem:compact_dual_maximizer_set}, the supremum over $\lambda\ge0$ may be
restricted to a compact interval $[0,\Lambda]$. Hence,
$$
G_t^\theta(x,a)
=
\sup_{\lambda\ge0}
F_t^\theta(\lambda;x,a)
=
\sup_{\lambda\in[0,\Lambda]}
F_t^\theta(\lambda;x,a).
$$
Moreover, by the uniform quotient bound in
Proposition \ref{eq:uniform-dir-diffuniform-bound-dq} and the $1$-Lipschitz
property of $u\mapsto-\mathcal F^\lambda(-u)$, the above difference quotients are
uniformly bounded for $h$ sufficiently small. Thus dominated convergence gives,
for fixed $\lambda$,
\begin{equation}
\label{eq:directional-psi}
D^+_\theta F^\theta_t(\lambda;x,a)[r]
=
\mathbb E_{X\sim\mathbb P_t^0(x,a,\cdot)}
\bigg[
\inf_{y\in Y_{t,\theta,\lambda}^{\star}(x,a;X)}
D_\theta V_{t+1}^{\theta}(y)[r]
\bigg].
\end{equation}
By the equicontinuity assumption, the quotients $\lambda\mapsto
\frac{
F_t^{\theta+hr}(\lambda;x,a)-F_t^\theta(\lambda;x,a)
}{h}$
are equicontinuous on the compact interval $[0,\Lambda]$. The fixed-$\lambda$
calculation above gives pointwise convergence to the right-hand side of
\eqref{eq:directional-psi}. The pointwise limit is continuous in $\lambda$: equicontinuity is inherited by pointwise limits (fix $\eta>0$, take $\delta>0$ from the equicontinuity of the quotient family uniformly in $h$, and let $h\downarrow0$ in $|q_h(\lambda)-q_h(\lambda')|\le\eta$ for $|\lambda-\lambda'|<\delta$). An equicontinuous family converging pointwise on a compact set converges uniformly; hence the convergence is uniform on $[0,\Lambda]$. Therefore the compact
directional Danskin theorem (see Theorem \ref{def : danskin_Theorem}) applies to $\theta\mapsto
G_t^\theta(x,a)
=
\sup_{\lambda\in[0,\Lambda]}F_t^\theta(\lambda;x,a).$
It yields
$$
D_\theta^+G_t^\theta(x,a)[r]
=
\sup_{\lambda\in\Lambda_t^{*,\theta}(x,a)}
\E_{X\sim \P_t^0(x,a,\cdot)}
\bigg[
\inf_{y\in Y_{t,\theta,\lambda}^*(x,a;X)}
D_\theta^+ V_{t+1}^\theta(y)[r]
\bigg],
$$ 
Using \eqref{eq:directional-psi}, we obtain the stated right-derivative formula.
The left derivative is obtained in the same way, using $h\uparrow0$ and the corresponding left active-set formula. This proves the theorem.
\end{proof}
\begin{Remark}\label{rem:dual-scope}
The dual-quotient equicontinuity in Proposition \ref{eq:uniform-dir-diffuniform-bound-dq} fails whenever, at some $\lambda_0\in(0,\Lambda)$, the transport set $Y^*_{t,\theta,\lambda_0}(x,a;x')$ is multivalued with heterogeneous directional value-derivatives on a set of nominal points $x'$ of positive $\P^0_t(x,a,\cdot)$-mass. In atomic models, $\lambda\mapsto F_t^\theta(\lambda;x,a)$ is piecewise affine and the quotients $q_h(\lambda)=(F_t^{\theta+hr}-F_t^\theta)(\lambda;x,a)/h$ then sweep the full derivative gap over $\lambda$-windows of width $O(h)$ around such a kink, so no modulus of continuity uniform in $h$ exists. Theorem \ref{gradient J and F} should therefore be read as covering the consistent-selector regime, in which all active optimizers carry the same directional value. The general nonsmooth case is handled by Proposition \ref{prop:primal-envelope} below.
\end{Remark}

Next, we formulate the dynamic programming principle for the directional derivative of the value function, which will serve as the foundation for the derivation of our actor--critic algorithm.

\begin{corollary}\label{corr: dpp_grad}
Under the conditions of Theorem \ref{gradient J and F}, or under those of Proposition \ref{prop:primal-envelope} applied at every $t\in\Tc$ (so that the one-sided derivatives $D_\theta G_t^\theta(\cdot,\cdot)[r]$ exist and are uniformly bounded on $\Xc\times A$), for all
$x\in\Xc$,
\begin{align}\label{eq : Dynamic programming gradient for robust value functions}
\begin{cases}
D_\theta V_T^\theta(x)[r]=0,\\[0.2cm]
D_\theta V_t^\theta(x)[r]
=
\mathbb E_{a\sim\pi_t^\theta(x,\cdot)}
\left[
G_t^\theta(x,a)
\left\langle
\nabla_\theta\log\pi_t^\theta(x,a),r
\right\rangle
+
D_\theta G_t^\theta(x,a)[r]
\right],
\qquad t=T-1,\ldots,0.
\end{cases}
\end{align}
\end{corollary}
\begin{proof}
    Recall from Remark \ref{G_dpp} that the DPP defined in \eqref{eq : Dynamic programming for robust value functions} can be written as
$$
V_T^\theta(x)=g(x),~ \text{and}~~
V_t^\theta(x)=\E_{a\sim\pi^\theta_t(x,\cdot)}\Big[G_t^\theta(x,a)\Big],
\quad t=T-1,\dots,0.
$$
Fix $t<T$, $x\in\Xc$ and a direction $r\in\R^d$. By the uniform boundedness of $G_t^\theta$ and
$D_\theta G_t^\theta(\cdot,\cdot)[r]$, the bounded-score assumption in
Assumption \ref{assump:regularity_pi_theta}, and dominated convergence, we may
differentiate under the policy integral. The score-function identity gives
$$
\begin{aligned}
D_\theta V_t^\theta(x)[r]
&=
D_\theta\left[
\int_A G_t^\theta(x,a)\,\pi_t^\theta(x,\d a)
\right][r] \\
&=
\int_A
\Big(
D_\theta G_t^\theta(x,a)[r]\,\pi_t^\theta(x,\d a)
+
G_t^\theta(x,a)\,D_\theta\pi_t^\theta(x,\d a)[r]
\Big)\\
&=
\int_A
\Big(
D_\theta G_t^\theta(x,a)[r]
+
G_t^\theta(x,a)
\left\langle
\nabla_\theta\log\pi_t^\theta(x,a),r
\right\rangle
\Big)
\pi_t^\theta(x,\d a)\\
&=
\E_{a\sim\pi_t^\theta(x,\cdot)}
\Big[
G_t^\theta(x,a)
\left\langle
\nabla_\theta\log\pi_t^\theta(x,a),r
\right\rangle
+
D_\theta G_t^\theta(x,a)[r]
\Big].
\end{aligned}
$$
Also $D_\theta V_T^\theta(x)[r]=0$, since $V_T^\theta(x)=g(x)$ does not depend on $\theta$. Hence, for all $(t,x)\in\{0,\dots,T\}\times\Xc$ and all directions $r\in\R^d$,
$$
D_\theta V_T^\theta(x)[r]=0,~~\text{and}~~D_\theta V_t^\theta(x)[r]
=
\E_{a\sim\pi_t^\theta(x,\cdot)}
\Big[
G_t^\theta(x,a)
\left\langle
\nabla_\theta\log\pi_t^\theta(x,a),r
\right\rangle
+
D_\theta G_t^\theta(x,a)[r]
\Big].
$$
This concludes the proof.
\end{proof}

\subsubsection{A primal envelope formula over worst-case transition laws}\label{subsec:primal}
The dual route of the previous subsection differentiates the scalar envelope over $\lambda$ and requires the dual-quotient stability condition. However, by \eqref{eq : duality result Wasserstein distance}, $G_t^\theta(x,a)$ is itself the infimum of a functional that is linear in the measure and smooth in $\theta$, taken over the $\theta$-independent, weakly compact Wasserstein ball. Danskin's argument can therefore be applied directly in the space of measures. The resulting formula expresses the one-sided derivatives through the set of worst-case transition laws, and remains valid when this set is not a singleton.

\begin{Proposition}\label{prop:primal-envelope}
Fix $t\in\Tc$, $(x,a)\in\Xc\times A$, $\theta\in\operatorname{int}(\Theta)$, and a feasible direction $r\in\R^{d_\theta}$. Assume the following value quotients converge uniformly,
\begin{equation}\label{eq:primal-hyp}
\sup_{y\in\Xc}
\bigg|
\frac{V_{t+1}^{\theta+hr}(y)-V_{t+1}^{\theta}(y)}{h}
-
D^+_\theta V_{t+1}^{\theta}(y)[r]
\bigg|
\xrightarrow[h\downarrow0]{}0,
\end{equation}
with $y\mapsto D^+_\theta V_{t+1}^\theta(y)[r]\in C(\Xc;\R)$ (by Proposition \ref{eq:uniform-dir-diffuniform-bound-dq} at time $t+1$). Write
$$
\Phi_t(\P,\theta)
:=
\int_{\Xc}\big(f(t,x,a,y)+V_{t+1}^\theta(y)\big)\,\P(\d y),
\quad
\Pc^{*}_t(x,a,\theta)
:=
\argmin_{\P\in\mathbb{B}_t^{\varepsilon,q}(\P_t^0(x,a,\cdot))}
\Phi_t(\P,\theta).
$$
 Then the right directional derivative of $\theta\mapsto G_t^\theta(x,a)$ exists and
\begin{equation}\label{eq:primal-right}
D^+_\theta G_t^\theta(x,a)[r]
=
\min_{\P^{*}\in\Pc^{*}_t(x,a,\theta)}
\int_{\Xc}
D^+_\theta V_{t+1}^{\theta}(y)[r]
\,\P^{*}(\d y).
\end{equation}
Under the corresponding left-sided hypothesis ($h\uparrow0$, limit $D^-_\theta V_{t+1}^\theta(\cdot)[r]\in C(\Xc;\R)$),
\begin{equation}\label{eq:primal-left}
D^-_\theta G_t^\theta(x,a)[r]
=
\max_{\P^{*}\in\Pc^{*}_t(x,a,\theta)}
\int_{\Xc}
D^-_\theta V_{t+1}^{\theta}(y)[r]
\,\P^{*}(\d y).
\end{equation}
\end{Proposition}
\begin{proof}
Throughout, write $H:=D^+_\theta V_{t+1}^\theta(\cdot)[r]\in C(\Xc;\R)$.

Since $\Xc$ is compact, weak convergence
$\P_n\rightharpoonup\P$ implies $\int\varphi\,\d\P_n\to\int\varphi\,\d\P$ for
every $\varphi\in C(\Xc)$. As $f(t,x,a,\cdot)+V_{t+1}^\theta\in C(\Xc)$ by
Assumption \ref{assumption : reward functions} and Proposition
\ref{eq:uniform-cont}, the map $\P\mapsto\Phi(\P,\theta)$ is weakly continuous
and linear. The ball $\mathbb{B}_t^{\varepsilon,q}(\P_t^0(x,a,\cdot))$ is weakly compact according to \citep[Theorem~1]{YueKuhnWiesemann2022},
whose hypotheses hold here since $\Xc$ is compact. Hence $\Pc^{*}_t(x,a,\theta)$ is nonempty, and it is weakly closed as a level
set of a continuous function, hence weakly compact.

 Set
$\rho(h):=\sup_{y\in\Xc}\big|(V_{t+1}^{\theta+hr}(y)-V_{t+1}^{\theta}(y))/h-H(y)\big|$,
so that $\rho(h)\to0$ as $h\downarrow0$ by \eqref{eq:primal-hyp}. For every $\P\in\mathbb{B}_t^{\varepsilon,q}(\P_t^0(x,a,\cdot))$ and $h>0$,
\begin{equation}\label{eq:unif-expansion}
\Big|\Phi(\P,\theta+hr)-\Phi(\P,\theta)-h\int_{\Xc}H\,\d\P\Big|
\le
h\,\rho(h),
\end{equation}
uniformly in $\P$. In particular,
$\sup_{\P\in\mathbb{B}_t^{\varepsilon,q}(\P_t^0(x,a,\cdot))}|\Phi(\P,\theta+hr)-\Phi(\P,\theta)|\le h(\|H\|_\infty+\rho(h))\to0$,
whence $G_t^{\theta+hr}(x,a)\to G_t^{\theta}(x,a)$ as $h\downarrow0$.

 Fix $\P^{*}\in\Pc^{*}_t(x,a,\theta)$. Since $G_t^{\theta+hr}(x,a)\le\Phi(\P^{*},\theta+hr)$ and $G_t^{\theta}(x,a)=\Phi(\P^{*},\theta)$, \eqref{eq:unif-expansion} gives, for $h>0$,
$$
\frac{G_t^{\theta+hr}(x,a)-G_t^{\theta}(x,a)}{h}
\le
\int_{\Xc}H\,\d\P^{*}+\rho(h).
$$
Hence $\limsup_{h\downarrow0}$ of the quotient is at most $\int H\,\d\P^{*}$ for every $\P^{*}\in\Pc^{*}_t(x,a,\theta)$, and therefore at most $\min_{\P^{*}\in\Pc^{*}}\int H\,\d\P^{*}$ (the minimum is attained by weak compactness of $\Pc^{*}_t(x,a,\theta)$ and weak continuity of $\P\mapsto\int H\,\d\P$). For each small $h>0$, choose $\P_h\in\argmin_{\P\in\mathbb{B}_t^{\varepsilon,q}(\P_t^0(x,a,\cdot))}\Phi(\P,\theta+hr)$, which is nonempty. Since $\Phi(\P_h,\theta)\ge G_t^\theta(x,a)$, \eqref{eq:unif-expansion} gives
$$
\frac{G_t^{\theta+hr}(x,a)-G_t^{\theta}(x,a)}{h}
=
\frac{\Phi(\P_h,\theta+hr)-G_t^{\theta}(x,a)}{h}
\ge
\int_{\Xc}H\,\d\P_h-\rho(h).
$$
Let $(h_n)_{n\in\N}$ realize the $\liminf$ of the quotient, such that $h_n\downarrow0$. By weak compactness of $\mathbb{B}_t^{\varepsilon,q}(\P_t^0(x,a,\cdot))$, along a subsequence $\P_{h_n}\rightharpoonup\bar\P\in\mathbb{B}_t^{\varepsilon,q}(\P_t^0(x,a,\cdot))$. Hence, using \ref{eq:unif-expansion}, we have $\Phi(\P_{h_n},\theta)
=\Phi(\P_{h_n},\theta+h_nr)+O(h_n)
=G_t^{\theta+h_nr}(x,a)+O(h_n)
\to G_t^{\theta}(x,a)$. By weak continuity of $\Phi(\cdot,\theta)$, $\Phi(\bar\P,\theta)=\lim_n\Phi(\P_{h_n},\theta)=G_t^\theta(x,a)$, i.e., $\bar\P\in\Pc^{*}_t(x,a,\theta)$. Since $H\in C(\Xc)$, $\int H\,\d\P_{h_n}\to\int H\,\d\bar\P$. Therefore
$$
\liminf_{h\downarrow0}
\frac{G_t^{\theta+hr}(x,a)-G_t^{\theta}(x,a)}{h}
\ge
\int_{\Xc}H\,\d\bar\P
\ge
\min_{\P^{*}\in\Pc^{*}_t(x,a,\theta)}\int_{\Xc}H\,\d\P^{*}.
$$
Combining all the previous steps proves \eqref{eq:primal-right}.

For the left derivative, write $H^-:=D^-_\theta V_{t+1}^\theta(\cdot)[r]$ and let $\rho^-(h)\to0$ denote the corresponding uniform error for $h\uparrow0$. For $h<0$ and $\P^{*}\in\Pc^{*}$, $G_t^{\theta+hr}\le\Phi(\P^{*},\theta+hr)=\Phi(\P^{*},\theta)+h\int H^-\d\P^{*}+o(h)$; dividing by $h<0$ reverses the inequality, so $\liminf_{h\uparrow0}$ of the quotient is at least $\int H^-\d\P^{*}$ for every $\P^{*}$, hence at least the maximum. Conversely, with $\P_h$ a minimizer at $\theta+hr$, $G_t^{\theta+hr}-G_t^{\theta}\ge h\int H^-\d\P_h+o(h)$, and dividing by $h<0$ gives that the quotient is at most $\int H^-\d\P_h+o(1)$; the cluster-point argument of Step 3 then bounds the $\limsup_{h\uparrow0}$ by $\int H^-\d\bar\P\le\max_{\Pc^{*}}\int H^-\d\P^{*}$. This proves \eqref{eq:primal-left}.
\end{proof}

\begin{Remark}\label{rem:primal-dual}
\begin{enumerate}
    \item When the dual maximizer is unique with $\lambda_t^{*}(x,a,\theta)>0$ and the transport sets $Y^{*}_{t,\theta,\lambda^{*}}(x,a;x')$ are singletons, strong duality and complementary slackness identify the worst-case law as the pushforward of the nominal law under the transport selector,
$\Pc^{*}_t(x,a,\theta)=\{\,y_t^{*}(x,a,\cdot,\theta)\#\P_t^0(x,a,\cdot)\,\}$
(see \cite{BlanchetMurthy2019,GaoKleywegt2023}). Then \eqref{eq:primal-right}--\eqref{eq:primal-left} reduce to
$D_\theta G_t^\theta(x,a)[r]=\E_{X\sim\P_t^0(x,a,\cdot)}\big[D_\theta V_{t+1}^\theta(y_t^{*}(x,a,X,\theta))[r]\big]$,
which is exactly the formula of Lemma \ref{lem_vector_grad1} and coincides with the dual expressions of Theorem \ref{gradient J and F}.
\item For $q>1$, strict convexity of $c(x',\cdot)$ makes the transport sets singletons whenever $y\mapsto f(t,x,a,y)+V_{t+1}^\theta(y)$ is convex (Corollary \ref{cor:unique_inner_selector}). More generally, the structural results of \cite{GaoKleywegt2023} show that worst-case laws can split the mass of at most one atom of $\P_t^0(x,a,\cdot)$, so that for non-atomic nominal kernels the primal and dual formulas coincide. Note that this alone does not make the dual active-set expressions equal to the primal formulas.
\item The primal formula does not require differentiating the selector $\theta\mapsto\P_t^{*,\theta}$, in line with Remark \ref{untractability_P_comment}. The possible multivaluedness of the worst case is absorbed by the minimum and maximum over $\Pc^{*}_t(x,a,\theta)$. The dual remains the computational device. It is through $(\widehat\lambda_t^{*},y^{*})$ that $G_t^\theta$ and elements of $\Pc^{*}_t(x,a,\theta)$ are evaluated in Algorithm \ref{alg:robust_actor_critic}.
\end{enumerate} 
\end{Remark}

\subsubsection{Directional sensitivity in finite state-action spaces}
When $\mathcal X$ and $A$ are finite, the measurability requirements used above hold trivially, and the value-level quotient condition of Proposition \ref{eq:uniform-dir-diffuniform-bound-dq} is automatic. Every family of functions on a finite set is equicontinuous, and pointwise convergence on a finite set is uniform. For this reason, the finite-state result below is derived from the primal formula of Proposition \ref{prop:primal-envelope}, which requires no stability condition. The worst-case set is then a polytope computable by linear programming.
\begin{Theorem}
    \label{thm: finite_case}
Suppose that $\Xc$ and $A$ are finite, that $\varepsilon>0$, and that the policy class
$\{\pi^\theta:\theta\in\Theta\}$ satisfies Assumption \ref{assump:regularity_pi_theta}. Fix
$\theta\in\operatorname{int}(\Theta)$ and a feasible direction $r$. For $t\in\Tc$ and $(x,a)\in\Xc\times A$,
$$H_\theta(y):=f(t,x,a,y)+V_{t+1}^\theta(y)~~ \text{and} ~~
\Pc^{*}_t(x,a,\theta)
:=
\argmin_{\P\in\mathbb{B}_t^{\varepsilon,q}(\P_t^0(x,a,\cdot))}
\sum_{y\in\Xc}\P(\{y\})\,H_\theta(y).
$$
Then, for every $x\in \Xc$ and every $t\in\bar\Tc$, the one-sided directional derivatives
$D^+_\theta V_t^\theta(x)[r]$ and $D^-_\theta V_t^\theta(x)[r]$ exist, with $D^{\pm}_\theta V_T^\theta(x)[r]=0$ and, for $t=T-1,\ldots,0$,
$$D^{\pm}_\theta V_t^\theta(x)[r]
=
\sum_{a\in A}\pi_t^\theta(x,a)
\left[
G_t^\theta(x,a)
\langle \nabla_\theta\log\pi_t^\theta(x,a),r\rangle
+
D^{\pm}_\theta G_t^\theta(x,a)[r]
\right],$$
where
{\small $$D^+_\theta G_t^\theta(x,a)[r]
=
\min_{\P^{*}\in\Pc^{*}_t(x,a,\theta)}
\sum_{y\in\Xc}\P^{*}(\{y\})\,
D^+_\theta V_{t+1}^\theta(y)[r],
\quad
D^-_\theta G_t^\theta(x,a)[r]
=
\max_{\P^{*}\in\Pc^{*}_t(x,a,\theta)}
\sum_{y\in\Xc}\P^{*}(\{y\})\,
D^-_\theta V_{t+1}^\theta(y)[r].$$}
Moreover, $\Pc^{*}_t(x,a,\theta)$ is a polytope. It is the set of second marginals of the optimal face of the linear program
\begin{equation}\label{eq:finite-LP}
\min_{\gamma\ge0}
\;\sum_{x',y\in\Xc}\gamma(x',y)\,H_\theta(y)
\quad\text{subject to}\quad
\sum_{y\in\Xc}\gamma(x',y)=\P_t^0(x,a,x'),
\quad
\sum_{x',y\in\Xc}\gamma(x',y)\,c(x',y)\le\varepsilon^q,
\end{equation}
where $x'\in\Xc$, $\gamma$ ranges over nonnegative kernels on $\Xc\times\Xc$ and $c(x',y)=\|x'-y\|^q$.
\end{Theorem}
\begin{proof}
The proof is by backward induction on $t$, the case $t=T$ being immediate as $V_T^\theta=g$ does not depend on $\theta$, so $D^{\pm}_\theta V_T^\theta\equiv0$. Assume $D^{\pm}_\theta V_{t+1}^\theta(y)[r]$ exists for every $y\in\Xc$. Because $\Xc$ is finite, the corresponding expansions are uniform in $y$ and the limit maps are continuous, so hypothesis \eqref{eq:primal-hyp} and its left analogue hold at time $t+1$. The ball $\mathbb{B}_t^{\varepsilon,q}(\P_t^0(x,a,\cdot))$ is a closed, hence compact, subset of the simplex $\Pc(\Xc)$, on which weak convergence amounts to convergence of the probability of every state. Proposition \ref{prop:primal-envelope} then yields the displayed formulas for $D^{\pm}_\theta G_t^\theta(x,a)[r]$. Moreover, $\theta\mapsto\pi_t^\theta(x,a)$ is $C^1$ (Assumption \ref{assump:regularity_pi_theta}(ii) applied to $\varphi=\mathbbm 1_{\{a\}}\in C(A)$, $A$ being finite) and $G_t^{\theta+hr}(x,a)\to G_t^\theta(x,a)$ as $h\to0$ (Step 1 of the proof of Proposition \ref{prop:primal-envelope}), so the product rule for one-sided derivatives applies to each summand of $V_t^\theta(x)=\sum_a\pi_t^\theta(x,a)G_t^\theta(x,a)$, with $\nabla_\theta\pi_t^\theta=\pi_t^\theta\nabla_\theta\log\pi_t^\theta$, this gives the displayed recursion for $D^{\pm}_\theta V_t^\theta(x)[r]$ and closes the induction.

 By the definition of $\Wc_q$, with $c(x',y)=\|x'-y\|^q$ as in Proposition \ref{prop:duality},
$\Wc_q(\P,\P_t^0(x,a,\cdot))^q=\min_{\gamma\in\Gamma}\sum_{x',y}\gamma(x',y)c(x',y)$,
the minimum running over the couplings $\gamma\ge0$ with first marginal $\P_t^0(x,a,\cdot)$ and second marginal $\P$; this set is nonempty (it contains the product coupling) and compact, so the minimum is attained. Hence $\P$ lies in the ball if and only if some such $\gamma$ has cost at most $\varepsilon^q$. Let $F$ denote the feasible set of \eqref{eq:finite-LP} and let $L\gamma:=\big(\sum_{x'}\gamma(x',y)\big)_{y}$ be the linear second-marginal map. $L\gamma$ is a probability vector, its total mass being $\sum_{x'}\P_t^0(x,a,x')=1$. Note that $L(F)=\mathbb{B}_t^{\varepsilon,q}(\P_t^0(x,a,\cdot))$, and the objective of \eqref{eq:finite-LP} satisfies $\sum_{x',y}\gamma(x',y)H_\theta(y)=\Phi_t(L\gamma,\theta)$. Consequently the two problems have the same optimal value, $\gamma$ is optimal for \eqref{eq:finite-LP} if and only if $L\gamma\in\Pc^{*}_t(x,a,\theta)$, and $\Pc^{*}_t(x,a,\theta)=L(F^{*})$, where $F^{*}$ is the optimal set of \eqref{eq:finite-LP}. The feasible set of \eqref{eq:finite-LP} is
\begin{align*}
F&:=\Big\{\gamma\in\R_{\ge0}^{\Xc\times\Xc}:\ \ \sum_{y\in\Xc}\gamma(x',y)=\P_t^0(x,a,x')\ \ \forall x'\in\Xc,
\ \ \sum_{x',y\in\Xc}\gamma(x',y)\,c(x',y)\le\varepsilon^q\Big\}\ \\&=\underbrace{\bigcap_{(x',y)\in\Xc\times\Xc}\big\{\gamma:\ \gamma(x',y)\ge0\big\}}_{\R_{\ge0}^{\Xc\times\Xc}}
\ \cap\
\underbrace{\bigcap_{x'\in\Xc}\big\{\gamma:\ \ell_{x'}(\gamma)=\P_t^0(x,a,x')\big\}}_{\text{level sets}}
\ \cap\
\underbrace{\big\{\gamma:\ \ell_c(\gamma)\le\varepsilon^q\big\}}_{\text{sublevel set}}
.
\end{align*}
As the intersection of $\R_{\ge0}^{\Xc\times\Xc}$ with the level and sublevel sets of finitely many
linear, hence continuous, functionals on $\R^{\Xc\times\Xc}$, $F$ is a closed polyhedron. Since $0\le\gamma(x',y)\le\P_t^0(x,a,x')\le1$, $F\subseteq[0,1]^{\Xc\times\Xc}$ is
bounded. Being closed and bounded, $F$ is compact. The
diagonal plan $\gamma^0(x',y):=\P_t^0(x,a,x')\mathbbm 1_{\{y=x'\}}$ lies in
$\R_{\ge0}^{\Xc\times\Xc}$, and has cost
$\sum_{x'}\P_t^0(x,a,x')\,c(x',x')=0\le\varepsilon^q$, since $c(x',x')=0$. Hence, $F$ is non-empty. By the Minkowski--Weyl
theorem \citep[Theorem 19.1]{Rockafellar1970}, the bounded polyhedron $F$ is a polytope and
$F=\mathrm{conv}\{w_1,\dots,w_N\}$ for its finite set of extreme points.

The objective $\ell(\gamma):=\sum_{x',y}\gamma(x',y)H_\theta(y)$ is linear, hence continuous, and
therefore attains its minimum $v^{*}$ on the nonempty compact $F$. Let
$W^{*}:=\{w_i:\ell(w_i)=v^{*}\}$ and $F^{*}:=\{\gamma\in F:\ell(\gamma)=v^{*}\}$. Then
$F^{*}=\mathrm{conv}(W^{*})$. The inclusion $\supseteq$ holds because $F^{*}$ is convex and contains
$W^{*}$. Conversely, any $\gamma\in F^{*}$ writes $\gamma=\sum_i\alpha_iw_i$ with $\alpha_i\ge0$,
$\sum_i\alpha_i=1$, hence $v^{*}=\sum_i\alpha_i\,\ell(w_i)$ with $\ell(w_i)\ge v^{*}$ for all $i$,
forcing $\ell(w_i)=v^{*}$ whenever $\alpha_i>0$, i.e.\ $\gamma\in\mathrm{conv}(W^{*})$. In
particular $W^{*}\ne\emptyset$ (the minimum is attained at an extreme point) and $F^{*}$ is a
nonempty polytope. It is a face of $F$, namely its intersection with the hyperplane
$\{\ell=v^{*}\}$, which supports $F$ since $F\subseteq\{\ell\ge v^{*}\}$. For a linear map $L$ and a finite set $S$, $L(\mathrm{conv}(S))=\mathrm{conv}(L(S))$. Hence, 
$$
\Pc^{*}_t(x,a,\theta)=L(F^{*})=L(\mathrm{conv}(W^{*}))=\mathrm{conv}(\{Lw:\ w\in W^{*}\}),
$$
the convex hull of finitely many probability vectors, i.e, a polytope.
\end{proof}

\begin{Remark}
 Finiteness does not remove possible nonsmoothness when the inner minimizer or
the dual maximizer is not unique. In that case, the natural object is the
one-sided directional derivative of Theorem \ref{gradient J and F}. A classical
gradient formula is recovered only if the active inner minimizer
$Y_{t,\theta,\lambda}^*(x,a;x')$ and the active dual maximizer
$\Lambda_t^{*,\theta}(x,a)$ are singletons, or if all active minimizers and
maximizers give the same directional value.
\end{Remark}
\subsubsection{A vector-valued gradient representation}\label{sec: vector valued}
The directional derivative obtained in the previous section involves the dual
optimizer set and the inner minimizer set. In general, these sets may contain several
elements, so the resulting expression need not be linear in the direction $r$.
If the active dual and transport selectors are singletons, the active-set max--min expressions collapse to evaluation at the unique selectors. The resulting recursion is linear in $r$, so the Riesz representation in finite-dimensional parameter space yields the vector-valued gradient and the max--min structure in the
directional Danskin formula disappears. This is made precise in the lemma below.
\begin{Lemma}\label{lem_vector_grad1} 
Suppose the hypotheses of Theorem \ref{gradient J and F} hold for every stage and direction. Assume additionally that, for all $t\in\mathcal T$, $x\in\mathcal X$, $a\in A$,
and $\theta\in\Theta$, the dual optimizer is unique, $\Lambda_t^{\star,\theta}(x,a)=\{\lambda_t^\star(x,a,\theta)\},$
and that, for every $x'\in\mathcal X$, the inner optimizer is unique,
$Y_{t,\theta,\lambda_t^\star}^{\star}(x,a;x')
=
\{y_t^\star(x,a,x',\theta)\}.$
Then $D_\theta V_t^\theta(x)[r]$ is linear in $r$, i.e.,
$
D_\theta V_t^\theta(x)[r]
=
\langle
\nabla_\theta V_t^\theta(x),
r
\rangle.
$
Moreover, we have
$$
\nabla_\theta G_t^\theta(x,a)
=
\mathbb E_{X\sim P_t^0(x,a,\cdot)}
\left[
\nabla_\theta V_{t+1}^\theta
\left(y_t^\star(x,a,X,\theta)\right)
\right].
$$
\end{Lemma}
\begin{proof}
    See Appendix \ref{app: gradient valued}.
\end{proof}
As a direct result, the DPP formulated in Corollary \ref{corr: dpp_grad} satisfies the following.
\begin{corollary}
Assume that, for every $t\in\mathcal T$, $x\in\mathcal X$, $a\in A$,
$\theta\in\Theta$, and $X'\in\mathcal X$, the optimizer sets are singletons. Under the Assumption \ref{assump:regularity_pi_theta}, for all $(t,x) \in \Tc\times\Xc$,
\begin{align}
\begin{cases}
\nabla_\theta V_T^\theta(x)=0,\\[0.2cm]
\nabla_\theta V_t^\theta(x)
=
\mathbb E_{a\sim\pi_t^\theta(x,\cdot)}
\Big[
G_t^\theta(x,a)
\nabla_\theta\log\pi_t^\theta(x,a)
+
\nabla_\theta G_t^\theta(x,a)
\Big].
\end{cases}
\end{align}
\end{corollary}

Next, we present the classical likelihood-ratio policy gradient, with the robust state--action value as weight, evaluated under the frozen current worst-case kernel.
\begin{corollary}\label{cor:unrolled}
Under the hypotheses of Lemma \ref{lem_vector_grad1}, let $\P^{*,\theta}$ denote the worst-case path measure with initial distribution $\mu_0$ introduced before \eqref{eq : J(theta)}. Then
$$
\nabla_\theta J(\theta)
=
\E^{\P^{*,\theta}}\Big[\sum_{t=0}^{T-1}G_t^\theta(X_t,a_t)\,\nabla_\theta\log\pi_t^\theta(X_t,a_t)\Big].
$$
\end{corollary}
\begin{proof}
Under the uniqueness assumptions of Lemma \ref{lem_vector_grad1}, the worst-case transition law at $(t,x,a)$ is the pushforward
$\P_t^{*,\theta}(x,a,\cdot)=y_t^\star(x,a,\cdot,\theta){\#}\P_t^0(x,a,\cdot)$. The vector recursion reads $\nabla_\theta V_t^\theta(x)=\E_{a\sim\pi_t^\theta(x,\cdot)}\big[G_t^\theta(x,a)\nabla_\theta\log\pi_t^\theta(x,a)+\E_{Y\sim\P_t^{*,\theta}(x,a,\cdot)}[\nabla_\theta V_{t+1}^\theta(Y)]\big]$. Iterating backward from $\nabla_\theta V_T^\theta=0$ expresses $\nabla_\theta V_0^\theta(x)$ as the sum over $t$ of the score terms averaged along the flow generated by $(\pi^\theta,\P^{*,\theta})$ started at $x$, i.e.\ the displayed expectation conditional on $X_0=x$. Integrating over $X_0\sim\mu_0$ and using $\nabla_\theta J(\theta)=\E_{\mu_0}[\nabla_\theta V_0^\theta(X_0)]$ gives the claim.
\end{proof}

Lemma \ref{lem_vector_grad1} requires uniqueness of two selectors. The first one is
the inner minimizer
$
Y^*_{t,\theta,\lambda}(x,a;x').
$
This uniqueness is easy to obtain under standard convexity assumptions as discussed in Section \ref{app:unique}. The dual maximizer may be set-valued.

 If the inner minimizer is unique and the
envelope theorem applies, then
$$
\partial_\lambda F_t^\theta(\lambda;x,a)
=
R_t^\theta(x,a,\lambda)-\varepsilon^q,
$$
where
$$
R_t^\theta(x,a,\lambda)
:=
\E_{X\sim\P_t^0(x,a,\cdot)}
\Big[
\|X-y_t^*(x,a,X,\lambda,\theta)\|^q
\Big].
$$
Thus, if $\lambda\mapsto R_t^\theta(x,a,\lambda)$ is continuous, strictly
decreasing on $(0,\Lambda)$, and satisfies
$$
R_t^\theta(x,a,0^+)>\varepsilon^q>
R_t^\theta(x,a,\Lambda^-),
$$
then there is a unique
$
\lambda_t^*(x,a,\theta)\in(0,\Lambda)
$
such that
$
R_t^\theta(x,a,\lambda_t^*)=\varepsilon^q.
$
Since $F_t^\theta(x,a,\cdot)$ is concave, this point is the unique maximizer.\\

\noindent We can check this setting in smooth quadratic models. For example, when $q=2$,
set
$
H(y):=f(t,x,a,y)+V_{t+1}^\theta(y).
$
If $H$ is $C^2$ and convex, and if the minimizer
$
y_\lambda(x'):=y_t^*(x,a,x',\lambda,\theta)
$
is interior, then
$$
\nabla H(y_\lambda)+2\lambda(y_\lambda-x')=0.
$$
Differentiating in $\lambda$ gives
$$
\partial_\lambda y_\lambda
=
2\big(\nabla^2 H(y_\lambda)+2\lambda I\big)^{-1}(x'-y_\lambda),
$$
and therefore
$$
\frac{d}{d\lambda}\|x'-y_\lambda\|^2
=
-4
(x'-y_\lambda)^\top
\big(\nabla^2 H(y_\lambda)+2\lambda I\big)^{-1}
(x'-y_\lambda)
\le 0.
$$
The inequality is strict when $x'\ne y_\lambda$. Hence, the map
$
\lambda\mapsto R_t^\theta(x,a,\lambda)
$
is strictly decreasing. In the Linear Quadratic (LQ) case (see \ref{sec: Numerical experiments} for example), $H$ is quadratic, so the
first-order condition is linear and the scalar dual problem can be checked
directly. The gradient-valued recursion of Lemma \ref{lem_vector_grad1} is then
appropriate when the dual maximizer is unique and interior. 

\subsection{Convergence of an exact tabular ascent scheme}\label{subsec:convergence}
This subsection builds an ascent scheme for $J$ and proves its convergence. The scheme follows the classical pattern of active-set methods with a shrinking tolerance \citep[Sections~1.2 and~2.4]{Polak1997}.

Throughout, $\Xc$ and $A$ are finite, $\varepsilon>0$, $\zeta\ge0$, and $\Theta=[-B,B]^{d_\theta}$. Each $\pi_t^\theta(x,a)$ is strictly positive and $C^2$ near $\Theta$. Critics, transport programs and dual problems are exact. The $\zeta$-active feasible directions at $\theta$ are
$$
T_\zeta(\theta):=\big\{r\in\R^{d_\theta}:\ \|r\|_\infty\le1,\ r_i\le0\ \text{if}\ \theta_i\ge B-\zeta,\ r_i\ge0\ \text{if}\ \theta_i\le-B+\zeta\big\},
$$
so that $\theta+hr\in\Theta$ for $r\in T_\zeta(\theta)$ and $h\in[0,\zeta]$. By Theorem \ref{thm: finite_case}, for each $(t,x,a)\in\Tc\times\Xc\times A$ the ball $\mathbb{B}_t^{\varepsilon,q}(\P_t^0(x,a,\cdot))\subseteq\Pc(\Xc)$ is a polytope. Let $\mathcal V_t(x,a)\subseteq\mathbb{B}_t^{\varepsilon,q}(\P_t^0(x,a,\cdot))$ denote its vertex set, which is finite and does not depend on $\theta$. A selection is a map
$
\sigma:\Tc\times\Xc\times A\to\Pc(\Xc),$ such that $
\sigma(t,x,a)\in\mathcal V_t(x,a),$ for all $(t,x,a)\in\Tc\times\Xc\times A,$ and the set of selections, $$\Sigma:=\prod_{(t,x,a)\in\Tc\times\Xc\times A}\mathcal V_t(x,a),$$ is finite. Each $\sigma\in\Sigma$ defines a non-robust MDP with the fixed transition kernels $\P_t(x,a,\cdot):=\sigma(t,x,a)$, and we write $V^{\sigma,\theta}_t:\Xc\to\R$ and $Q^{\sigma,\theta}_t:\Xc\times A\to\R$ for the value functions of the policy $\pi^\theta$ in this MDP, for $t\in\Tc$, and $J_\sigma:\Theta\to\R$, $J_\sigma(\theta):=\E_{\mu_0}[V^{\sigma,\theta}_0(X_0)]$, for its performance. With $\langle \P,\phi\rangle:=\sum_{y\in\Xc}\P(\{y\})\phi(y)$ for $\P\in\Pc(\Xc)$ and $\phi:\Xc\to\R$, define the gaps $g^{\sigma,\theta}_t:\Xc\times A\to[0,\infty)$ and, for $\zeta\ge0$, the $\zeta$-active selections $S_\zeta(\theta)\subseteq\Sigma$ by
{\small $$
g^{\sigma,\theta}_t(x,a):=\big\langle\sigma(t,x,a),\,f(t,x,a,\cdot)+V^{\theta}_{t+1}\big\rangle-G_t^\theta(x,a),
\quad
S_\zeta(\theta):=\big\{\sigma\in\Sigma:\ g^{\sigma,\theta}_t(x,a)\le\zeta,~ \forall(t,x,a)\in\Tc\times\Xc\times A\big\}.
$$}
Writing $\mathcal V^{\zeta,\theta}_t(x,a):=\{\P\in\mathcal V_t(x,a):\langle \P,f(t,x,a,\cdot)+V^\theta_{t+1}\rangle\le G^\theta_t(x,a)+\zeta\}$ for the $\zeta$-optimal vertices of the transport program at $(t,x,a)$, one has $S_\zeta(\theta)=\prod_{(t,x,a)\in\Tc\times\Xc\times A}\mathcal V^{\zeta,\theta}_t(x,a)$.

\begin{Theorem}\label{lem:min-structure} The following results hold:
\begin{itemize}
\item[(i)] Each $J_\sigma$ is $C^2$ near $\Theta$, $J=\min_{\sigma\in\Sigma}J_\sigma$ on $\Theta$, $S_0(\theta)\ne\emptyset$, and $S_0(\theta)\subseteq\Sigma^{*}(\theta):=\{\sigma\in\Sigma:J_\sigma(\theta)=J(\theta)\}$.
\item[(ii)] For every direction $r\in\R^{d_\theta}$ with $\theta+hr\in\Theta$ for all small $h>0$,
$$
D_\theta^{+}J(\theta)[r]=\min_{\sigma\in\Sigma^{*}(\theta)}\big\langle\nabla_\theta J_\sigma(\theta),r\big\rangle,
\quad
D_\theta^{-}J(\theta)[r]=\max_{\sigma\in\Sigma^{*}(\theta)}\big\langle\nabla_\theta J_\sigma(\theta),r\big\rangle,
$$
the second formula holding for directions feasible on both sides, consistently with Theorem \ref{thm: finite_case}. Here $\nabla_\theta J_\sigma(\theta)=\E^{\pi^\theta,\sigma}\big[\sum_{t=0}^{T-1}Q^{\sigma,\theta}_t(X_t,a_t)\nabla_\theta\log\pi_t^\theta(X_t,a_t)\big]$ is the classical policy gradient of the $\sigma$-MDP.
\item[(iii)] $J_\sigma(\theta)-J(\theta)=\E^{\pi^\theta,\sigma}\big[\sum_{t=0}^{T-1}g^{\sigma,\theta}_t(X_t,a_t)\big]$ for every $\sigma\in\Sigma$ and $\theta\in\Theta$, under the path law of $(\mu_0,\pi^\theta,\sigma)$. In particular $J_\sigma(\theta)\le J(\theta)+T\zeta$ for every $\sigma\in S_\zeta(\theta)$.
\item[(iv)] There are $C_0,C_1,C_2$, depending only on the model data and $B$, such that all $J_\sigma$ and $J$ are $C_0$-Lipschitz on $\Theta$ in the norm $\|\cdot\|_\infty$, $|g^{\sigma,\theta}_t(x,a)-g^{\sigma,\theta'}_t(x,a)|\le C_1\|\theta-\theta'\|_\infty$ for all $\sigma\in\Sigma$, $(t,x,a)\in\Tc\times\Xc\times A$ and $\theta,\theta'\in\Theta$, and $|J_\sigma(\theta+hr)-J_\sigma(\theta)-h\langle\nabla_\theta J_\sigma(\theta),r\rangle|\le C_2h^2$ for all $\sigma\in\Sigma$, $\|r\|_\infty\le1$ and $\theta,\theta+hr\in\Theta$.
\end{itemize}
\end{Theorem}
\begin{proof}
(i) $J_\sigma(\theta)=\sum_{(x_0,a_0,\dots,x_{T})}\mu_0(x_0)\prod_{t=0}^{T-1}\pi_t^\theta(x_t,a_t)\,\sigma(t,x_t,a_t)(\{x_{t+1}\})\Big[\textstyle\sum_{t=0}^{T-1}f(t,x_t,a_t,x_{t+1})+g(x_T)\Big],$ is a finite sum and $\pi_t^\theta(x_t,a_t)$ is $C^2$. Hence $J_\sigma\in C^2$. Backward induction gives $V_t^\theta\le V_t^{\sigma,\theta}$. Indeed $G_t^\theta(x,a)\le\langle\sigma(t,x,a),f+V^\theta_{t+1}\rangle\le\langle\sigma(t,x,a),f+V^{\sigma,\theta}_{t+1}\rangle$, and averaging over $\pi^\theta$ preserves the inequality, so $J\le\min_{\sigma\in\Sigma}J_\sigma$. Conversely, \eqref{eq:finite-LP} reaches its minimum at a vertex, so some $\sigma^{*}\in\Sigma$ has $g^{\sigma^{*},\theta}\equiv0$, and backward induction gives $V^{\sigma^{*},\theta}=V^\theta$. Hence $J_{\sigma^{*}}(\theta)=J(\theta)$, $S_0(\theta)\ne\emptyset$, and $S_0(\theta)\subseteq\Sigma^{*}(\theta)$.

(ii) For $\sigma\in\Sigma^{*}(\theta)$, $J(\theta+hr)\le J_\sigma(\theta+hr)=J(\theta)+h\langle\nabla J_\sigma(\theta),r\rangle+o(h)$, so the upper limit of the quotient as $h\downarrow0$ is at most the stated minimum. For the lower bound, set $\ell:=\liminf_{h\downarrow0}\big(J(\theta+hr)-J(\theta)\big)/h$ and choose a sequence $(h_n)_{n\ge0}$ with $h_n\downarrow0$ along which the quotient converges to $\ell$. For each $n\ge0$, (i) applied at $\theta+h_nr$ provides $\sigma_n\in\Sigma$ with $J_{\sigma_n}(\theta+h_nr)=J(\theta+h_nr)$. Since $\Sigma$ is finite, some $\sigma\in\Sigma$ satisfies $\sigma_n=\sigma$ for infinitely many $n\ge0$, and we pass to the corresponding subsequence, still denoted $(h_n)_{n\ge0}$, so that
$J_\sigma(\theta+h_nr)=J(\theta+h_nr)$ for all $n\ge0.$ Both $J_\sigma$ and $J$ are continuous on $\Theta$, the first being $C^2$ and the second a finite minimum of continuous functions, so letting $n\to\infty$ in this identity gives $J_\sigma(\theta)=J(\theta)$, that is $\sigma\in\Sigma^{*}(\theta)$. Along the subsequence the two difference quotients therefore coincide, and differentiability of $J_\sigma$ gives
$$
\frac{J(\theta+h_nr)-J(\theta)}{h_n}=\frac{J_\sigma(\theta+h_nr)-J_\sigma(\theta)}{h_n}\ \xrightarrow[n\to\infty]{}\ \big\langle\nabla_\theta J_\sigma(\theta),r\big\rangle.
$$
Hence $\ell=\langle\nabla_\theta J_\sigma(\theta),r\rangle\ge\min_{\sigma'\in\Sigma^{*}(\theta)}\langle\nabla_\theta J_{\sigma'}(\theta),r\rangle$. Together with the upper bound on the upper limit, the quotient converges and $D^{+}_\theta J(\theta)[r]$ equals the stated minimum. The left formula follows from the same argument applied to $h\uparrow0$, with minima replaced by maxima.

(iii) With $\Delta_t:=V^{\sigma,\theta}_t-V^\theta_t$ for $t\in\Tc$ and $\Delta_T\equiv0$, $\Delta_t(x)=\E_{a\sim\pi_t^\theta(x,\cdot)}\big[g^{\sigma,\theta}_t(x,a)+\langle\sigma(t,x,a),\Delta_{t+1}\rangle\big]$, which telescopes along the $(\pi^\theta,\sigma)$-flow.

(iv) Finitely many $C^2$ functions on the compact $\Theta$, over the finite $\Sigma$, give $C_0$ and $C_2$. The gap $g^{\sigma,\theta}_t$ depends on $\theta$ through $V^\theta_{t+1}$ and $G^\theta_t$, which are finite minima of $C_0$-Lipschitz functions, and this gives $C_1$.
\end{proof}

\begin{algorithm}[H]
\caption{Exact robust policy ascent}
\label{alg:tabular_ascent}
\footnotesize
\begin{algorithmic}[1]
\Require $\theta_0\in\Theta$, $\zeta_0>0$, $\nu\in(0,1)$, and $C_1,C_2$ of Theorem \ref{lem:min-structure}(iv)
\State \textbf{for} $k=0,1,2,\dots$ \textbf{do}
\State \quad Run exact robust DP at $\theta_k$, computing $V^{\theta_k}_t$, $G^{\theta_k}_t$, the $\zeta_k$-optimal vertex sets of \eqref{eq:finite-LP}, and hence $S_{\zeta_k}(\theta_k)$
\State \quad Solve the linear program $\Delta_k:=\displaystyle\max_{r\in T_{\zeta_k}(\theta_k)}\ \min_{\sigma\in S_{\zeta_k}(\theta_k)}\big\langle\nabla_\theta J_\sigma(\theta_k),r\big\rangle$ and pick a maximizer $r_k$
\State \quad \textbf{if} $\Delta_k\le\zeta_k$ \textbf{then} $\theta_{k+1}\leftarrow\theta_k$ and $\zeta_{k+1}\leftarrow\nu\zeta_k$ \Comment{refinement}
\State \quad \textbf{else} $\theta_{k+1}\leftarrow\theta_k+h_kr_k$ with $h_k:=\min\{\zeta_k,\ \zeta_k/C_1,\ \Delta_k/(2C_2)\}$ and $\zeta_{k+1}\leftarrow\zeta_k$ \Comment{ascent}
\end{algorithmic}
\end{algorithm}

\begin{Lemma}\label{lem:increase}
On an ascent step, $\theta_{k+1}\in\Theta$ and $J(\theta_{k+1})\ge J(\theta_k)+\tfrac12h_k\Delta_k>J(\theta_k)+\tfrac12h_k\zeta_k$.
\end{Lemma}
\begin{proof}
Feasibility holds since $r_k\in T_{\zeta_k}(\theta_k)$ and $h_k\le\zeta_k$. Write $\theta,h,r$ for $\theta_k,h_k,r_k$. Choose $\sigma_h\in\Sigma$ with $J_{\sigma_h}(\theta+hr)=J(\theta+hr)$ and $\sigma^\star\in S_0(\theta)$, both of which exist by Theorem \ref{lem:min-structure}(i), write $\P^{h}$ for the path law of $(\mu_0,\pi^{\theta+hr},\sigma_h)$, and call $(t,x)\in\Tc\times\Xc$ reached if $\P^{h}(X_t=x)>0$. Since $\pi^{\theta+hr}_t(x,a)>0$ for all $(t,x,a)\in\Tc\times\Xc\times A$, one has $\P^{h}(X_t=x,a_t=a)>0$ for every reached $(t,x)$ and every $a\in A$. Theorem \ref{lem:min-structure}(iii) applied at $\theta+hr$ gives $\E^{\P^{h}}\big[\sum_{t=0}^{T-1}g^{\sigma_h,\theta+hr}_t(X_t,a_t)\big]=J_{\sigma_h}(\theta+hr)-J(\theta+hr)=0$, and the summands being nonnegative, $g^{\sigma_h,\theta+hr}_t(x,a)=0$ whenever $\P^{h}(X_t=x,a_t=a)>0$, that is at every reached triple. Define 
$$
\sigma'(t,x,a):=
\begin{cases}
\sigma_h(t,x,a) & \text{if }(t,x)\text{ is reached},\\[2pt]
\sigma^\star(t,x,a) & \text{otherwise}.
\end{cases}\in\Sigma.
$$
The marginals $\P^{h}(X_t=\cdot)$ are determined recursively by the kernels at reached triples only, so the path law of $(\mu_0,\pi^{\theta+hr},\sigma')$ equals $\P^{h}$ and $J_{\sigma'}(\theta+hr)=J(\theta+hr)$. Moreover $\sigma'\in S_{\zeta_k}(\theta)$, since $g^{\sigma',\theta}_t(x,a)=g^{\sigma^\star,\theta}_t(x,a)=0$ at unreached triples, while at reached triples (iv) gives $g^{\sigma',\theta}_t(x,a)\le g^{\sigma',\theta+hr}_t(x,a)+C_1h=C_1h\le\zeta_k$. Hence, by  (iv) and $J_{\sigma'}(\theta)\ge J(\theta)$,
$$
J(\theta+hr)=J_{\sigma'}(\theta+hr)\ \ge\ J_{\sigma'}(\theta)+h\big\langle\nabla_\theta J_{\sigma'}(\theta),r\big\rangle-C_2h^2\ \ge\ J(\theta)+h\Delta_k-C_2h^2,
$$
where the last step uses $\sigma'\in S_{\zeta_k}(\theta)$ and the definition of $\Delta_k$ as a max--min. Finally $C_2h\le\Delta_k/2$ by the choice of $h_k$, and $\Delta_k>\zeta_k$ on ascent steps.

\end{proof}

\begin{Definition}\label{def:stationarity}
A point $\bar\theta\in\Theta$ is directionally stationary if $D_\theta^{+}J(\bar\theta)[r]\le0$ for all $r\in T_0(\bar\theta)$. For $\zeta\ge0$, a point $\theta\in\Theta$ is $\zeta$-stationary if for every $r\in T_\zeta(\theta)$ there exists $\sigma\in S_\zeta(\theta)$ with $\langle\nabla_\theta J_\sigma(\theta),r\rangle\le\zeta$.
\end{Definition}
\begin{Theorem}\label{thm:convergence}
Let $(\theta_k,\zeta_k)_{k\ge0}$ be generated by Algorithm \ref{alg:tabular_ascent} and let $\{k\ge0:\ \Delta_k\le\zeta_k\}$ be the set of refinement iterations. $(J(\theta_k))_{k\ge0}$ is nondecreasing and convergent, refinements are infinitely many, $\zeta_k\to0$ as $k\to\infty$, and every accumulation point $\bar\theta$ of the refinement iterates is directionally stationary.
\end{Theorem}
\begin{proof}
$J$ is bounded on $\Theta$, and each ascent step at tolerance $\zeta$ gains at least $\tfrac12\min\{\zeta,\zeta/C_1,\zeta/(2C_2)\}\zeta>0$ by Lemma \ref{lem:increase}. Each level therefore holds finitely many ascent steps, refinements are infinite, and $\zeta_k\to0$.

Let $(\theta_{k_j})_{j\ge0}$ be refinement iterates with $\theta_{k_j}\to\bar\theta$ as $j\to\infty$, so that $\Delta_{k_j}\le\zeta_{k_j}$ for all $j\ge0$, and fix $r\in T_0(\bar\theta)$. Two facts hold for all large $j$. First, $r\in T_{\zeta_{k_j}}(\theta_{k_j})$: if $\theta_{k_j,i}\ge B-\zeta_{k_j}$ along a subsequence then $\bar\theta_i=B$ and $r_i\le0$ as required, while coordinates with $\bar\theta_i<B$ are eventually unconstrained, the lower bound being symmetric. Second, $S_{\zeta_{k_j}}(\theta_{k_j})\subseteq\Sigma^{*}(\bar\theta)$: any $\sigma\notin\Sigma^{*}(\bar\theta)$ has some gap $g^{\sigma,\bar\theta}_t(x,a)\ge2\delta>0$, for otherwise $\sigma\in S_0(\bar\theta)\subseteq\Sigma^{*}(\bar\theta)$ by (i), so $g^{\sigma,\theta_{k_j}}_t(x,a)\ge2\delta-C_1\|\theta_{k_j}-\bar\theta\|_\infty>\zeta_{k_j}$ eventually, and the claim follows from finiteness of $\Sigma$. Combining,
$$
\zeta_{k_j}\ \ge\ \Delta_{k_j}\ \ge\ \min_{\sigma\in S_{\zeta_{k_j}}(\theta_{k_j})}\big\langle\nabla_\theta J_\sigma(\theta_{k_j}),r\big\rangle\ \ge\ \min_{\sigma\in\Sigma^{*}(\bar\theta)}\big\langle\nabla_\theta J_\sigma(\theta_{k_j}),r\big\rangle\ \xrightarrow[j\to\infty]{}\ D_\theta^{+}J(\bar\theta)[r].
$$
The second inequality holds because $r$ is feasible for the direction problem by the first fact, the third because a minimum over a subset dominates the minimum over the superset by the second fact, and the limit follows from continuity of the finitely many gradients and (ii).
\end{proof}

\begin{corollary}\label{cor:complexity}
Let $\kappa:=\min\{1,1/C_1,1/(2C_2)\}$ and $B_J:=\sup_{\Theta}J-J(\theta_0)$.
\begin{itemize}
\item[(i)] Every accumulation point of $(\theta_k)_{k\ge0}$ has objective value $J_\infty:=\lim_{k\to\infty}J(\theta_k)$.
\item[(ii)] Each ascent step at tolerance $\zeta$ gains at least $(\kappa/2)\zeta^{2}$. Hence, for every $i\ge0$, at most $2B_J/(\kappa\nu^{2i}\zeta_0^{2})$ ascent steps occur at level $\nu^{i}\zeta_0$, and the refinement ending this level occurs by iteration $(i+1)+\tfrac{2B_J}{\kappa(1-\nu^{2})}(\nu^{i}\zeta_0)^{-2}$.
\item[(iii)] At every refinement iteration $k$, the iterate $\theta_k$ is $\zeta_k$-stationary. For every $\zeta\in(0,\zeta_0]$, a $\zeta$-stationary iterate is produced within $\lceil\log(\zeta_0/\zeta)/\log(1/\nu)\rceil+1+\tfrac{2B_J}{\kappa(1-\nu^{2})\nu^{2}\zeta^{2}}$ iterations.
\end{itemize}
\end{corollary}
\begin{proof}
(i) $J(\theta_k)\uparrow J_\infty$ as $k\to\infty$ and $J$ is continuous. (ii) On such a step $\Delta_k>\zeta$ and $h_k\ge\kappa\zeta$, so Lemma \ref{lem:increase} gives the gain. With $A_l$ the number of ascent steps at level $l\ge0$, the bound $\sum_{l\ge0}A_l(\kappa/2)(\nu^l\zeta_0)^2\le B_J$ gives $A_i\le2B_J/(\kappa\nu^{2i}\zeta_0^{2})$ and $\sum_{l=0}^{i}A_l\le\tfrac{2B_J}{\kappa(1-\nu^{2})(\nu^{i}\zeta_0)^{2}}$, and the refinement ending level $i$ follows these ascent steps and the $i+1$ earlier refinements. (iii) At a refinement, $\Delta_k\le\zeta_k$, and since $\Delta_k$ is the maximum over $r\in T_{\zeta_k}(\theta_k)$ of $\min_{\sigma\in S_{\zeta_k}(\theta_k)}\langle\nabla_\theta J_\sigma(\theta_k),r\rangle$, every such $r$ admits $\sigma\in S_{\zeta_k}(\theta_k)$ with $\langle\nabla_\theta J_\sigma(\theta_k),r\rangle\le\zeta_k$, which is $\zeta_k$-stationarity. For the count, take $i:=\lceil\log(\zeta_0/\zeta)/\log(1/\nu)\rceil$, so that $\nu\zeta<\nu^{i}\zeta_0\le\zeta$. By (ii) the refinement ending level $i$ occurs within the stated bound, and its iterate is $\nu^{i}\zeta_0$-stationary, hence $\zeta$-stationary.
\end{proof}

\subsubsection{A Robust Actor-Critic implementation}
 Throughout this subsection, we assume that the active optimizer is unique, or that all active optimizers induce the same sensitivity vector.
\begin{algorithm}[H]
\caption{Robust Actor-Critic Gradient Algorithm}
\label{alg:robust_actor_critic}
\small                              
\algrenewcommand{\algorithmicindent}{1em}  
\setlength{\abovedisplayskip}{2pt}
\setlength{\belowdisplayskip}{2pt}
\begin{algorithmic}[1]

\Require Policy parameter $\theta \in \Theta$; parametric families $V_{\psi,t} : \Xc \to \R$ and $U_{\xi,t} : \Xc \to \R^{d_\theta}$; samples $(x_t, a_t)$ with $a_t \sim \pi^\theta_t(x_t, \cdot)$ and $X_{t+1}^{(i)} \sim \P_t^0(x_t, a_t, \cdot)$; initial distribution $\mu_0$; actor step size $\eta_\theta$

\State Set $V_{\psi,T}(x) \leftarrow g(x)$ and $U_{\xi,T}(x) \leftarrow 0$ for all $x \in \Xc$

\State \textbf{for} $t = T-1, T-2, \dots, 0$ \textbf{do}
    \State \quad Compute
        $$\widehat{G}_t^\theta(x_t, a_t)
        \;=\;
        \sup_{\lambda \in [0,\Lambda]}
        \Big(
            -\frac{1}{N} \sum_{i=1}^N
            \Fc^\lambda\big(-f(t,x_t,a_t,\cdot)-V_{\psi,t+1}(\cdot)\big)\left(X_{t+1}^{(i)}\Big)
            - \varepsilon^q \lambda
        \right)$$
    \State \quad Let $\widehat\lambda_t^*(x_t,a_t)$ denote a maximizer.
    \State \quad Update $\psi$ by minimizing:
    \[
        \Lc_V(\psi) \;=\; \E_{a\sim\pi^\theta_t(x,\cdot)}\left[\left(V_{\psi,t}(x_t) - \widehat{G}_t^\theta(x_t, a_t)\right)^2\right]
    \]
    \State \quad For each $i = 1, \dots, N$, compute 
    \[
        y^*\left(X_{t+1}^{(i)}\right)
        \;\in\;
        \argmin_{y \in \Xc}
        \left\{
            f(t,x_t,a_t,y)+V_{\psi,t+1}(y)
            + \widehat{\lambda}_t^*(x_t,a_t)\, c\left(X_{t+1}^{(i)}, y\right)
        \right\}
    \]
    \State \quad Use the directional derivative identity
    \[
        -D\Fc^{\widehat{\lambda}_t^*(x_t,a_t)}\left(-f(t,x_t,a_t,\cdot)-V_{\psi,t+1}(\cdot)\right)\left(X_{t+1}^{(i)}\right)\left[U_{\xi,t+1}\right]
        \;\approx\;
        U_{\xi,t+1}\left(y^*\left(X_{t+1}^{(i)}\right)\right)
    \]
    \State \quad Compute:
    \[
        \widehat{\nabla_\theta G_t^\theta}(x_t, a_t)
        \;=\;
        \frac{1}{N} \sum_{i=1}^N
        U_{\xi,t+1}\left(y^*\left(X_{t+1}^{(i)}\right)\right)
    \]
    \State \quad Form the gradient target:
        $$\widehat{z}_t
        \;=\;
        \widehat{G}_t^\theta(x_t,a_t)\,
        \nabla_\theta \log \pi^\theta_t(x_t,a_t)
        \;+\;
        \widehat{\nabla_\theta G_t^\theta}(x_t,a_t)$$
    \State \quad Update $\xi$ by minimizing:
        $$\Lc_U(\xi) \;=\; \E\left[\left\|U_{\xi,t}(x_t) - \widehat{z}_t\right\|^2\right]$$
\State Draw $X_0^{(j)} \sim \mu_0$ for $j = 1, \dots, M$ and update:
    $$\theta
\leftarrow
\Pi_{\Theta}
\left(
\theta
+
\eta_{\theta}\,
\frac{1}{M}
\sum_{j=1}^{M}
U_{\xi,0}\left(X_0^{(j)}\right)
\right),$$
    where $\Pi_{\Theta}$ denotes Euclidean projection onto $\Theta$.
\end{algorithmic}
\end{algorithm}

\begin{Remark} Note that:
\begin{enumerate}
\item In the tabular softmax experiments we take
$\Theta=[-B,B]^{d_\theta}$ and project the iterates after each actor update.
\item While Algorithm \ref{alg:robust_actor_critic} is stated for Wasserstein
ambiguity sets, the same actor--critic architecture can be adapted to other
families of state--action rectangular ambiguity sets
$\{\mathcal B(\mathbb P^0(x,a,\cdot))\}_{x,a}$. Such an extension requires that
the corresponding robust Bellman problem admit a tractable dual
representation, that the worst-case value be attained, and that the relevant
active dual or primal optimizers satisfy the envelope or directional-sensitivity
conditions needed to differentiate the robust Bellman operator. Under these
problem-specific conditions, the implementation changes mainly in two places:
(i) the dual or variational representation used to compute the robust Bellman
target in the critic step, and (ii) the active optimizer, or worst-case
transition law $\P^{*,\theta}$, used in the sensitivity term of the actor update. The remaining
structure of the algorithm is independent of the particular geometry of the ambiguity set.
\item The introduction of parametric families
  $$V_{\psi,t}:\Xc\to\R, \quad U_{\xi,t}:\Xc\to\R^{d_\theta},$$
intended to approximate $V_t^\theta(x)$ and $\nabla_\theta V_t^\theta(x)$ respectively, is not an intrinsic feature of the algorithm but rather a practical necessity in large or continuous state spaces. In settings where $\Xc$ is finite and of moderate size, such as the numerical experiments below, both quantities can be computed exactly by direct tabulation, rendering the approximation step superfluous. Function approximation should therefore be understood as a scalability device.
\item A convergence analysis of Algorithm \ref{alg:robust_actor_critic} appears feasible in standard restricted settings, for example in the finite-state finite-action case with tabular critics, exact solution of the inner robust Bellman problems, sufficient exploration, and Robbins–Monro step sizes. In that regime, the critic and sensitivity updates can be viewed as stochastic approximation schemes for the robust Bellman and sensitivity recursions, while the actor update follows the corresponding robust policy-gradient, or selected directional-gradient, direction. We do not pursue such a result here because the full algorithm also allows function approximation, Monte Carlo Bellman targets, numerical dual optimization, and approximate transport selection. Controlling the interaction of these errors would require additional assumptions and a separate stochastic approximation analysis.
\end{enumerate}
\end{Remark}

\section{Numerical experiments}\label{sec: Numerical experiments}
In this section, we implement the policy-gradient recursion derived in Corollary \ref{corr: dpp_grad}. We first consider finite discrete state and action spaces, where exact dynamic programming can be used as a benchmark for the actor--critic training. We then turn to a continuous-state example to illustrate the applicability of the method beyond the tabular setting. Additional numerical experiments are reported in Appendix \ref{app:num}, including exact robust-DP recovery tests, Bellman-sensitivity ablations, dual-grid sensitivity checks, a supplementary multi-armed bandits example, and further continuous-control experiments. The code can be found at \url{https://github.com/YadhHafsi/Policy-Gradient-Learning-for-Distributionally-Robust-Optimization}.\footnote{The authors are grateful to Ariel Neufeld and Julian Sester for sharing code that was helpful in validating some of the numerical experiments.}
\subsection{Examples}
We report two types of numerical experiments. 
The first group consists of finite tabular examples where the equicontinuity condition is automatic and exact robust dynamic programming can be used as a benchmark. For the finite-state coin-toss and inventory examples, all inner optimizations are over finite sets. We therefore compute the active maximizer/minimizer sets exactly and use the directional derivative recursion. The uniqueness assumptions are not imposed in these experiments. The second group consists of a continuous linear-quadratic example. 
This example lies outside the compact state-action assumptions of the main theory and is compared with a related penalized Riccati formulation rather than with the constrained Wasserstein-ball problem. It is included as a structural sanity check for the implementation with function approximation, not as a theorem-backed benchmark.
\subsubsection{Coin Toss}
We consider a finite-horizon robust control problem on the discrete state space
$\mathcal{X} = \{0,1,\dots,n\},$ where $X_t \in \mathcal{X}$ denotes the number of heads
observed in a block of $n \in \N^*$ Bernoulli trials at time t. We set the number of time steps to be equal to $T = 10$. The action space is given by
$A = \{-1,0,1\},$ where $a=-1$ corresponds to betting that the next block outcome will be lower, $a=0$ to abstaining from betting, and $a=1$ to betting that the next block outcome will be higher. Under the reference model, the next-state distribution is independent of the current state-action pair and satisfies
$$
X_{t+1} \sim \P^0(x,a,\cdot) = \mathrm{Binomial}(n,p_0),
\quad \forall (x,a)\in \mathcal{X}\times A,
$$
where $p_0 \in (0,1)$ denotes the reference success probability. The running reward is time-homogeneous and depends only on $(x,a,x')$: we set
$$
f(t,x,a,x')
=
a\,\mathbbm{1}_{\{x' > x\}}
-
a\,\mathbbm{1}_{\{x' < x\}}
-
|a|\,\mathbbm{1}_{\{x' = x\}},
\quad \forall(t,x,a,x') \in \Tc\times\Xc\times A\times\Xc,
$$
and the terminal reward is
$$
g(x) = 0, \quad \forall x\in\Xc.
$$
Hence a correct directional bet yields reward $+1$, an incorrect bet yields $-1$, and a tie incurs a penalty of $-|a|$. The horizon is $T=10$, the block length is $n\in\N^*$, and the only coefficient of the model is the reference Bernoulli parameter $p_0\in(0,1)$. We equip the state space with the ground cost
$c(x,y)=|x-y|$ for all $(x,y)\in\Xc\times\Xc$, and define the Wasserstein
ambiguity set around the reference transition kernel $\P^0$.

We parametrize $\pi^\theta$ as a time-dependent tabular softmax policy
$$
\pi^\theta_t(x,a) = \frac{\exp(\theta_{t,x,a})}{\sum_{a'\in A}\exp(\theta_{t,x,a'})}, \quad \forall (t,x,a)\in\Tc\times\Xc\times A,
$$
with parameters $\theta = (\theta_{t,x,a})_{t,x,a}\in\R^{|\Tc|\times|\Xc|\times|A|}$. Table \ref{tab:coin_policy_by_epsilon} reports the resulting greedy actions $a^*_t(x) = \argmax_{a\in A}\pi^\theta_t(x,a)$
at each state for increasing $\varepsilon$. As $\varepsilon$ grows, the betting region shrinks symmetrically around
$X_t=5$. The agent abstains from betting on more states, reflecting
increased caution under larger ambiguity sets. At $\varepsilon=2$, the ambiguity set is large enough that the robust policy
abstains at every state. The learned policies match those specified in \cite{neufeld_markov_2023}.
{\small\begin{table}[H]
\centering
\small
\begin{tabular}{cccccccccccc}
\toprule
$X_t$ & $0$ & $1$ & $2$ & $3$ & $4$ & $5$ & $6$ & $7$ & $8$ & $9$ & $10$ \\
\midrule
$a^{\mathrm{non\mbox{-}robust}}_t(X_t)$
& $1$ & $1$ & $1$ & $1$ & $1$ & $0$ & $-1$ & $-1$ & $-1$ & $-1$ & $-1$ \\

$a^{\mathrm{robust},\,\varepsilon=0.5}_t(X_t)$
& $1$ & $1$ & $1$ & $0$ & $0$ & $0$ & $0$ & $0$ & $-1$ & $-1$ & $-1$ \\

$a^{\mathrm{robust},\,\varepsilon=1}_t(X_t)$
& $1$ & $1$ & $0$ & $0$ & $0$ & $0$ & $0$ & $0$ & $0$ & $-1$ & $-1$ \\

$a^{\mathrm{robust},\,\varepsilon=2}_t(X_t)$
& $0$ & $0$ & $0$ & $0$ & $0$ & $0$ & $0$ & $0$ & $0$ & $0$ & $0$ \\
\bottomrule
\end{tabular}
\caption{Greedy action at time $t\in\Tc$ in the coin-toss example as a function
of the state $X_t\in\mathcal{X}$ for the non-robust policy and robust policies with
different Wasserstein radii $\varepsilon \geq 0$.}
\label{tab:coin_policy_by_epsilon}
\end{table}}
Figure \ref{fig:misspecification} shows the cumulated profit of each policy
as the true bias $p_{\mathrm{true}}$ departs from the reference $p_0$. For each strategy, we simulate $10^5$ rounds of the coin-toss game while varying the true bias $p_{\mathrm{true}}$ away from the reference parameter $p_0$.
\begin{figure}[H]
    \centering
    \includegraphics[width=0.45\textwidth]{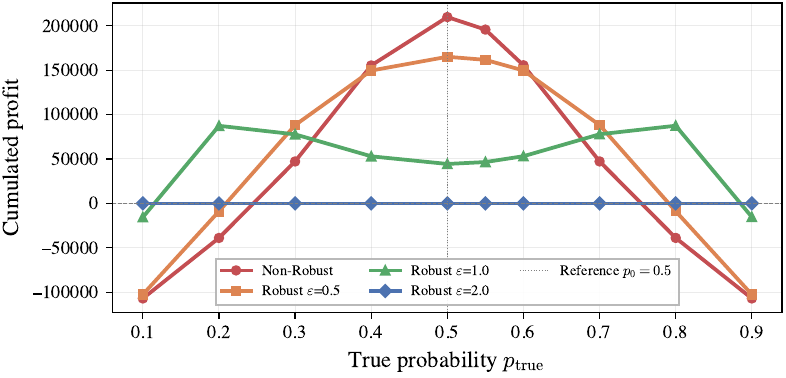}%
    \hspace{0.3cm}
    \includegraphics[width=0.45\textwidth]{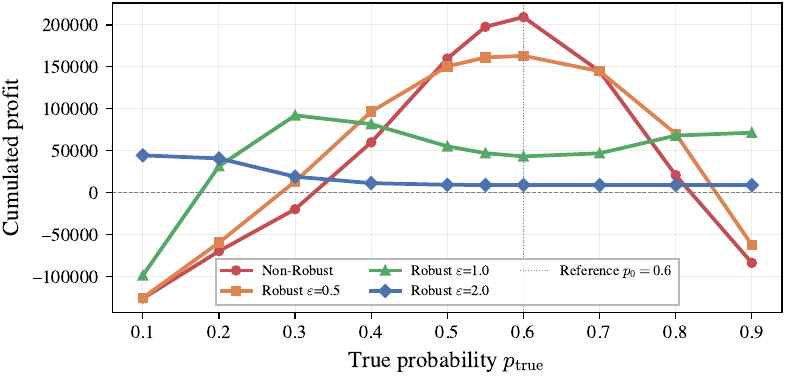}
    \caption{%
        Cumulated profit under model misspecification for $p_0=0.5$ (left)
        and $p_0=0.6$ (right).
    }
    \label{fig:misspecification}
\end{figure}
As $\varepsilon$ increases, the robust policy sacrifices nominal performance near $p_0$ in exchange for substantially improved worst-case performance over the full range of values of $p_{\mathrm{true}}$. This behavior is consistent with the intended effect of distributionally robust control, namely, larger $\varepsilon$ flattens the profit curve: the agent sacrifices peak
performance near $p_0$ for uniform robustness across the full range
of $p_{\mathrm{true}}$.

\subsubsection{Supply chain}
We also consider a finite-horizon inventory control problem with state space
$\mathcal{X} = \{0,1,\dots,n\},$ where $X_t \in \mathcal{X}$ denotes the on-hand
inventory at time $t \in \Tc$. The action space is $A = \{0,1,\dots,n\},$
where $A_t \in A$ denotes the order quantity placed at time $t$. Given state-action pair
$(x,a)\in\mathcal{X}\times A$, the post-order inventory is
$\bar{x} = \min\{n,\,x+a\} \in \mathcal{X}$. Demand $D_t$ is drawn from the reference
distribution $D_t \sim \mathrm{Uniform}\big(\{0,1,\dots,n\}\big),$ and the next inventory
evolves as
$$
X_{t+1} = \max\{0,\,\bar{x} - D_t\} \in \mathcal{X}.
$$
The one-step cost is
$$
\ell(x,a,d)
=
h\,(\bar{x}-d)_+
+
p\,(d-\bar{x})_+
+
k\,\mathbbm{1}_{\{a>0\}},
\quad \forall(x,a,d)\in\mathcal{X}\times A\times\{0,\dots,n\},
$$
where $h>0$ is the holding-cost coefficient, $p>0$ is the shortage-cost coefficient, and
$k>0$ is a fixed ordering cost. The reward is the negative cost,
$$
r(x,a,d) = -\ell(x,a,d), \quad \forall(x,a,d)\in\mathcal{X}\times A\times\{0,\dots,n\}.
$$
Since the actor--critic interface is written in terms of $(x,a,x')$, we equivalently work with the conditional expected reward given the realized next-state inventory level $x'\in\Xc$. Concretely, we set
$$
f(t,x,a,x') = -\,\E_{D\sim\mathrm{Uniform}\{0,\dots,n\}}\big[\ell(x,a,D)\,\big|\,X_{t+1} = x'\big], \quad g(x) = 0, \quad \forall(t,x,a,x')\in\Tc\times\Xc\times A\times\Xc.
$$
For robustness, we place a Wasserstein ambiguity set around the reference demand-induced transition kernel, with state-space cost $c(x,y)=|x-y|$ for all $(x,y)\in\Xc\times\Xc$. The model coefficients are: maximal on-hand inventory $n$, horizon $T$, Wasserstein radius $\varepsilon\ge 0$, holding cost $h>0$, shortage cost $p>0$, and fixed ordering cost $k>0$. In the numerical experiments, we take $n=10$, $\varepsilon=1$, $T=5
$, $h=1$, $p=3$, and $k=2$. As in the coin-toss example, we parametrize $\pi^\theta$ as a time-dependent tabular softmax policy
$$
\pi^\theta_t(x,a) = \frac{\exp(\theta_{t,x,a})}{\sum_{a'\in A}\exp(\theta_{t,x,a'})}, \quad \forall (t,x,a)\in\Tc\times\Xc\times A,
$$
and illustrate the greedy action $a^*_t(x) = \argmax_{a\in A}\pi^\theta_t(x,a)$ extracted from the learned stochastic policy. Figure \ref{fig:supply_chain} reports the results for the learned policy and compares the costs for different ambiguity sets.
\begin{figure}[H]
    \centering
    \includegraphics[width=0.49\textwidth]{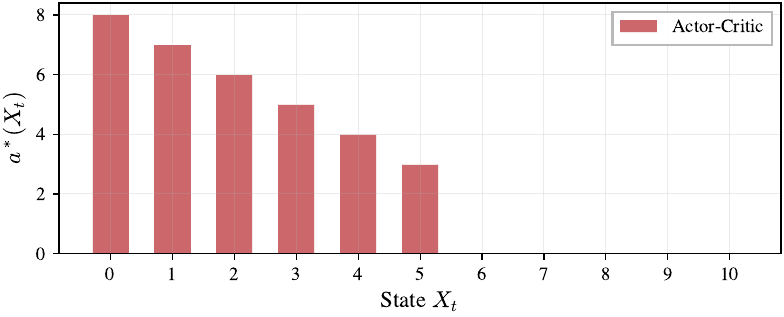}%
    \hspace{0.3cm}
\includegraphics[width=0.28\textwidth]{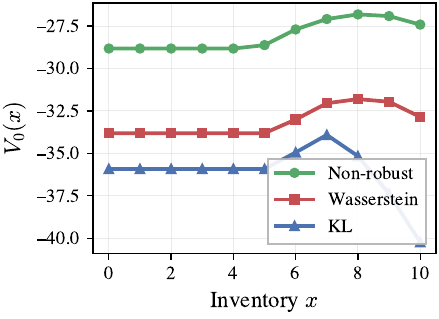}
    \caption{%
        Learned greedy ordering policy $a^*(x)$ at $t=0$ (left); value function $V_0(x)$ for the learned policy under
        Wasserstein ambiguity, KL ambiguity, and the non-robust setting (right).
    }
    \label{fig:supply_chain}
\end{figure}

\subsubsection{Robust Linear-Quadratic Control}
\label{sec:robust_lq}

We consider a finite-horizon linear-quadratic control problem under distributional
uncertainty on the noise. A closely related penalized robust LQ problem admits an explicit Riccati
solution. We use it as a structural sanity check for the learned policy. This benchmark falls outside the compact state-action framework of the theory and serves to test the method in a standard control setting. We include it to test whether the proposed implementation recovers the qualitative structure of a related robust LQ solution. The state dynamics are
\begin{equation}\label{eq:lq_dynamics}
    X_{t+1} = A X_t + B u_t + \Xi w_t, \quad X_0 = x \in \R^d,
    \quad t = 0,\dots,T-1,
\end{equation}
where $A \in \R^{d\times d}$, $B \in \R^{d \times m}$, $\Xi \in \R^{d\times k}$,
$u_t \in \R^m$ is the control, and $w_t \in \R^k$ is a random disturbance with
distribution $\mu_t \in \Pc(\R^k)$. Under the reference model, each $w_t$ is drawn
i.i.d.\ from an empirical distribution
$$\nu = \frac{1}{N}\sum_{i=1}^{N}\delta_{\hat w^{(i)}}$$ constructed from samples
$\{\hat w^{(1)},\dots,\hat w^{(N)}\}$. The running and terminal costs are 
\begin{equation*}
    f(t, x, u, x') = x^\top Q x + u^\top R u,
    \quad
    g(x) = x^\top P_T x, \quad \forall(t,x,u,x') \in \Tc \times \R^d \times \R^m \times \R^d,
\end{equation*}
with $Q \in \R^{d\times d}$, $Q \succeq 0$ (state penalty); $R \in \R^{m\times m}$, $R \succ 0$ (control penalty); and $P_T \in \R^{d\times d}$, $P_T \succeq 0$ (terminal penalty). The remaining coefficients of the model are the dynamics matrices $A\in\R^{d\times d}$, $B\in\R^{d\times m}$, $\Xi\in\R^{d\times k}$, the number of empirical samples $N\in\N^*$, the Wasserstein radius $\varepsilon\ge 0$, and the penalty parameter $\lambda>0$ used in the Riccati benchmark of Proposition \ref{prop:robust_lq}.\\

The agent distrusts the nominal distribution and solves the robust minimax problem
\begin{equation}\label{eq:lq_minimax}
    V_t(x) \;=\; \underset{(\pi_s)_{s \in \mathcal{T}_t} }{\inf}~ \underset{(\gamma_s)_{s \in \mathcal{T}_t}}{\sup}\;
    \E^{\pi,\gamma}\left[
        \sum_{s=t}^{T-1}\bigl(X_s^\top Q\, X_s + u_s^\top R\, u_s\bigr)
        + X_T^\top P_T\, X_T
    \right],
\end{equation}
subject to \eqref{eq:lq_dynamics}, where the controller chooses feedback policies
$u_t = \pi_t(X_t) \in \R^m$ and the adversary chooses, at each stage, a distribution
$\gamma_t(X_t, u_t) \in \Pc\bigl(\{\hat w^{(1)},\dots,\hat w^{(N)}\}\bigr)$
for the noise $w_t$, constrained by
\begin{equation}\label{eq:lq_wass_constraint}
    W_q\bigl(\gamma_t(X_t, u_t),\;\nu\bigr) \;\leq\; \varepsilon,
    \quad t = 0,\dots,T{-}1.
\end{equation}
For reference, we compare against the Wasserstein penalty framework
of \cite{yang2021} and \cite{KimYang2021},
which solves a closely related but different problem. Starting from the same empirical
nominal $\nu$, their adversary transports each atom to any $w \in \R^k$
(not just the $N$~existing atoms), and the robust value is defined as
\begin{equation}\label{eq:lq_ky_value}
    \inf_{\pi}\;\sup_{\gamma}\;
    \E^{\pi,\gamma}\left[
        X_T^\top P_T\, X_T
        + \sum_{s=t}^{T-1}\bigl(X_s^\top Q\, X_s + u_s^\top R\, u_s
        - \lambda\, W_q(\gamma_s, \nu)^p\bigr)
    \;\middle|\; X_0 = x\right],
\end{equation}
where $\lambda > 0$ is a fixed parameter. In this penalized LQ formulation, the quadratic structure is preserved by the Riccati recursion, so the value function remains quadratic at each stage. The supremum over
$\gamma_t \in \Pc(\R^k)$ admits a closed-form solution, yielding the Riccati
recursion of Proposition \ref{prop:robust_lq}.

\begin{Proposition}[\cite{KimYang2021}]
\label{prop:robust_lq}
Define $\Phi := B R^{-1} B^\top - \frac{1}{\lambda}\,\Xi\Xi^\top$ and denote by
$\bar\lambda_t$ the largest eigenvalue of $\Xi^\top \Pi_t \Xi$.
Assume $\lambda > \bar\lambda_t$ for all $t \geq 1$. Then the value function takes the form
$$
    V_t(x) = x^\top \Pi_t\, x + 2\,r_t^\top x + z_t,
    \quad \forall\, (t,x) \in \{0,\dots,T\}\times\R^d,
$$
where $\Pi_t \in \Sc_+^d$ satisfies the robust Riccati recursion
\begin{equation}
\label{eq:robust_riccati}
    \Pi_t = Q + A^\top \bigl(I + \Pi_{t+1}\,\Phi\bigr)^{-1}\Pi_{t+1}\, A,
    \quad \Pi_T = P_T,
    \quad t = T{-}1,\dots,0,
\end{equation}
while $r_t \in \R^d$ and $z_t \in \R$ are given by
\begin{align*}
    r_t &= A^\top (I + \Pi_{t+1}\,\Phi)^{-1}
           (\Pi_{t+1}\,\Xi\,\bar w + r_{t+1}), \quad r_T = 0,\\[4pt]
    z_t &= z_{t+1}
           + \mathrm{tr}\bigl[
               (I - \Xi^\top \Pi_{t+1}\Xi/\lambda)^{-1}\,
               \Xi^\top \Pi_{t+1}\Xi\,\Sigma
             \bigr] \\
        &\quad
           + \bar w^\top \Xi^\top
             \bigl[(I+\Pi_{t+1}\Phi)^{-1}
                   -(I-\Pi_{t+1}\Xi\Xi^\top/\lambda)^{-1}\bigr]
             \Pi_{t+1}\Xi\,\bar w \\
        &\quad
           + (2\,\bar w^\top \Xi^\top - r_{t+1}^\top \Phi)
             (I+\Pi_{t+1}\Phi)^{-1} r_{t+1},
           \quad z_T = 0,
\end{align*}
with $\bar w := \E_\nu[w]$ and $\Sigma := \E_\nu[w\,w^\top]$.
The unique optimal control and worst-case distribution are
\begin{align*}
    u_t^* = K_t\, X_t + L_t, ~ 
    K_t := -R^{-1}B^\top(I+\Pi_{t+1}\Phi)^{-1}\Pi_{t+1}\,A,~
    L_t := -R^{-1}B^\top(I+\Pi_{t+1}\Phi)^{-1}(\Pi_{t+1}\Xi\,\bar w + r_{t+1}),
\end{align*}
where $\gamma_t^*(x) = \frac{1}{N}\sum_{i=1}^{N}\delta_{w_t^{*,(i)}(x)}$, and the support points of the worst-case distribution are
$$
    w_t^{*,(i)}(x) = (\lambda I - \Xi^\top \Pi_{t+1}\Xi)^{-1}
    \bigl(\Xi^\top \Pi_{t+1}(A x + B u_t^*) + \Xi^\top r_{t+1}
          + \lambda\,\hat w^{(i)}\bigr).
$$
\end{Proposition}
 The Riccati solution provides a related conservative reference problem, not an exact benchmark for the constrained Wasserstein-ball formulation used in the
main theory.

We illustrate Algorithm \ref{alg:robust_actor_critic} on the LQ instance described above, specialized to the one-dimensional case $d = m = 1$ for visualization. The LQ example is written in cost-minimization form. It is equivalent to the
reward-maximization convention of the previous sections by setting
$$f(t,x,u,x')=-(x^\top Qx+u^\top Ru),
\quad
g(x)=-x^\top P_Tx, \quad \forall(t,x,u,x') \in \Tc \times \R^d \times \R^m \times \R^d.$$
Thus the reward value in the previous sections is the negative of the cost value
reported in this subsection. Since both the state and action spaces are continuous, we parametrize $\pi^\theta$ as a truncated Gaussian density with variance bounded away from zero, such that
with
$$
\pi_t^\theta(x,a)
=
\frac{
\exp\left(
-\frac{\|a-\mu_\theta(t,x)\|^2}{2\sigma_\theta^2}
\right)
}{
\int_A
\exp\left(
-\frac{\|b-\mu_\theta(t,x)\|^2}{2\sigma_\theta^2}
\right)\d b
},
\quad \forall a\in A.
$$
where $(t,x)\in\Tc\times\R^d$ and the mean $\mu_\theta : \Tc\times\R^d \to \R^m$ is given by a two-layer feedforward neural network (one hidden layer, ReLU activation) taking $(t,x)$ as input, and $\log\sigma_\theta\in\R$ is a learned scalar log-standard-deviation. The parameter $\theta$ thus collects both the network weights and the log-volatility, and Algorithm \ref{alg:robust_actor_critic} jointly optimizes the mean $\mu_\theta$ and the standard deviation $\sigma_\theta$. We display the mean of the learned policy, i.e., the map $x\mapsto \mu_{\theta^*}(t,x)$, where $\theta^*$ denotes the parameter returned by Algorithm \ref{alg:robust_actor_critic}. We compare its form against the continuous-adversary Riccati solution $u^*_t(x) = K_t\,x + L_t$ of Proposition \ref{prop:robust_lq}. 
\begin{figure}[H]
    \centering
    \includegraphics[width=0.6\textwidth]{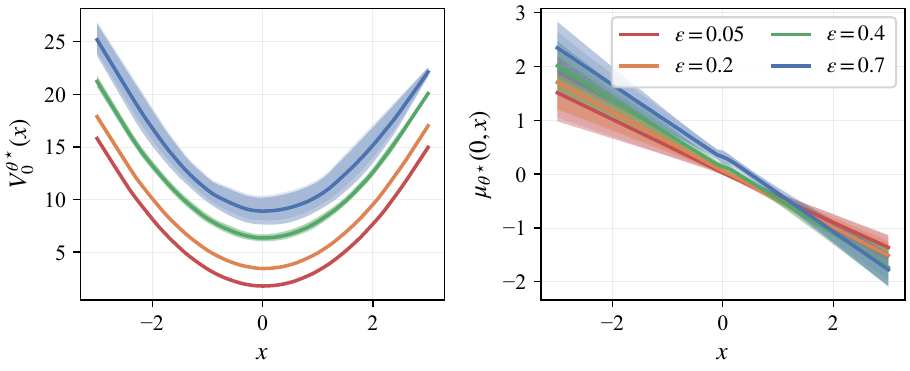}
    \caption{%
        Value function (left) and optimal policy (right)
        for four values of $\varepsilon$.
    }
    \label{fig:lq_comparison}
\end{figure}
Figure \ref{fig:lq_comparison} shows the learned policy and its corresponding value function. Since the Riccati benchmark has a quadratic value function and a linear optimal
control, the learned curves provide a structural consistency check for the
algorithm. The numerical results match this structure, and the learned stochastic policy becomes nearly deterministic, with $\sigma_{\theta^\star}\approx 0$.

\appendix

\renewcommand{\thesection}{\Alph{section}}




\section{Some useful results}
    In this section, we recall several concepts and results that will be useful for the proofs.

\begin{Definition} Let $\mathcal{X}$ be a Polish space, $\mathcal{B}(\mathcal{X})$ its Borel $\sigma$-algebra and $(\Omega,\mathcal{F})$ be a measurable space. Let $\psi : \Omega  \rightrightarrows \mathcal{X}$, i.e., for any $\omega \in \Omega$, $\psi(\omega) \subset \mathcal{X}$. We say that the set-valued map or correspondence $\psi$ is weakly-measurable if for every open subset $U \subset \mathcal{X}$, we have
\begin{align}
    \lbrace \omega \in \Omega :  \psi(\omega) \cap U \neq \emptyset \big \rbrace \in \mathcal{F}. \notag 
\end{align}
Given a correspondence $\varphi : \mathcal{X} \rightrightarrows Y $, we define its graph  as
\begin{align}
    \text{Gr}(\varphi) := \big \lbrace (x,y) \in \mathcal{X} \times Y : y \in \varphi(x) \big \rbrace \notag 
\end{align}    
\end{Definition}

\begin{Definition}
Let $\varphi : \mathcal{X} \rightrightarrows Y $ be a measurable correspondence between two topological spaces. 
\begin{enumerate}
    \item [(1)] $\varphi $ is called upper hemicontinuous  if $\lbrace x \in \mathcal{X} : \varphi(x)  \subset A \big \rbrace$ is an open set for all open sets $A \subset Y$.
    \item [(2)] $\varphi $ is called lower hemicontinuous  if $\lbrace x \in \mathcal{X} : \varphi(x)  \cap A \neq \emptyset \big \rbrace$ is an open set for all open sets $A \subset Y$.
    \item [(3)] We say that $\varphi$ is continuous if $\varphi$ is upper and lower hemicontinuous.
\end{enumerate}
    
\end{Definition}

\begin{Theorem}
\label{def : danskin_Theorem}
Let $C$ be a compact metric space and let $U$ be a subset of a normed vector
space. Let $f:C\times U\to\mathbb R$ be continuous, and define
$$
v(u):=\sup_{x\in C} f(x,u),~~\text{and}~~
S(u):=\argmax_{x\in C} f(x,u).
$$
Fix $u_0\in U$ and a feasible direction $d$. Suppose that, for every $x\in C$,
the directional derivative
$$
f'_x(u_0,d)
:=
\lim_{t\downarrow0}
\frac{f(x,u_0+td)-f(x,u_0)}{t}
$$
exists. Suppose moreover that, for every sequence $t_n\downarrow0$ and every
sequence $x_n\in C$ with $x_n\to\bar x$,
$$
\limsup_{n\to\infty}
\frac{
f(x_n,u_0+t_n d)-f(x_n,u_0)
}{t_n}
\leq
f'_{\bar x}(u_0,d).
$$
Then $v$ admits a right directional derivative at $u_0$ in the direction $d$ and
$$
v'(u_0,d)
=
\sup_{x\in S(u_0)} f'_x(u_0,d).
$$
\end{Theorem}
\begin{proof}
    See Proposition 4.12 \cite{Bonnans1998} and Proposition B.22 in \cite{Bertsekas1999}.
\end{proof}

\begin{Theorem}\label{thm : Continuity under weak convergence} Let $X$ be a complete metric space and $A$ a compact metric space. Let $h : \Xc \times A \ni (x,a) \mapsto h(x,a) \in \R$ be a continuous mapping and $\Xc \ni x \mapsto \mu_x \in \Pc(A)$ be a weakly continuous mapping, i.e.,
\begin{align}
    d(x_n,x) \to 0 \Rightarrow  \mu_{x_n} \text{ weakly converges to } \mu_x.
\end{align}
Then, the mapping 
\begin{align}
    F : \Xc \ni x \mapsto \int_{A} h(x,a) \d \mu_x(a),
\end{align}
is continuous.

\end{Theorem}
\begin{proof}
    See Theorem 2.1 \cite{billingsley_convergence_1999}.
\end{proof}

\begin{Lemma}\label{Lemma : berge_maximum_Theorem} Let $\varphi : \mathcal{X} \rightrightarrows \mathcal{Y}$ be a upper and lower hemicontinuous correspondence between topological spaces with nonempty compact values, and suppose that  $f : \lbrace (x,y) \in \mathcal{X} \times \mathcal{Y} : y \in \varphi(x) \big \rbrace \to \mathbb{R}$ is continuous. Then the map $m : \mathcal{X} \ni x \mapsto  \underset{y \in \varphi(x)}{\sup} f(x,y)$ is continuous.

\end{Lemma}

\begin{proof}
    The proof of this result can be found in Theorem 17.31 in \cite{AliprantisBorder2006}.
\end{proof}

\begin{Lemma}\label{Lemma : measurable_maximum_Theorem}Let $\mathcal{X}$ be a Polish space and let $\varphi : S \rightrightarrows \mathcal{X}$ be a weakly measurable correspondence with non empty compact values and  suppose $f : S \times \mathcal{X} \to \mathbb{R}$ is a Carathéodory map, i.e, a map satisfying the following assumptions
\begin{enumerate}
    \item [(1)] For any $x \in \mathcal{X}$, the map $f(\cdot,x) $ is $\Sigma$-measurable
    \item [(2)] For any $s \in S$, the map $f(s,\cdot)$ is continuous.
\end{enumerate}
Define the map $m: S \to \mathbb{R}$ 
\begin{align}
    S \ni s \mapsto m(s) := \underset{x \in \varphi(s)}{\inf} f(s,x), \notag 
\end{align}
and the correspondence of minimizers by the map $\mu :$
\begin{align}
    S \ni s \mapsto  \mu(s) := \big \lbrace  x \in \varphi(s) : f(s,x) = m(s) \big \rbrace. \notag 
\end{align}
Then, the following holds
\begin{enumerate}
    \item [(1)] The value  function $m$ is measurable
    \item [(2)] The argmin correspondence $\mu$ has nonempty and compact values
    \item [(3)] The argmin correspondence $\mu$ is weakly measurable and admits a measurable selector
\end{enumerate}

\begin{proof}
    The proof of this result can be found in Theorem 18.19 in \cite{AliprantisBorder2006}.
\end{proof}

\end{Lemma}

\begin{Theorem}\label{thm : measurable_selection} Let $X$ be a Polish space, $\mathcal{B}(X)$  its Borel $\sigma-$algebra and $(\Omega,\mathcal{F})$ be a measurable space. Let $\psi$ be a weakly-measurable correspondence in $\Omega$ taking values in the set of nonempty closed subsets of $X$. Then $\psi$ admits a measurable selection, i.e a $(\mathcal{B}(X),\mathcal{F})$ measurable map.
\end{Theorem}

\begin{proof}
    The proof of this result can be found in \cite{kuratowski_general_1965}
\end{proof}

\section{Proofs of additional results}
    This section provides additional proofs of the results stated above.
\subsection{Proof of the interchange of minimization and integration}
\begin{Theorem}\label{thm : Interchange of Min and Int}
Let $(\Omega,\Fc,\mu)$ be a $\sigma$-finite measurable space. Let $X$ be a Polish space. Let $h:\Omega\times X\rightarrow \R$ be a normal integrand, which means that $h$ satisfies the following conditions:
\begin{itemize}
\item [(i)] $h:\Omega\times X\rightarrow\R$ is $\Fc\otimes\Bc(X)$-measurable;
    \item[(ii)] for every $a\in \Omega$, $h(a,\cdot)$ is lower semi-continuous. 
\end{itemize}
Let $\Xc=\{x: \Omega\rightarrow X\,|\,x \textit{ is a measurable function such that $h(\cdot,x(\cdot))$ is integrable}\}$. Assume further that for every $a\in \Omega$, $\inf_{x\in X}h(a,x)$ is finite. Then, if $\int_{\Omega}h(a,x(a))\mu(\d a)\not\equiv\infty$ on $\Xc$, one has
\begin{align}
\inf_{x\in\Xc}\int_{\Omega}h(a,x(a))\mu(\d a)=\int_{\Omega}\inf_{x\in \Xc}h(a,x(a))\mu(\d a).
\end{align}
\end{Theorem}
\begin{proof}
See Theorem $14.60$ in \cite{RockafellarWets1998}.
\end{proof}

\subsection{Proofs of the continuity of the value function}\label{app: proof_cont}

\begin{proof}[Proof of Proposition \ref{eq:uniform-cont}]
We argue by backward induction. The claim holds for $t=T$ since
$V_T^\theta=g\in \mathcal{C}(\mathcal{X};\mathbb{R})$.
Assume $V_{t+1}^\theta\in \mathcal{C}(\mathcal{X};\mathbb{R})$. Fix
$\lambda\ge 0$, $X\sim \P_t^0(x,a,\cdot)$ and $\theta\in\Theta$, and recall that
$$
F_t^\theta(\lambda;x,a)
:=
\E_{X\sim\P_t^0(x,a,\cdot)}
\big[-\Fc^{\lambda}(-f(t,x,a,\cdot)-V^\theta_{t+1}(\cdot))(X)\big]
-\varepsilon^q\lambda,
\quad \forall (t,x,a)\in\Tc\times\Xc\times A,
$$
 Here some quantities depend on $(x,a)$, but we omit this dependence to alleviate
the notation. By Lemma \ref{lem:compact_dual_maximizer_set}, there exists a
finite constant $\Lambda\geq0$, independent of $(t,\theta,x,a)$, such that
$$
\sup_{\lambda\geq0}F_t^\theta(\lambda;x,a)
=
\sup_{\lambda\in[0,\Lambda]}F_t^\theta(\lambda;x,a).
$$
 Because $\Xc$ is compact and $c$ is continuous, $c$ is bounded. For any fixed $\lambda,\lambda'\ge 0$, we have the Lipschitz bound
\begin{equation}
\label{eq:Phi-lambda-Lip}
\|\Fc^{\lambda}(-f(t,x,a,\cdot)-V_{t+1}^\theta(\cdot))-\Fc^{\lambda'}(-f(t,x,a,\cdot)-V_{t+1}^\theta(\cdot))\|_\infty
\le
\|c\|_\infty\,|\lambda-\lambda'|.
\end{equation}
Indeed, for each $x'\in \Xc$, because $f(t,x,a,\cdot)$, $V^\theta_{t+1}$, and $c(x',\cdot)$ are continuous on the compact $\Xc$, $$Y^*_{t,\theta,\lambda}(x,a;x')
:=
\argmin_{y\in\Xc}\big\{f(t,x,a,y)+V^\theta_{t+1}(y)+\lambda c(x',y)\big\}$$ is non-empty. Pick any $y_\lambda(x,a;x')\in Y^*_{t,\theta,\lambda}(x,a;x')$ and $y_{\lambda'}(x,a;x')\in Y^*_{t,\theta,\lambda'}(x,a;x')$. Since $y_\lambda(x,a;x')$ is optimal at $\lambda$ and $y_{\lambda'}(x,a;x')$ is feasible, and vice versa,
$$
-\Fc^{\lambda}(-f(t,x,a,\cdot)-V^\theta_{t+1}(\cdot))(x') - (-\Fc^{\lambda'}(-f(t,x,a,\cdot)-V^\theta_{t+1}(\cdot))(x'))
\le (\lambda-\lambda')\,c(x',y_{\lambda'}(x,a;x')),
$$
and symmetrically; taking absolute values yields
$$\big|-\Fc^{\lambda}(-f(t,x,a,\cdot)-V^\theta_{t+1}(\cdot))(x')+\Fc^{\lambda'}(-f(t,x,a,\cdot)-V^\theta_{t+1}(\cdot))(x')\big|\le \|c\|_\infty\,|\lambda-\lambda'|,$$ which is the Lipschitz bound \eqref{eq:Phi-lambda-Lip} uniformly in $(x,a, x')$.
This yields continuity of $\lambda\mapsto -\Fc^{\lambda}(-f(t,x,a,\cdot)-V^\theta_{t+1}(\cdot))(x')$ uniformly in $x'$, and therefore continuity of
$\lambda\mapsto F_t^\theta(\lambda;x,a)$. By Lemma \ref{lem:compact_dual_maximizer_set},
$$
\Lambda_t^{*,\theta}(x,a)
:=
\argmax_{\lambda\ge 0}
\left\{
\E_{X\sim \P_t^0(x,a,\cdot)}
\big[
-\Fc^\lambda(-f(t,x,a,\cdot)-V_{t+1}^\theta(\cdot))(X)
\big]
-\varepsilon^q\lambda
\right\}
\subset[0,\Lambda],
$$
and $\Lambda_t^{*,\theta}(x,a)$ is nonempty and compact.

 For fixed $\lambda\in[0,\Lambda]$, the map
$b\mapsto \Fc^\lambda(b)$ is $1$-Lipschitz for $\|\cdot\|_\infty$:
$$
\begin{aligned}
\big|\Fc^\lambda(b)(x')-\Fc^\lambda(b')(x')\big|
&=
\big|\sup_{y\in\Xc}\{b(y)-\lambda c(x',y)\}
-\sup_{y\in\Xc}\{b'(y)-\lambda c(x',y)\}\big|
\\
&\leq \sup_{y\in\Xc}|b(y)-b'(y)|
=
\|b-b'\|_\infty,
\end{aligned}
$$
for $b,b'\in \mathcal{C}(\mathcal{X};\mathbb{R})$ and any $x'\in\Xc$.
Moreover, since $c$ is continuous on the compact set
$\Xc\times\Xc$ and $f(t,\cdot,\cdot,\cdot)$,
$V_{t+1}^\theta$ are continuous, Berge's maximum theorem
\cite[Theorem~17.31]{AliprantisBorder2006} applied to
$$
(x,a,x')\mapsto
\sup_{y\in\Xc}
\{-f(t,x,a,y)-V_{t+1}^\theta(y)-\lambda c(x',y)\}
$$
yields continuity of
$$
(x,a,x')\mapsto
\Fc^\lambda(-f(t,x,a,\cdot)-V_{t+1}^\theta(\cdot))(x')
$$
on $\Xc\times A\times\Xc$.

By the weak continuity of the nominal transition kernel
$(x,a)\mapsto\P_t^0(x,a,\cdot)$, it follows that
$$
(x,a)\mapsto
\E_{X\sim \P_t^0(x,a,\cdot)}
\big[
-\Fc^\lambda(-f(t,x,a,\cdot)-V_{t+1}^\theta(\cdot))(X)
\big]
$$
is continuous on $\Xc\times A$. The same argument, combined with
\eqref{eq:Phi-lambda-Lip}, shows that
$$
(\lambda,x,a)\mapsto
\E_{X\sim \P_t^0(x,a,\cdot)}
\big[
-\Fc^\lambda(-f(t,x,a,\cdot)-V_{t+1}^\theta(\cdot))(X)
\big]
-\varepsilon^q\lambda
$$
is jointly continuous on $[0,\Lambda]\times\Xc\times A$.

A second application of Berge's maximum theorem
\cite[Theorem~17.31]{AliprantisBorder2006}, now over the compact set
$[0,\Lambda]$, yields continuity of
$$
(x,a)\mapsto
\sup_{\lambda\in[0,\Lambda]}
\left(
\E_{X\sim \P_t^0(x,a,\cdot)}
\big[
-\Fc^\lambda(-f(t,x,a,\cdot)-V_{t+1}^\theta(\cdot))(X)
\big]
-\varepsilon^q\lambda
\right).
$$
By Lemma \ref{lem:compact_dual_maximizer_set}, the supremum over
$\mathbb R_+$ coincides with the supremum over $[0,\Lambda]$. Consequently,
$(x,a)\mapsto G_t^\theta(x,a)$ is continuous and bounded on $\Xc\times A$.

Finally, by the weak continuity of the policy kernel
$x\mapsto\pi_t^\theta(x,\cdot)$, the map
$$
x\mapsto
\int_A G_t^\theta(x,a)\,\pi_t^\theta(x,\d a)
$$
is continuous on $\Xc$. Thus, from the DPP in Remark \ref{G_dpp},
$V_t^\theta\in \mathcal{C}(\mathcal{X};\mathbb{R})$.
\end{proof}

\subsection{Proofs of directional differentiability}\label{app: proof_diff}
\begin{proof}[Proof of Proposition \ref{eq:uniform-dir-diffuniform-bound-dq}]
Fix $\theta\in\Theta$ and a feasible direction $r\in\mathbb R^{d_\theta}$. We prove the claim by backward induction on $t\in\bar{\mathcal T}$. For $t=T$, $V_T^\theta=g$ is independent of $\theta$. Hence, for all admissible one-sided $h$, $$ \frac{V_T^{\theta+hr}(x)-V_T^\theta(x)}{h}=0,\quad \forall x\in\mathcal X. $$ Thus $D_\theta V_T^\theta(\cdot)[r]=0$, the difference quotients converge uniformly, and the uniform boundedness claim is immediate. Define for any $t \in \mathcal{T}$ the map
\begin{align}
    \mathcal{X} \times A \ni (x,a) \mapsto G_{t}^{\theta}(x,a) &:=\underset{\lambda \geq 0}{\text{sup }}  \Big( \int_{\mathcal{X}} - \mathcal{F}^{\lambda}(- f(t,x,a,\cdot) - V_{t+1}^{\theta})(x') \mathbb{P}_t^0(x,a, \ud x') - \varepsilon^q \lambda \Big) .
\end{align}
From the proof of Proposition \ref{eq:uniform-cont}, we know that $G_{t}^{\theta}$ is a  continuous bounded map on $\mathcal{X} \times A$.
\begin{equation*}
\begin{split}
    \frac{V_t^{\theta + rh}(x) - V_t^{\theta}(x)}{h} 
    &= \frac{1}{h}  \Big(\int_{A} G_t^{\theta + hr}(x,a) \pi_{t}^{\theta + hr}(x, \ud a)   
    - \int_{A} G_t^{\theta}(x,a) \pi_t^{\theta}(x,\ud a) \Big) \\
    &= \int_{A} \frac{G_t^{\theta + hr}(x,a)- G_t^{\theta}(x,a)}{h} \pi_t^{\theta} (x,\ud a) + \int_{A} G_t^{\theta + hr}(x,a) 
    \frac{ \pi_t^{\theta + hr}(x,\ud a) - \pi_t^{\theta}(x,\ud a)}{h}.
    \end{split}
\end{equation*}
Fix $t\in\mathcal T$ and assume that the result holds at time $t+1$, i.e., there exists by induction a map $D_{t+1}^{\theta}(\cdot)[r]\in \mathcal{C}(\mathcal{X};\mathbb{R})$ such that 
\small
\begin{align}
\label{eq : induction_hypothesis_value_function}
\left\|\frac{V^{\theta+hr}_{t+1} - V^\theta_{t+1}}{h}\;-\;D_{t+1}^{\theta}(\cdot)[r]\right\|_\infty\;\xrightarrow[h\to 0]{}\;0,
~ \text{and}~~
\sup_{|h|\le h_0}\left\|\frac{V^{\theta+hr}_{t+1} - V^\theta_{t+1}}{h}\right\|_\infty<+\infty
\quad\text{for some }h_0>0.
\end{align}
\normalsize
 The induction hypothesis \eqref{eq : induction_hypothesis_value_function} at index $t+1$ yields
$$V_{t+1}^{\theta+hr}=V_{t+1}^\theta+ h D_{t+1}^{\theta}(\cdot)[r] +o(h)\quad\text{in }\|\cdot\|_\infty,$$
together with the uniform bound $\sup_{|h|\le h_0}\|(V_{t+1}^{\theta+hr}-V_{t+1}^\theta)/h\|_\infty<+\infty$. Consider the map $F_t^\theta$ 
$$F_t^\theta(\lambda;x,a)
:=
\E_{X\sim\P_t^0(x,a,\cdot)}\big[-\Fc^{\lambda}(-f(t,x,a,\cdot)-V^\theta_{t+1}(\cdot))(X)\big]-\varepsilon^q\lambda,
\quad \forall \lambda\ge 0,$$
 and recall that
$$G_t^\theta(x,a)
=
\sup_{\lambda\in[0,\Lambda]}
F_t^\theta(\lambda;x,a).$$
The $1$-Lipschitz property of
$u\mapsto-\mathcal F^\lambda(-u)$ with respect to the supremum norm implies that
the $o(h)$ remainder above may be passed through the inner envelope. Hence, for
each fixed $\lambda$,
$$
\frac{
F_t^{\theta+hr}(\lambda;x,a)-F_t^\theta(\lambda;x,a)
}{h}
\longrightarrow
D_\theta F_t^\theta(\lambda;x,a)[r],
$$
where
$$
D_\theta^+F_t^\theta(\lambda;x,a)[r]
=
\int_{\mathcal X}
\inf_{y\in Y_{t,\theta,\lambda}^*(x,a;X)}
D_\theta V_{t+1}^\theta(y)[r]
\,\P_t^0(x,a,\ud X)
$$
for the right derivative, and
$$
D_\theta^-F_t^\theta(\lambda;x,a)[r]
=
\int_{\mathcal X}
\sup_{y\in Y_{t,\theta,\lambda}^*(x,a;X)}
D_\theta V_{t+1}^\theta(y)[r]
\,\P_t^0(x,a,\ud X)
$$
for the left derivative.

Since, by Lemma \ref{lem:compact_dual_maximizer_set}, the supremum over
$\lambda\ge0$ may be restricted to $[0,\Lambda]$. By the equicontinuity assumption, the quotients $\lambda\mapsto
\frac{
F_t^{\theta+hr}(\lambda;x,a)-F_t^\theta(\lambda;x,a)
}{h}$
are equicontinuous on the compact interval $[0,\Lambda]$. Since this pointwise limit is continuous in
$\lambda$, the convergence is uniform on $[0,\Lambda]$. Therefore, the envelope
argument over $\lambda$ (see Theorem \ref{def : danskin_Theorem}) gives the pointwise derivative
$$
D_\theta^+G_t^\theta(x,a)[r]
=
\sup_{\lambda\in\Lambda_t^{*,\theta}(x,a)}
D_\theta^+F_t^\theta(\lambda;x,a)[r],
$$
and similarly
$$
D_\theta^-G_t^\theta(x,a)[r]
=
\inf_{\lambda\in\Lambda_t^{*,\theta}(x,a)}
D_\theta^-F_t^\theta(\lambda;x,a)[r].
$$
Moreover, the Lipschitz property of $u\mapsto-\mathcal F^\lambda(-u)$ and the
uniform bound on the value-function quotients imply that, for $h$ sufficiently
small,
$$
\sup_{|h|\le h_0}\sup_{(x,a)\in \mathcal{X} \times A}
\left|\frac{G_t^{\theta+hr}(x,a)-G_t^\theta(x,a)}{h}\right|<\infty.
$$
Hence, for every fixed $x\in\mathcal X$, the dominated convergence theorem,
applied with respect to the fixed measure $\pi_t^\theta(x,\ud a)$, gives
$$
\int_{A}\frac{G_t^{\theta+hr}(x,a)-G_t^\theta(x,a)}{h}\,
\pi^{\theta}_t(x,\ud  a)
\xrightarrow[h\to 0]{}
\int_{A}D_\theta G_t^\theta(x,a)[r]\,\pi^\theta_t(x,\ud a).
$$
Moreover,
$$
\sup_{|h|\le h_0}\left\|
\int_{A}\frac{G_t^{\theta+hr}(x,a)-G_t^\theta(x,a)}{h}\,
\pi^{\theta}_t(x,\ud  a)
\right\|_\infty<+\infty.
$$

For the second term, Assumption \ref{assump:regularity_pi_theta} (ii) and the bounded-score condition imply by an application of the dominated convergence theorem the standard identity
$$
\frac{\pi_t^{\theta+hr}(x,\ud a)-\pi_t^\theta(x,\ud a)}{h}
\to
\big\langle \nabla_\theta\log\pi_t^\theta(x,a),r\big\rangle\,\pi_t^\theta(x,\ud a),
$$
in the weak sense uniformly in $x$. In addition, the Lipschitz property of
$u\mapsto-\mathcal F^\lambda(-u)$ implies
$$
\|G_t^{\theta+hr}-G_t^\theta\|_\infty
\le
\|V_{t+1}^{\theta+hr}-V_{t+1}^\theta\|_\infty
=O(h).
$$
Thus $G_t^{\theta+hr}\to G_t^\theta$ uniformly on $\mathcal X\times A$. Since
$G_t^\theta$ is bounded and $\|\nabla_\theta\log\pi_t^\theta\|$ is uniformly
bounded, we have
$$
\int_A G_t^{\theta+hr}(x,a)\,
\frac{\pi_t^{\theta+hr}(x,\ud a)-\pi_t^\theta(x,\ud a)}{h}
\xrightarrow[h\to0]{}
\int_A G_t^\theta(x,a)
\big\langle \nabla_\theta\log\pi_t^\theta(x,a), r \big\rangle
\pi_t^\theta(x,\ud a),
$$
and
$$
\sup_{|h|\le h_0}\left\|
\int_A G_t^{\theta+hr}(\cdot,a)\,
\frac{\pi_t^{\theta+hr}(\cdot,\ud a)-\pi_t^\theta(\cdot,\ud a)}{h}
\right\|_\infty<\infty.
$$
Collecting terms yields pointwise convergence in $x$ of
$\frac{V_t^{\theta+hr}-V_t^\theta}{h}$ to
\begin{align}
    D_\theta V_t^\theta(x)[r]
    :=
    \int_{A} D_\theta G_t^\theta(x,a)[r] \pi_t^{\theta}(x,\ud a)
    +
    \int_{A} G_t^{\theta}(x,a)
    \big\langle  \nabla_{\theta}\log \pi_t^{\theta}(x,a), r \big\rangle
    \pi_t^{\theta}(x,\ud a).
    \notag
\end{align}
The continuity of $D_\theta V_t^\theta(\cdot)[r]$ will be obtained below from
the equicontinuity of the difference quotients.

It remains to upgrade this pointwise convergence to uniform convergence. For the
chosen one-sided limit, set
$$
Q_h(x)
:=
\frac{V_t^{\theta+hr}(x)-V_t^\theta(x)}{h}.
$$
By assumption, the family $(Q_h)_{0<|h|\le h_0}$ is equicontinuous on the compact
set $\mathcal X$.

First, we prove that the limit is continuous. Let $(x_n)_{n\ge1}\subset\mathcal X$ be an arbitrary sequence such that
$x_n\to x\in\mathcal X$. Given $\eta>0$, choose
$\delta>0$ such that
$$
\big|Q_h(z)-Q_h(z')\big|<\eta
$$
whenever $\|z-z'\|<\delta$, uniformly over $0<|h|\le h_0$. For $n$ large enough,
$\|x_n-x\|<\delta$. Passing to the limit $h\to0$ in $|Q_h(x_n)-Q_h(x)|<\eta$ gives
$$
\big|D_\theta V_t^\theta(x_n)[r]-D_\theta V_t^\theta(x)[r]\big|
\le \eta.
$$
Thus
$D_\theta V_t^\theta(\cdot)[r]\in C(\mathcal X;\mathbb R).$

We now prove uniform convergence. Suppose, by contradiction, that the convergence
is not uniform. Then there exist $\eta>0$, a sequence $(h_n)_{n\ge1}$ converging to $0$ from
the chosen side, with $0<|h_n|\le h_0$, and a sequence
$(x_n)_{n\ge1}\subset\mathcal X$ such that, for every $n\ge1$, $\big|Q_{h_n}(x_n)-D_\theta V_t^\theta(x_n)[r]\big|\ge\eta.$
By compactness of $\mathcal X$, there exists a subsequence, still denoted
$(x_n)_{n\ge1}$ for simplicity, and a point $x\in\mathcal X$ such that $x_n\to x.
$ By equicontinuity of $Q_h$, $\big|Q_{h_n}(x_n)-Q_{h_n}(x)\big|\to0.$ By continuity of $D_\theta V_t^\theta(\cdot)[r]$,
$$
\big|D_\theta V_t^\theta(x_n)[r]-D_\theta V_t^\theta(x)[r]\big|\to0.
$$
Finally, $Q_{h_n}(x)\to D_\theta V_t^\theta(x)[r].$ Combining these three limits gives $Q_{h_n}(x_n)-D_\theta V_t^\theta(x_n)[r]\to0,$
contradicting the lower bound by $\eta$. Hence
$$
\Big\|
Q_h-D_\theta V_t^\theta(\cdot)[r]
\Big\|_\infty
\to0.
$$
The uniform boundedness of the quotients follows from the uniform convergence
above and from the boundedness of the continuous limit on compact $\mathcal X$.
This completes the induction.
\end{proof}

\subsection{Proofs of vector gradient representation}\label{app: gradient valued}
\begin{proof}[Proof of Lemma \ref{lem_vector_grad1}]
We prove the result by backward induction. At terminal time, $V_T^\theta=g$ does not depend on $\theta$, so
$$
\nabla_\theta V_T^\theta(x)=0,
\quad x\in\Xc.
$$

Assume that, for some $t\in\Tc$, the map $\theta\mapsto V_{t+1}^\theta(y)$ is differentiable for every $y\in\Xc$ and that $y\mapsto\nabla_\theta V_{t+1}^\theta(y)$ is continuous on $\Xc$. We denote the unique elements of $\Lambda_t^{*,\theta}(x,a)$ and $Y_{t,\theta,\lambda_t^*}^*(x,a;X)$, for every $(x,a,X)\in\Xc\times A\times\Xc$, by
$$
\Lambda_t^{*,\theta}(x,a)=\{\lambda_t^*(x,a,\theta)\}
$$
and
$$
Y_{t,\theta,\lambda_t^*(x,a,\theta)}^*(x,a;X)
=
\{y_t^*(x,a,X,\lambda_t^*(x,a,\theta),\theta)\}.
$$
For notational simplicity, set
$$
\bar y_t^*(x,a,X,\theta)
:=
y_t^*(x,a,X,\lambda_t^*(x,a,\theta),\theta),
\quad \forall(x,a,X)\in\Xc\times A\times\Xc.
$$
Therefore the right and left formulas in Theorem \ref{gradient J and F} coincide. Indeed, for every $(x,a)\in\Xc\times A$ and every direction $r\in\R^{d}$,
\begin{align*}
D_\theta^+G_t^\theta(x,a)[r]
= D_\theta^-G_t^\theta(x,a)[r] =
\E_{X\sim \P_t^0(x,a,\cdot)}
\left[
D_\theta V_{t+1}^\theta(\bar y_t^*(x,a,X,\theta))[r]
\right].
\end{align*}
By the induction hypothesis, for every $y\in\Xc$ and every $r\in\R^{d}$,
$$
D_\theta V_{t+1}^\theta(y)[r]
=
\left\langle\nabla_\theta V_{t+1}^\theta(y),r\right\rangle.
$$
Hence the right and left directional derivatives of $G_t^\theta(x,a)$ coincide and are linear in $r$. Thus, for every $(x,a)\in\Xc\times A$ and every $r\in\R^{d}$,
$$
D_\theta G_t^\theta(x,a)[r]
=
\left\langle
\E_{X\sim \P_t^0(x,a,\cdot)}
\left[
\nabla_\theta V_{t+1}^\theta(\bar y_t^*(x,a,X,\theta))
\right],r
\right\rangle.
$$
Therefore, for every $(x,a)\in\Xc\times A$,
$$
\nabla_\theta G_t^\theta(x,a)
=
\E_{X\sim \P_t^0(x,a,\cdot)}
\left[
\nabla_\theta V_{t+1}^\theta(\bar y_t^*(x,a,X,\theta))
\right].
$$

It remains to differentiate the policy expectation. For every $x\in\Xc$,
$$
V_t^\theta(x)=\int_A G_t^\theta(x,a)\pi_t^\theta(x,\mathrm da).
$$
Using the score-function identity and the bounded-score assumption, for every $x\in\Xc$ and every direction $r\in\R^{d}$,
\begin{align*}
D_\theta V_t^\theta(x)[r]
&=
\int_A
D_\theta G_t^\theta(x,a)[r]\pi_t^\theta(x,\mathrm da)
+
\int_A
G_t^\theta(x,a)
\left\langle\nabla_\theta\log\pi_t^\theta(x,a),r\right\rangle
\pi_t^\theta(x,\mathrm da)\\
&=
\left\langle
\E_{a\sim\pi_t^\theta(x,\cdot)}
\left[
\nabla_\theta G_t^\theta(x,a)
+
G_t^\theta(x,a)\nabla_\theta\log\pi_t^\theta(x,a)
\right],r
\right\rangle.
\end{align*}
Thus $\theta\mapsto V_t^\theta(x)$ is differentiable for every $x\in\Xc$, and
$$
\nabla_\theta V_t^\theta(x)
=
\E_{a\sim\pi_t^\theta(x,\cdot)}
\left[
G_t^\theta(x,a)\nabla_\theta\log\pi_t^\theta(x,a)
+
\nabla_\theta G_t^\theta(x,a)
\right],
\forall x\in\Xc.
$$
Finally, since $\mu_0$ is independent of $\theta$ and the gradient is continuous and bounded on the compact space $\Xc$,
$$
\nabla_\theta J(\theta)
=
\E_{X_0\sim\mu_0}\left[\nabla_\theta V_0^\theta(X_0)\right].
$$
Thus the common directional derivative is linear and continuous in $r$, hence it is represented by a vector in $\R^d$. The induction is complete.
\end{proof}
\section{Convexity and uniqueness of selectors}
\label{app:unique}
\subsection{Uniqueness of the inner selector}
\label{subsec:convexity_uniqueness_inner_selector}

We now give sufficient conditions under which the inner selector is unique. 

\begin{Assumption}
\label{assump:convexity_structure}
The following convexity conditions hold.

\begin{enumerate}
\item[(i)] The state space $\mathcal X\subset\mathbb R^d$ is compact and convex.

\item[(ii)] The terminal reward $g$ is convex on $\mathcal X$.

\item[(iii)] For every $t\in\mathcal T$ and every $a\in A$, the map
$$
(x,y)\mapsto f(t,x,a,y)
$$
is continuous and jointly convex on $\mathcal X\times\mathcal X$.

\item[(iv)] The transport cost $c:\mathcal X\times\mathcal X\to\mathbb R_+$
is continuous and jointly convex.

\item[(v)] For every $t\in\mathcal T$ and every fixed $a\in A$, the reference
kernel $\P_t^0$ preserves convexity in the following sense: whenever $\Phi:\mathcal X\times\mathcal X\to\mathbb R$
is continuous and jointly convex, the map
$$
x\mapsto
\int_{\mathcal X}\Phi(x,x')\,\P_t^0(x,a,\ud x')
$$
is convex on $\mathcal X$.

\item[(vi)] For every $t\in\mathcal T$ and $\theta\in\Theta$, the policy
averaging operator preserves convexity in $x$: whenever $\Phi:\mathcal X\times A\to\mathbb R$
is continuous and $x\mapsto \Phi(x,a)$ is convex for every fixed $a\in A$, the
map
$$
x\mapsto
\int_A \Phi(x,a)\,\pi_t^\theta(x,\ud a)
$$
is convex on $\mathcal X$.
\end{enumerate}
\end{Assumption}

\begin{Proposition}
\label{prop:convexity_propagation}
Fix $\theta\in\Theta$ and assume Assumption \ref{assump:convexity_structure}.
Then, for every $t\in\bar{\mathcal T}$, the robust value function
$V_t^\theta$ is convex on $\mathcal X$.
\end{Proposition}

\begin{proof}
We argue by backward induction. At terminal time, $V_T^\theta=g$, which is
convex by Assumption \ref{assump:convexity_structure}.

Assume that $V_{t+1}^\theta$ is convex on $\mathcal X$. Fix $a\in A$ and
$\lambda\ge0$. The map
$$
(x,x',y)\mapsto
f(t,x,a,y)+V_{t+1}^\theta(y)+\lambda c(x',y)
$$
is jointly convex. Indeed, $(x,y)\mapsto f(t,x,a,y)$ is jointly convex,
$V_{t+1}^\theta$ is convex by the induction hypothesis, and $c$ is jointly
convex. Since partial minimization over a convex set preserves convexity,
$$
(x,x')\mapsto \inf_{y\in\mathcal X}
\left\{
f(t,x,a,y)+V_{t+1}^\theta(y)+\lambda c(x',y)
\right\}
$$
is jointly convex on $\Xc\times\Xc$. By Assumption \ref{assump:convexity_structure}(v), the reference kernel preserves
convexity in $x$. Hence,
$$
x\mapsto
\int_{\mathcal X}
\inf_{y\in\mathcal X}
\left\{
f(t,x,a,y)+V_{t+1}^\theta(y)+\lambda c(x',y)
\right\}\,\P_t^0(x,a,\ud x')
$$
is convex on $\mathcal X$. Subtracting the constant $\varepsilon^q\lambda$ does
not affect convexity. Therefore, for every fixed $a\in A$ and
$\lambda\ge0$,
$$
x\mapsto
\int_{\mathcal X}
\inf_{y\in\mathcal X}
\left\{
f(t,x,a,y)+V_{t+1}^\theta(y)+\lambda c(x',y)
\right\}\,\P_t^0(x,a,\ud x')
-
\varepsilon^q\lambda
$$
is convex. Taking the supremum over $\lambda\ge0$ preserves convexity. By the DPP,
$$
V_t^\theta(x)
=
\int_A \sup_{\lambda\ge 0}\Big(
\E_{X\sim \P_t^0(x,a,\cdot)}\big[-\Fc^{\lambda}(-f(t,x,a,\cdot)-V_{t+1}^\theta(\cdot))(X)\big]
-\varepsilon^q\lambda
\Big)\,\pi_t^\theta(x,\ud a).
$$
Since the policy averaging operator preserves convexity in $x$ by Assumption
\ref{assump:convexity_structure}(vi),
$V_t^\theta$ is convex. 
\end{proof}

\begin{Corollary}
\label{cor:unique_inner_selector}
Assume Assumption \ref{assump:convexity_structure}. Suppose in addition that,
for every $x\in\mathcal X$, the map $y\mapsto c(x,y)$ is strictly convex on $\mathcal X$. Then, for every $t\in\mathcal T$,
$\theta\in\Theta$, $(x,a,x')\in\mathcal X\times A\times\mathcal X$, and
$\lambda>0$, the set
$$
Y_{t,\theta,\lambda}^*(x,a;x')
=
\argmin_{y\in\mathcal X}
\left\{
f(t,x,a,y)+V_{t+1}^\theta(y)+\lambda c(x',y)
\right\}
$$
is a singleton. 
\end{Corollary}

\begin{proof}
By Proposition \ref{prop:convexity_propagation}, $V_{t+1}^\theta$ is convex on
$\mathcal X$. By Assumption \ref{assump:convexity_structure}(v), for every fixed
$(t,x,a)\in X\times A$, the map $y\mapsto f(t,x,a,y)$ is convex. Hence
$$
y\mapsto f(t,x,a,y)+V_{t+1}^\theta(y)
$$
is convex. By the additional assumption, $y\mapsto c(x',y)$ is strictly convex. Therefore,
for every $\lambda>0$,
$$
y\mapsto
f(t,x,a,y)+V_{t+1}^\theta(y)+\lambda c(x',y)
$$
is strictly convex on $\mathcal X$. The objective is continuous and $\mathcal X$ is compact, so the minimum is
attained. Since the objective is strictly convex on the convex set
$\mathcal X$, the minimizer is unique. Hence
$Y_{t,\theta,\lambda}^*(x,a;x')$ is a singleton for every $\lambda>0$.
\end{proof}

\begin{Remark}
For the Euclidean Wasserstein cost $c(x,y)=\|x-y\|^q$, the map
$y\mapsto c(x,y)$ is strictly convex whenever $q>1$. In particular, for $q=2$,
it is strongly convex. Hence Corollary \ref{cor:unique_inner_selector} applies
at every positive dual maximizer. If $q=1$, strict convexity generally fails,
and uniqueness of the inner selector must come from additional strict convexity
of $y\mapsto f(t,x,a,y)+V_{t+1}^\theta(y)$.
\end{Remark}

\subsection{Uniqueness of the dual selector}
\label{subsec:uniqueness_dual_selector}

Fix $t\in\mathcal T$, $\theta\in\Theta$, and
$(x,a)\in\mathcal X\times A$. Throughout this subsection, we assume that the results of the previous subsection
give uniqueness of the inner selector on $ (0,\Lambda]$. 

\begin{Lemma}
\label{lem:no_flat_dual_uniqueness}
 $\Lambda_t^{*,\theta}(x,a)$ is a singleton if and only if
there do not exist $\lambda_0<\lambda_1$ such that
$$
[\lambda_0,\lambda_1]\subseteq \Lambda_t^{*,\theta}(x,a).
$$
\end{Lemma}

\begin{proof}
If the maximizer is unique, the claim is immediate. Conversely, suppose
$\lambda_0<\lambda_1$ are two maximizers. Since $\lambda\mapsto F_t^\theta(\lambda;x,a)$ is concave on
$[0,\Lambda]$ and both endpoints attain the maximum, every point of
$[\lambda_0,\lambda_1]$ also attains the maximum. Hence the maximizer set
contains a nontrivial interval.
\end{proof}

\begin{Proposition}
\label{prop:strict_concavity_moving_minimizers}
Assume that, for every $\lambda_0\ne\lambda_1$ in $ (0,\Lambda]$,
$$
\P_t^0
\left(
x,a,
\left\{
x'\in\mathcal X:
y_{t,\theta,\lambda_0}(x,a;x')
\ne
y_{t,\theta,\lambda_1}(x,a;x')
\right\}
\right)>0.
$$
Then $F_t^\theta(\cdot;x,a)$ is strictly concave on $ (0,\Lambda]$. In particular, it has at
most one maximizer on $ (0,\Lambda]$.
\end{Proposition}

\begin{proof}
For fixed $x'\in\mathcal X$, the inner value
$$
\lambda\mapsto
\phi_{x'}(\lambda):=\inf_{y\in\mathcal X}
\left\{
f(t,x,a,y)+V_{t+1}^\theta(y)+\lambda c(x',y)
\right\}
$$
is concave, as an infimum of affine functions of $\lambda$. Fix $\lambda_0\ne\lambda_1$ in $ (0,\Lambda]$, $s\in(0,1)$, and
$\lambda_s=(1-s)\lambda_0+s\lambda_1$. If equality holds in
$$
\phi_{x'}(\lambda_s)
\ge
(1-s)\phi_{x'}(\lambda_0)+s\phi_{x'}(\lambda_1),
$$
then the minimizer at $\lambda_s$ is also a minimizer at both $\lambda_0$ and
$\lambda_1$. Since the inner minimizer is unique on $ (0,\Lambda]$, this implies
$$
y_{t,\theta,\lambda_0}(x,a;x')
=
y_{t,\theta,\lambda_1}(x,a;x').
$$
Thus the inequality is strict on the set where the two selectors differ. Integrating with respect to $\P_t^0(x,a,\cdot)$ gives strict concavity of
$$
\lambda\mapsto
\int_{\mathcal X}\phi_{x'}(\lambda)\,\P_t^0(x,a,\ud x'),
$$
because the strict inequality holds on a set of positive
$\P_t^0(x,a,\cdot)$-mass by assumption. The affine term $-\varepsilon^q\lambda$ does not
affect strict concavity, so $F_t^\theta(\cdot;x,a)$ is strictly concave on
$(0,\Lambda]$.
\end{proof}

\begin{Proposition}
\label{prop:primitive_selector_movement}
Define
$$
H^\theta(t,x,a,y)
:=
f(t,x,a,y)+V_{t+1}^\theta(y).
$$
Suppose that:

\begin{enumerate}
\item[(i)] Assumptions \ref{assump:convexity_structure} hold, i.e, $y\mapsto H^\theta(t,x,a,y)$ is convex on $\mathcal X$ and has a unique
minimizer
$$
\argmin_{y\in\mathcal X}H^\theta(t,x,a,y)
=
\{\bar y(t,x,a,\theta)\}.
$$

\item[(ii)] $y\mapsto H^\theta(t,x,a,y)\in C^1(\operatorname{int}(\mathcal X))$.

\item[(iii)] For every $x'\in\mathcal X$ and every $\lambda\in  (0,\Lambda]$, we have that
$y_{t,\theta,\lambda}(x,a;x')\in\operatorname{int}(\mathcal X)$.

\item[(iv)] $\P_t^0(x,a,\cdot)$ is non-atomic:
$$
\P_t^0(x,a,y)=0,
\quad
\forall y\in\mathcal X.
$$
\end{enumerate}
Then, for every $\lambda_0\ne\lambda_1$ in $ (0,\Lambda]$,
$$
\P_t^0
\left(
x,a,
\left\{
x'\in\Xc:
y_{t,\theta,\lambda_0}(x,a;x')
\ne
y_{t,\theta,\lambda_1}(x,a;x')
\right\}
\right)>0.
$$
More precisely,
$$
\left\{
x'\in \Xc:
y_{t,\theta,\lambda_0}(x,a;x')
=
y_{t,\theta,\lambda_1}(x,a;x')
\right\}
\subseteq
\{\bar y(t,x,a,\theta)\}.
$$
\end{Proposition}

\begin{proof}
Fix $\lambda_0\ne\lambda_1$ in $ (0,\Lambda]$ and suppose that, for some $x'\in\mathcal X$,
$$
y_{t,\theta,\lambda_0}(x,a;x')
=
y_{t,\theta,\lambda_1}(x,a;x')
=:y.
$$
By assumption, $y\in\operatorname{int}(\mathcal X)$. The first-order optimality
conditions give
$$
\nabla H^\theta(t,x,a,y)
+
\lambda_i\nabla_y\|X-y\|^q
=
0,
\quad i=\{0,1\}.
$$
Subtracting yields
$$
(\lambda_0-\lambda_1)\nabla_y\|x'-y\|^q=0.
$$
Since $\lambda_0\ne\lambda_1$,
$$
\nabla_y\|x'-y\|^q=0.
$$
For $q>1$, this is equivalent to $y=x'$. Hence $x'=y$. Returning to the first-order
condition gives
$$
\nabla H^\theta(t,x,a,x')=0.
$$
Since $y\mapsto H^\theta(t,x,a,y)$ is convex, $x'$ minimizes $y\mapsto H^\theta(t,x,a,y)$. By
uniqueness of the minimizer, $x'=\bar y(t,x,a,\theta).$
Therefore the equality set is contained in
$\{\bar y(t,x,a,\theta)\}$. By non-atomicity, this singleton has zero
$\P_t^0(x,a,\cdot)$-mass. Hence the set where the two selectors differ has positive mass.
\end{proof}
\begin{Remark}
Assumption \ref{assump:convexity_structure}(ii) is stated directly for
readability. It can be verified from more explicit smoothness conditions: for
instance, if $g\in C^1(\operatorname{int}(\mathcal X))$,
$(x,y)\mapsto f(t,x,a,y)$ is $C^1$, the inner and dual optimizers are unique and
interior, and differentiation can be passed through the transition and action
integrals, then an envelope-theorem argument propagates $C^1$ regularity
backward through the Bellman recursion. We omit these stronger conditions to
keep the presentation concise.
\end{Remark}

\begin{Corollary}
\label{cor:primitive_dual_uniqueness}
Assume the conditions of Proposition \ref{prop:primitive_selector_movement}.
Then $F_t^\theta(\cdot;x,a)$ is strictly concave on $ (0,\Lambda]$. Consequently, if the
dual maximizer belongs to $ (0,\Lambda]$, then $\Lambda_t^{*,\theta}(x,a)$ is a singleton.
\end{Corollary}

\begin{proof}
The movement condition in
Proposition \ref{prop:strict_concavity_moving_minimizers} follows from
Proposition \ref{prop:primitive_selector_movement}. Hence
$F_t^\theta(\cdot;x,a)$ is strictly concave on $ (0,\Lambda]$. A strictly concave function
has at most one maximizer on a convex set.
\end{proof}

\begin{Remark}
Strict concavity is only sufficient, not necessary, for uniqueness of the dual
selector. Since $F_t^\theta(\cdot;x,a)$ is concave, uniqueness is equivalent to
the absence of a nontrivial flat interval of maximizers, as stated in
Lemma \ref{lem:no_flat_dual_uniqueness}.
\end{Remark}

\section{Additional Numerical Results}\label{app:num}
\subsection{Robust self-exciting Multi-armed bandits}
We next consider a self-exciting bandit, which extends the static robust
bandit rule to a finite-horizon Markovian setting. In a standard robust
bandit, the reward law of each arm $j\in\{1,\ldots,K\}$ is uncertain and
belongs to a Wasserstein ball around a nominal law
$\P_j^0\in\Pc_1(\mathbb R)$. For the $1$-Wasserstein distance on $\mathbb R$,
the worst-case mean reward of arm $j$ is
$$
\inf_{\Q\in\Pc_1(\mathbb R):\,\Wc_1(\Q,\P_j^0)\le \varepsilon_j}
\E_{R\sim\Q}[R]
=
\mu_j^0-\varepsilon_j,
\quad
\mu_j^0:=\E_{R\sim\P_j^0}[R],
$$
where $\varepsilon_j>0$ is the ambiguity radius. Thus the static robust decision
selects an arm maximizing $\mu_j^0-\varepsilon_j$.

We now introduce a dynamic version in which the success probability of an arm
depends on the previous outcome. Let $K\in\mathbb N^*$ be the number of arms,
$K_{\sup}\in\mathbb N^*$ the maximal stake, and $T\in\mathbb N^*$ the horizon,
with time index $t\in\Tc:=\{0,\ldots,T-1\}$. The state is
$$
x_t=(m_t,b_t)\in\Xc,
\qquad
\Xc
=
\{-K_{\sup},\ldots,-1,1,\ldots,K_{\sup}\}
\times
\{1,\ldots,K\},
$$
where $m_t\in\{-K_{\sup},\ldots,-1,1,\ldots,K_{\sup}\}$ is the signed outcome
of the previous investment and $b_t\in\{1,\ldots,K\}$ is the arm played in the
previous round. The action is
$$
a_t=(k_t,j_t)\in A,
\qquad
A=\{1,\ldots,K_{\sup}\}\times\{1,\ldots,K\},
$$
where $k_t\in\{1,\ldots,K_{\sup}\}$ is the stake and
$j_t\in\{1,\ldots,K\}$ is the selected arm.

Let $p_j\in(0,1)$ denote the base success probability of arm
$j\in\{1,\ldots,K\}$. Given $x_t=(m_t,b_t)\in\Xc$ and
$a_t=(k_t,j_t)\in A$, the success probability is
$$
\tilde p(x_t,a_t)
=
p_{j_t}
+
\lambda\,\operatorname{sign}(m_t)\,
\mathbbm{1}_{\{b_t=j_t\}},
$$
where $\lambda>0$ is the excitation parameter. We assume $0<\lambda<\min_{j\in\{1,\ldots,K\}}\{p_j,1-p_j\},$ so that $\tilde p(x_t,a_t)\in(0,1)$ for all $(x_t,a_t)\in\Xc\times A$. The next state is generated by
$$
\xi_t\sim\operatorname{Bernoulli}(\tilde p(x_t,a_t)),
\qquad
m_{t+1}=k_t(2\xi_t-1),
\qquad
b_{t+1}=j_t,
$$
where $\xi_t\in\{0,1\}$ is the success indicator. Thus
$x_{t+1}=(m_{t+1},b_{t+1})\in\Xc$. The one-step reward is the signed payoff
$$
f(t,x,a,x')=m',
\qquad
g(x)=0,
\quad
\forall (t,x,a,x')\in\Tc\times\Xc\times A\times\Xc,
$$
where $x'=(m',b')\in\Xc$ is the next state. Hence a successful play with stake
$k_t$ gives reward $+k_t$, while an unsuccessful play gives reward $-k_t$.

For the robust formulation, the reference transition kernel
$\P^0_t(x,a,\cdot)\in\Pc(\Xc)$ is the kernel induced by the nominal probabilities
$\tilde p(x,a)$. The agent evaluates each policy under the worst-case transition
kernel within a Wasserstein ball of radius $\varepsilon>0$ around
$\P_t^0(x,a,\cdot)$. As in the previous finite-state examples, we parametrize
the policy by a time-dependent tabular softmax,
$$
\pi_t^\theta(x,a)
=
\frac{\exp(\theta_{t,x,a})}
{\sum_{a'\in A}\exp(\theta_{t,x,a'})},
\quad
\forall(t,x,a)\in\Tc\times\Xc\times A,
$$
where
$\theta=(\theta_{t,x,a})_{(t,x,a)\in\Tc\times\Xc\times A}
\in\mathbb R^{|\Tc|\times|\Xc|\times|A|}.$
For illustration, we report the greedy action
$
a_t^*(x)=\argmax_{a\in A}\pi_t^\theta(x,a),$ for all $(t,x)\in\Tc\times\Xc,$
extracted from the learned stochastic policy. Figure \ref{fig:bandits} compares
the greedy actions learned by Algorithm \ref{alg:robust_actor_critic} with the
exact robust solution. The exact solution is computed by backward induction:
starting from the terminal condition, we sweep backward in time and, at each
state, evaluate the robust Bellman operator over all feasible actions.

\begin{figure}[H]
    \centering
    \includegraphics[width=0.4\textwidth]{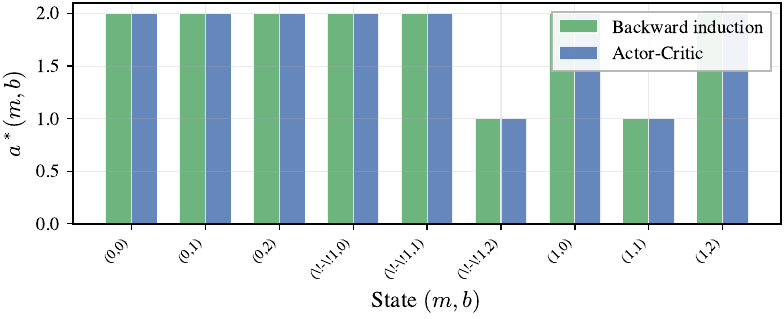}
    \caption{%
        Greedy actions at $t=0$ for $\varepsilon=0.3$ at every state $(m,b)$.
    }
    \label{fig:bandits}
\end{figure}
Both methods agree at every state $(m,b)$: the agent predominantly
selects arm $2$ (highest base probability) but switches to arm $1$
at states where the self-exciting feedback makes it locally preferable.
\subsection{Exact dynamic-programming recovery}
\label{subsec:multiseed_dp}

As a preliminary validation, Table \ref{tab:dp_recovery} reports recovery of exact robust dynamic programming over $10$ seeds. The value error is
$$\Delta_V:=\|V_0^{\theta^\star}-V_0^{\rm DP}\|_\infty,$$
where $V_0^{\theta^\star}$ is the value of the learned policy and $V_0^{\rm DP}$ is the exact robust-DP value. The policy error is the greedy-action mismatch
$$\Delta_\pi
:=
\frac{1}{T|\mathcal X|}
\sum_{t,x}
\mathds{1}_{\{\hat a_t(x)\neq a_t^{\rm DP}(x)\}}.$$

{\small\begin{table}[H]
\centering
\caption{Exact robust-DP recovery over $10$ seeds. $\Delta_V=\|V_0^{\theta^\star}-V_0^{\rm DP}\|_\infty$ is the value error and $\Delta_\pi$ is the greedy-policy mismatch against exact robust DP.}
\label{tab:dp_recovery}
\small
\begin{tabular}{lccc}
\toprule
Environment & $\varepsilon$ & $\Delta_V \downarrow$ & $\Delta_\pi \downarrow$ \\
\midrule
Coin toss    & $0.5$ & $0.0012$ & $0.0000$ \\
Coin toss    & $1.0$ & $0.0017$ & $0.0000$ \\
Coin toss    & $2.0$ & $0.0021$ & $0.0000$ \\
Supply chain & $0.5$ & $0.0069$ & $0.0000$ \\
Supply chain & $1.0$ & $0.0114$ & $0.0000$ \\
Supply chain & $2.0$ & $0.0264$ & $0.0727$ \\
\bottomrule
\end{tabular}
\end{table}}
\subsection{Role of the robust-sensitivity term}
 We compare the full recursion with a naive robust actor--critic variant that retains only the score-function term and drops the robust-sensitivity correction. Concretely, the full recursion uses the target
\[
\widehat z_t
=
\widehat G_t^\theta(x_t,a_t)\nabla_\theta\log\pi_t^\theta(x_t,a_t)
+
\widehat{\nabla_\theta G_t^\theta}(x_t,a_t),
\]
whereas the naive variant uses
\[
\widehat z_t^{\,\rm naive}
=
\widehat G_t^\theta(x_t,a_t)\nabla_\theta\log\pi_t^\theta(x_t,a_t).
\]
This comparison tests whether the Bellman-sensitivity correction is needed to recover the exact robust dynamic-programming solution.

{\small\begin{table}[H]
\centering
\caption{Robust-sensitivity term: policy mismatch
$\Delta_\pi$ to exact Wasserstein dynamic programming for the full
recursion. Mean $\pm$ standard error over $5$ seeds; lower
is better.}
\label{tab:ablation_dpi}\label{tab:ablation_dpi2}
\small
\begin{tabular}{llcc}
\toprule
Environment & Method & $\varepsilon$ & $\Delta_\pi \downarrow$ \\
\midrule
Coin toss & Full & $0.5$ & $0.000\pm0.000$ \\
Coin toss & Naive & $0.5$ & $0.638\pm0.023$ \\
Coin toss & Full & $1.0$ & $0.000\pm0.000$ \\
Coin toss & Naive & $1.0$ & $0.620\pm0.022$ \\
Coin toss & Full & $2.0$ & $0.000\pm0.000$ \\
Coin toss & Naive & $2.0$ & $0.620\pm0.017$ \\
\midrule
Supply chain & Full & $0.5$ & $0.000\pm0.000$ \\
Supply chain & Naive & $0.5$ & $0.855\pm0.020$ \\
Supply chain & Full & $1.0$ & $0.000\pm0.000$ \\
Supply chain & Naive & $1.0$ & $0.855\pm0.020$ \\
Supply chain & Full & $2.0$ & $0.073\pm0.000$ \\
Supply chain & Naive & $2.0$ & $0.862\pm0.015$ \\
\bottomrule
\end{tabular}
\end{table}}

\subsection{Inner-optimization sensitivity}
\label{subsec:inner_sensitivity}
The robust Bellman backup requires an inner maximization over the dual variable
$\lambda$, which we approximate on a grid with $n_\lambda$ points. We vary
$n_\lambda$ to test the effect of this discretization on the objective, policy,
gradient estimate, and runtime. The finest grid, $n_\lambda=200$, is used as the
reference. We also report the relative gradient error
$$\frac{\|\nabla_\theta J-\nabla_\theta J^{\rm FD}\|_2}
{\|\nabla_\theta J^{\rm FD}\|_2},$$
where $\nabla_\theta J^{\rm FD}$ is the central finite-difference estimate.

{\small\begin{table}[H]
\centering
\caption{
Sensitivity to the dual-grid resolution $n_\lambda$. $J$ is the robust
objective and $\Delta_\pi$ is the greedy-policy mismatch relative to the reference policy. Time/iter. is the average wall-clock time per iteration.
}
\label{tab:inner_opt}
\small
\begin{tabular}{ccccc}
\toprule
$n_\lambda$ & $J \uparrow$ & $\Delta_\pi$ vs. ref. $\downarrow$
& Gradient rel.~$\ell_2 \downarrow$ & Time/iter. $\downarrow$ \\
\midrule
$5$   & $+0.1028$ & $0.000$ & $1.46\times10^{-10}$ & $4.0$ ms \\
$10$  & $+0.1305$ & $0.000$ & $1.44\times10^{-10}$ & $5.3$ ms \\
$25$  & $+0.1473$ & $0.000$ & $1.45\times10^{-10}$ & $8.9$ ms \\
$50$  & $+0.1541$ & $0.000$ & $1.46\times10^{-10}$ & $13.3$ ms \\
$100$ & $+0.1545$ & $0.000$ & $1.42\times10^{-10}$ & $24.7$ ms \\
$200$ & $+0.1554$ & $0.000$ & $1.41\times10^{-10}$ & $44.2$ ms \\
\bottomrule
\end{tabular}
\end{table}}
The reported gradient error measures consistency with finite differences computed using the same discretized dual grid. It is not an error relative to the continuous dual problem.
\subsection{Additional multi-dimensional LQ results}
\label{appendix:multidim_lq}

As a complementary check, we evaluate nominal and robust Riccati
controllers on random LQ instances under a perturbation grid. The robust controller is indexed
by a robustness level $\alpha$, with $\lambda=\alpha\bar\lambda$. For each
controller $\pi$, let $C_{\rm nom}(\pi)$ be the average quadratic cost under
nominal dynamics and let $C_{\rm worst}(\pi)$ be the largest cost over the
perturbation grid. We report paired differences
$$\Delta C_{\rm worst}
=
C_{\rm worst}(\pi_{\rm robust})-C_{\rm worst}(\pi_{\rm nominal}),
\quad
\Delta C_{\rm nom}
=
C_{\rm nom}(\pi_{\rm robust})-C_{\rm nom}(\pi_{\rm nominal}).$$
Costs are reported, so lower is better and negative $\Delta C$ means that the
robust controller improves over the nominal controller. For each dimension,
results are paired over the same $10$ random instances.

{\small\begin{table}[H]
\centering
\caption{
Paired high-dimensional LQ stress test. We report paired differences
$\Delta C=C_{\rm robust}-C_{\rm nominal}$ over $10$ random instances per
dimension.
}
\label{tab:lq_paired}
\small
\begin{tabular}{ccccc}
\toprule
$d$ & Robust setting & $\Delta C_{\rm worst}\downarrow$ & $\Delta C_{\rm nom}$ & Sign agreement \\
\midrule
$2$  & $\alpha=10$ & $-0.080\pm0.024$ & $+0.008\pm0.004$ & $10/10$ \\
$2$  & $\alpha=25$ & $-0.033\pm0.009$ & $+0.000\pm0.000$ & $10/10$ \\
$5$  & $\alpha=5$  & $-0.227\pm0.071$ & $+0.090\pm0.037$ & $10/10$ \\
$5$  & $\alpha=10$ & $-0.140\pm0.042$ & $+0.014\pm0.006$ & $10/10$ \\
$10$ & $\alpha=5$  & $-0.310\pm0.075$ & $+0.017\pm0.022$ & $10/10$ \\
$10$ & $\alpha=10$ & $-0.166\pm0.038$ & $-0.006\pm0.010$ & $10/10$ \\
$50$ & $\alpha=5$  & $-0.482\pm0.050$ & $+0.060\pm0.020$ & $10/10$ \\
$50$ & $\alpha=25$ & $-0.103\pm0.010$ & $+0.000\pm0.000$ & $10/10$ \\
\bottomrule
\end{tabular}
\end{table}}

\subsection{Learning Curves}

\begin{figure}[H]
    \centering

    \begin{minipage}[H]{0.62\textwidth}
        \centering
        \includegraphics[width=\textwidth]{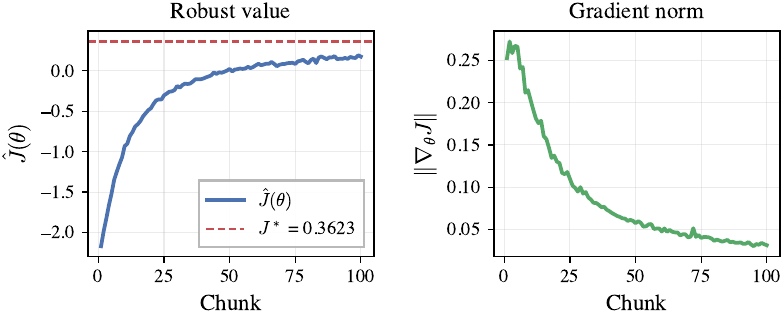}
        {\small (a) $\varepsilon = 0.5$}
    \end{minipage}

    \vspace{0.5em}

    \begin{minipage}[H]{0.62\textwidth}
        \centering
        \includegraphics[width=\textwidth]{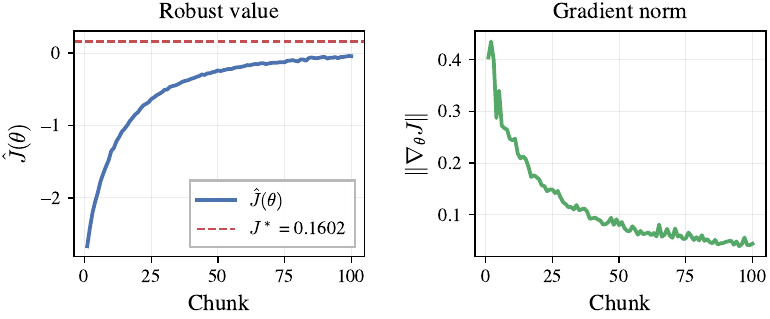}
        {\small (b) $\varepsilon = 1$}
    \end{minipage}

    \vspace{0.5em}

    \begin{minipage}[H]{0.62\textwidth}
        \centering
        \includegraphics[width=\textwidth]{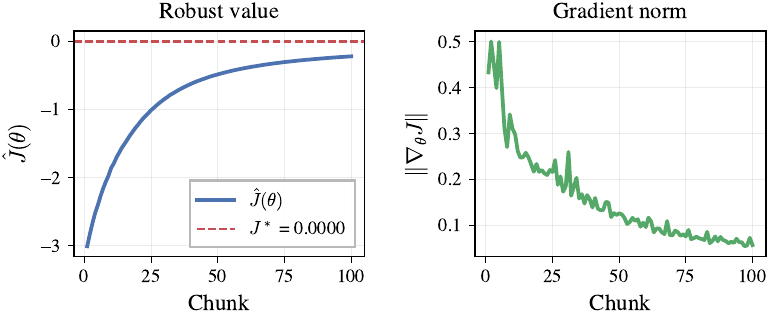}
        {\small (c) $\varepsilon = 2$}
    \end{minipage}

    \caption{Learning curves of the coin toss example over training for different values of $\varepsilon$.}
    \label{fig:coin_toss_eps}
\end{figure}

\begin{figure}[H]
    \centering
        \centering
        \includegraphics[width=0.65\textwidth]{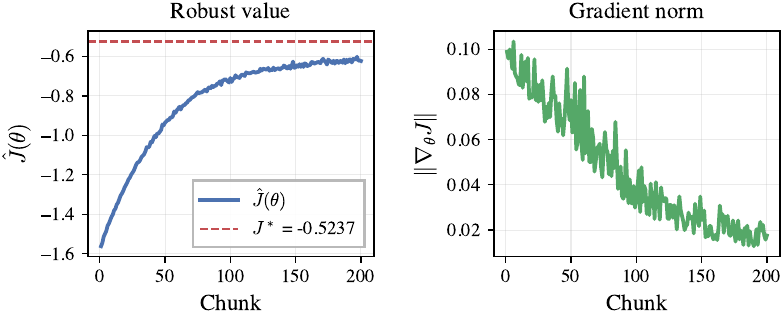}
\caption{Learning curves of the robust self-exciting Multi-armed bandits example over training for $\varepsilon = 0.3$.}
    \label{fig:bandits_eps}
\end{figure}

\begin{figure}[H]
    \centering
        \centering
        \includegraphics[width=0.65\textwidth]{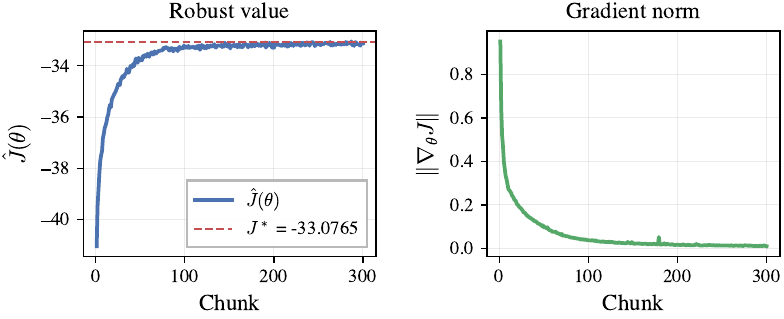}
\caption{Learning curves of the supply chain example over training for $\varepsilon = 1$.}
    \label{fig:supply_chain_eps}
\end{figure}
\begin{figure}[H]
    \centering
        \centering
        \includegraphics[width=0.65\textwidth]{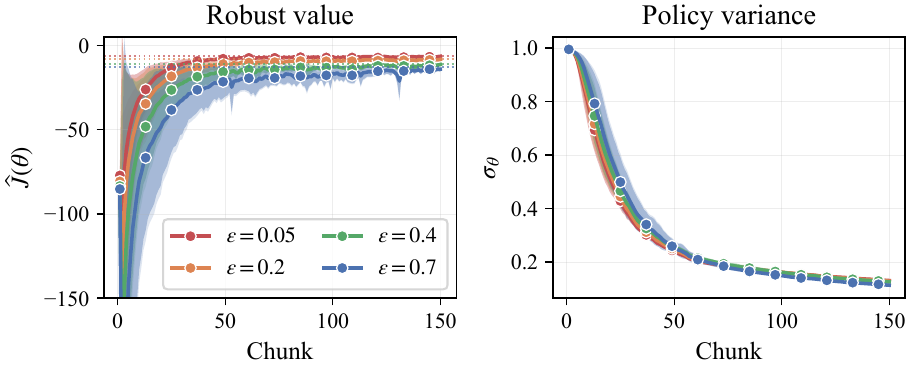}
\caption{Learning curves of the robust linear-quadratic control example over training for $\varepsilon = 0.3$.}
    \label{fig:LQ_eps}
\end{figure}

\begin{small}
\bibliography{References.bib}   

@inproceedings{wang_policy_2022,
  author    = {Wang, Y. and Zou, S.},
  title     = {Policy Gradient Method for Robust Reinforcement Learning},
  booktitle = {Proceedings of the 39th International Conference on Machine Learning},
  series    = {Proceedings of Machine Learning Research},
  volume    = {162},
  pages     = {23484--23526},
  publisher = {PMLR},
  year      = {2022}
}

@article{neufeld_markov_2023,
  author  = {Neufeld, A. and Sester, J. and {\v{S}}iki{\'c}, M.},
  title   = {Markov Decision Processes under Model Uncertainty},
  journal = {Mathematical Finance},
  volume  = {33},
  number  = {3},
  pages   = {618--665},
  year    = {2023}
}

@article{kuratowski_general_1965,
  author  = {Kuratowski, K. and Ryll-Nardzewski, C.},
  title   = {A General Theorem on Selectors},
  journal = {Bulletin de l'Acad{\'e}mie Polonaise des Sciences, S{\'e}rie des Sciences Math{\'e}matiques, Astronomiques et Physiques},
  volume  = {13},
  pages   = {397--403},
  year    = {1965}
}

@misc{bayraktar2025dno,
  author        = {Bayraktar, E. and Feng, Q. and Zhang, Z. and Zhang, Z.},
  title         = {Deep Neural Operator Learning for Probabilistic Models},
  year          = {2025},
  eprint        = {2511.07235},
  archiveprefix = {arXiv}
}

@book{Bertsekas1999,
  author    = {Bertsekas, D. P.},
  title     = {Nonlinear Programming},
  edition   = {2},
  publisher = {Athena Scientific},
  address   = {Belmont, MA},
  year      = {1999}
}

@article{WiesemannKuhnSim2014,
  author  = {Wiesemann, W. and Kuhn, D. and Sim, M.},
  title   = {Distributionally Robust Convex Optimization},
  journal = {Operations Research},
  volume  = {62},
  number  = {6},
  pages   = {1358--1376},
  year    = {2014}
}

@article{EsfahaniKuhn2018,
  author  = {Mohajerin Esfahani, P. and Kuhn, D.},
  title   = {Data-Driven Distributionally Robust Optimization Using the {W}asserstein Metric: Performance Guarantees and Tractable Reformulations},
  journal = {Mathematical Programming},
  volume  = {171},
  number  = {1--2},
  pages   = {115--166},
  year    = {2018}
}

@article{GaoKleywegt2023,
  author  = {Gao, R. and Kleywegt, A. J.},
  title   = {Distributionally Robust Stochastic Optimization with {W}asserstein Distance},
  journal = {Mathematics of Operations Research},
  volume  = {48},
  number  = {2},
  pages   = {603--655},
  year    = {2023}
}

@article{BlanchetKangMurthy2019,
  author  = {Blanchet, J. and Kang, Y. and Murthy, K.},
  title   = {Robust {W}asserstein Profile Inference and Applications to Machine Learning},
  journal = {Journal of Applied Probability},
  volume  = {56},
  number  = {3},
  pages   = {830--857},
  year    = {2019}
}

@article{BartlDrapeauOblojWiesel2021,
  author  = {Bartl, D. and Drapeau, S. and Ob{\l}{\'o}j, J. and Wiesel, J.},
  title   = {Sensitivity Analysis of {W}asserstein Distributionally Robust Optimization Problems},
  journal = {Proceedings of the Royal Society A},
  volume  = {477},
  number  = {2256},
  pages   = {20210176},
  year    = {2021},
  doi     = {10.1098/rspa.2021.0176}
}

@article{CarteaBhudisaksangSanchezBetancourt2025,
  author  = {Bhudisaksang, T. and Cartea, {\'A}. and S{\'a}nchez-Betancourt, L.},
  title   = {Adaptive-Robust Portfolio Optimisation},
  journal = {Mathematics and Financial Economics},
  volume  = {20},
  pages   = {171--202},
  year    = {2026},
  doi     = {10.1007/s11579-025-00411-4}
}

@article{BertsimasThiele2006,
  author  = {Bertsimas, D. and Thiele, A.},
  title   = {A Robust Optimization Approach to Inventory Theory},
  journal = {Operations Research},
  volume  = {54},
  number  = {1},
  pages   = {150--168},
  year    = {2006}
}

@book{AliprantisBorder2006,
  author    = {Aliprantis, C. D. and Border, K. C.},
  title     = {Infinite Dimensional Analysis: A Hitchhiker's Guide},
  edition   = {3},
  publisher = {Springer},
  address   = {Berlin--Heidelberg},
  year      = {2006}
}

@article{BlanchetChenZhou2022,
  author  = {Blanchet, J. and Chen, L. and Zhou, X. Y.},
  title   = {Distributionally Robust Mean-Variance Portfolio Selection with {W}asserstein Distances},
  journal = {Management Science},
  volume  = {68},
  number  = {9},
  pages   = {6382--6410},
  year    = {2022}
}

@misc{SaulduboisTouzi2024,
  author        = {Sauldubois, N. and Touzi, N.},
  title         = {First-Order Martingale Model Risk and Semi-Static Hedging},
  year          = {2024},
  eprint        = {2410.06906},
  archiveprefix = {arXiv}
}

@misc{CompointSaulduboisTouzi2025,
  author        = {Compoint, A. and Sauldubois, N. and Touzi, N.},
  title         = {Sensitivity Analysis of Distributionally Robust BSDEs and RBSDEs},
  year          = {2025},
  eprint        = {2511.01828},
  archiveprefix = {arXiv}
}

@article{KimChung2024,
  author  = {Kim, Y. G. and Chung, B. D.},
  title   = {Data-Driven {W}asserstein Distributionally Robust Dual-Sourcing Inventory Model under Uncertain Demand},
  journal = {Omega},
  volume  = {127},
  pages   = {103112},
  year    = {2024}
}

@article{NilimElGhaoui2005,
  author  = {Nilim, A. and El Ghaoui, L.},
  title   = {Robust Control of {M}arkov Decision Processes with Uncertain Transition Matrices},
  journal = {Operations Research},
  volume  = {53},
  number  = {5},
  pages   = {780--798},
  year    = {2005}
}

@article{Iyengar2005,
  author  = {Iyengar, G. N.},
  title   = {Robust Dynamic Programming},
  journal = {Mathematics of Operations Research},
  volume  = {30},
  number  = {2},
  pages   = {257--280},
  year    = {2005}
}

@article{yang2021,
  author  = {Yang, I.},
  title   = {{W}asserstein Distributionally Robust Stochastic Control: A Data-Driven Approach},
  journal = {IEEE Transactions on Automatic Control},
  volume  = {66},
  number  = {8},
  pages   = {3863--3870},
  year    = {2021}
}

@article{KimYang2021,
  author  = {Kim, K. and Yang, I.},
  title   = {Distributional Robustness in Minimax Linear Quadratic Control with {W}asserstein Distance},
  journal = {SIAM Journal on Control and Optimization},
  volume  = {61},
  number  = {2},
  pages   = {458--483},
  year    = {2023}
}

@article{coache2024robust,
  author  = {Coache, A. and Jaimungal, S.},
  title   = {Robust Reinforcement Learning with Dynamic Distortion Risk Measures},
  journal = {SIAM Journal on Mathematics of Data Science},
  volume  = {8},
  number  = {1},
  pages   = {1--22},
  year    = {2026},
  doi     = {10.1137/24M1699802}
}

@incollection{Scarf1958,
  author    = {Scarf, H. E.},
  title     = {A Min--Max Solution of an Inventory Problem},
  booktitle = {Studies in the Mathematical Theory of Inventory and Production},
  pages     = {201--209},
  publisher = {Stanford University Press},
  year      = {1958}
}

@article{GallegoMoon1993,
  author  = {Gallego, G. and Moon, I.},
  title   = {The Distribution-Free Newsboy Problem: Review and Extensions},
  journal = {Journal of the Operational Research Society},
  volume  = {44},
  number  = {8},
  pages   = {825--834},
  year    = {1993},
  doi     = {10.1057/jors.1993.141}
}

@inproceedings{LimXuMannor2013,
  author    = {Lim, S. H. and Xu, H. and Mannor, S.},
  title     = {Reinforcement Learning in Robust {M}arkov Decision Processes},
  booktitle = {Advances in Neural Information Processing Systems},
  volume    = {26},
  pages     = {701--709},
  year      = {2013}
}

@inproceedings{WangHoPetrik2023,
  author    = {Wang, Q. and Ho, C. P. and Petrik, M.},
  title     = {Policy Gradient in Robust {MDP}s with Global Convergence Guarantee},
  booktitle = {Proceedings of the 40th International Conference on Machine Learning},
  series    = {Proceedings of Machine Learning Research},
  volume    = {202},
  pages     = {35763--35797},
  publisher = {PMLR},
  year      = {2023}
}

@article{NeufeldSester2024,
  author  = {Neufeld, A. and Sester, J.},
  title   = {Robust {Q}-Learning Algorithm for {M}arkov Decision Processes under {W}asserstein Uncertainty},
  journal = {Automatica},
  volume  = {168},
  pages   = {111825},
  year    = {2024}
}

@article{BlanchetMurthy2019,
  author  = {Blanchet, J. and Murthy, K.},
  title   = {Quantifying Distributional Model Risk via Optimal Transport},
  journal = {Mathematics of Operations Research},
  volume  = {44},
  number  = {2},
  pages   = {565--600},
  year    = {2019}
}

@book{billingsley_convergence_1999,
  author    = {Billingsley, P.},
  title     = {Convergence of Probability Measures},
  edition   = {2},
  series    = {Wiley Series in Probability and Statistics},
  publisher = {Wiley},
  year      = {1999}
}

@article{Bonnans1998,
  author  = {Bonnans, J.-F. and Shapiro, A.},
  title   = {Optimization Problems with Perturbations: A Guided Tour},
  journal = {SIAM Review},
  volume  = {40},
  number  = {2},
  pages   = {228--264},
  year    = {1998},
  doi     = {10.1137/S0036144596302644}
}

@inproceedings{KumarDermanGeistLevyMannor2023,
  author    = {Kumar, N. and Derman, E. and Geist, M. and Levy, K. Y. and Mannor, S.},
  title     = {Policy Gradient for Rectangular Robust {M}arkov Decision Processes},
  booktitle = {Advances in Neural Information Processing Systems},
  volume    = {36},
  pages     = {59477--59501},
  year      = {2023}
}

@inproceedings{LiLanMurthySrikant2024,
  author    = {Sun, Z. and He, S. and Miao, F. and Zou, S.},
  title     = {Policy Optimization for Robust Average-Reward {MDP}s},
  booktitle = {Advances in Neural Information Processing Systems},
  volume    = {37},
  pages     = {17348--17372},
  year      = {2024}
}

@inproceedings{YangGuoXuLiuAnandkumar2023,
  author    = {Yang, Z. and Guo, Y. and Xu, P. and Liu, A. and Anandkumar, A.},
  title     = {Distributionally Robust Policy Gradient for Offline Contextual Bandits},
  booktitle = {Proceedings of The 26th International Conference on Artificial Intelligence and Statistics},
  series    = {Proceedings of Machine Learning Research},
  volume    = {206},
  pages     = {6443--6462},
  publisher = {PMLR},
  year      = {2023}
}

@book{RockafellarWets1998,
  author    = {Rockafellar, R. T. and Wets, R. J.-B.},
  title     = {Variational Analysis},
  series    = {Grundlehren der mathematischen Wissenschaften},
  volume    = {317},
  publisher = {Springer},
  address   = {Berlin},
  year      = {1998},
  doi       = {10.1007/978-3-642-02431-3}
}

@book{Rockafellar1970,
  author    = {Rockafellar, R. T.},
  title     = {Convex Analysis},
  series    = {Princeton Mathematical Series},
  volume    = {28},
  publisher = {Princeton University Press},
  address   = {Princeton, NJ},
  year      = {1970}
}

@misc{AbdullahEtAl2019,
  author        = {Abdullah, M. A. and Ren, H. and Bou Ammar, H. and Milenkovic, V. and Luo, R. and Zhang, M. and Wang, J.},
  title         = {{W}asserstein Robust Reinforcement Learning},
  year          = {2019},
  eprint        = {1907.13196},
  archiveprefix = {arXiv}
}

@inproceedings{GrandClementKroer2021,
  author    = {Grand-Cl{\'e}ment, J. and Kroer, C.},
  title     = {First-Order Methods for {W}asserstein Distributionally Robust {MDP}s},
  booktitle = {Proceedings of the 38th International Conference on Machine Learning (ICML)},
  series    = {Proceedings of Machine Learning Research},
  volume    = {139},
  pages     = {2010--2019},
  publisher = {PMLR},
  year      = {2021}
}

@inproceedings{FastBellmanWDRMDP2023,
  author    = {Yu, Z. and Dai, L. and Xu, S. and Gao, S. and Ho, C. P.},
  title     = {Fast {B}ellman Updates for {W}asserstein Distributionally Robust {MDP}s},
  booktitle = {Advances in Neural Information Processing Systems},
  volume    = {36},
  pages     = {30554--30578},
  year      = {2023}
}

@article{WDRContextualBandits2023,
  author  = {Shen, Y. and Xu, P. and Zavlanos, M. M.},
  title   = {{W}asserstein Distributionally Robust Policy Evaluation and Learning for Contextual Bandits},
  journal = {Transactions on Machine Learning Research},
  year    = {2024}
}

@article{YueKuhnWiesemann2022,
  author  = {Yue, M.-C. and Kuhn, D. and Wiesemann, W.},
  title   = {On Linear Optimization over {W}asserstein Balls},
  journal = {Mathematical Programming},
  volume  = {195},
  pages   = {1107--1122},
  year    = {2022}
}

@article{wiesemannetal2013,
  author  = {Wiesemann, W. and Kuhn, D. and Rustem, B.},
  title   = {Robust {M}arkov Decision Processes},
  journal = {Mathematics of Operations Research},
  volume  = {38},
  number  = {1},
  pages   = {153--183},
  year    = {2013}
}

@article{xumannor2012,
  author  = {Xu, H. and Mannor, S.},
  title   = {Distributionally Robust {M}arkov Decision Processes},
  journal = {Mathematics of Operations Research},
  volume  = {37},
  number  = {2},
  pages   = {288--300},
  year    = {2012}
}

@inproceedings{sietal2020,
  author    = {Si, N. and Zhang, F. and Zhou, Z. and Blanchet, J.},
  title     = {Distributionally Robust Policy Evaluation and Learning in Offline Contextual Bandits},
  booktitle = {Proceedings of the 37th International Conference on Machine Learning},
  series    = {Proceedings of Machine Learning Research},
  volume    = {119},
  pages     = {8884--8894},
  publisher = {PMLR},
  year      = {2020}
}

@article{jaimungaletal2022,
  author  = {Jaimungal, S. and Pesenti, S. M. and Wang, Y. S. and Tatsat, H.},
  title   = {Robust Risk-Aware Reinforcement Learning},
  journal = {SIAM Journal on Financial Mathematics},
  volume  = {13},
  number  = {1},
  pages   = {213--226},
  year    = {2022},
  doi     = {10.1137/21M144640X}
}

@book{Polak1997,
  author    = {Polak, E.},
  title     = {Optimization: Algorithms and Consistent Approximations},
  series    = {Applied Mathematical Sciences},
  volume    = {124},
  publisher = {Springer-Verlag},
  address   = {New York},
  year      = {1997}
}
\end{small}

\end{document}